\documentclass[11pt]{article}
\usepackage[ansinew]{inputenc}
\usepackage{graphicx}
\usepackage{epstopdf}
\usepackage{color}

\usepackage{enumerate,latexsym}
\usepackage{latexsym}
\usepackage{amsmath,amssymb}
\usepackage{graphicx}
\usepackage{times}

\newfont{\bb}{msbm10 at 11pt}
\newfont{\bbsmall}{msbm8 at 8pt}
\def\rth{\mathbb{R}^3}
\def\R{\mathbb{R}}
\def\B{\mathbb{B}}
\def\N{\mathbb{N}}
\def\Z{\mathbb{Z}}
\def\C{\mathbb{C}}

\def\D{\mathbb{D}}

\newcommand{\re}{\mbox{\bb R}}

\newcommand{\ben}{\begin{enumerate}}
\newcommand{\bit}{\begin{itemize}}
\newcommand{\een}{\end{enumerate}}
\newcommand{\eit}{\end{itemize}}
\newcommand{\wh}{\widehat}
\newcommand{\ov}{\overline}
\newcommand{\Int}{\mbox{Int}}
\renewcommand{\S}{\Sigma}

\newcommand{\wt}{\widetilde}

\def\a{{\alpha}}
\def\cL{\mathcal{L}}

\def\g{{\gamma}}
\def\G{{\Gamma}}
\def\L{{\Lambda}}
\def\de{{\delta}}
\def\be{{\beta}}
\def\ve{{\varepsilon}}

\def\centerbmp#1#2#3{\vskip#2\relax\centerline{\hbox to#1{\special
    {bmp:#3 x=#1, y=#2}\hfil}}}

\newtheorem{theorem}{Theorem}[section]
\newtheorem{lemma}[theorem]{Lemma}
\newtheorem{proposition}[theorem]{Proposition}

\newtheorem{remark}[theorem]{Remark}
\newtheorem{corollary}[theorem]{Corollary}
\newtheorem{definition}[theorem]{Definition}

\newtheorem{assertion}[theorem]{Assertion}

\newtheorem{claim}[theorem]{Claim}

\newenvironment{proof}{\smallskip\noindent{\it Proof.}\hskip \labelsep}
{\hfill\penalty10000\raisebox{-.09em}{$\Box$}\par\medskip}

\begin{document}
\begin{title}
{The local picture theorem on the scale of topology}
\end{title}
\vskip .2in

\begin{author}
{William H. Meeks III
   \and Joaqu\'\i n P\' erez
\and Antonio Ros
}
\end{author}
\maketitle
\begin{abstract}
We prove a descriptive
theorem on the extrinsic geometry of an embedded minimal surface of
injectivity radius zero in
a homogeneously regular Riemannian three-manifold, in a certain small intrinsic neighborhood of a
point of {\it almost-minimal injectivity radius.}
This structure theorem includes a limit object which we call a
{\it minimal parking garage structure} on $\rth$, whose theory we also develop.

\vspace{.1cm}


\end{abstract}

\section{Introduction.}
\footnotetext{First author's financial support: This material is based upon
work for the NSF under Award No. DMS-1309236.  Any opinions, findings,
and conclusions or recommendations expressed in this publication are those
of the authors and do not necessarily reflect the views of the NSF.
Second and third author's financial support: Research partially supported by
the MINECO/FEDER grant no. MTM2014-52368-P.}
This paper is devoted to an analysis of the extrinsic geometry of
any embedded minimal surface $M$ in small intrinsic balls in a homogeneously regular Riemannian
three-manifold\footnote{A Riemannian three-manifold $N$ is {\it
homogeneously regular} if there exists an $\ve > 0$ such that
the image by the exponential map of any $\ve$-ball in a tangent space
$T_xN$, $x\in N$, is uniformly close to an $\ve$-ball in $\rth$ in
the $C^2$-norm. In particular, $N$ has positive injectivity radius. Note that
if $N$ is compact, then $N$ is
homogeneously regular.}, such that the injectivity radius function of $M$ is
sufficiently small in terms of the ambient geometry of the balls. We
carry out this analysis by blowing-up such an $M$ at a sequence of
points with {\it almost-minimal injectivity radius} (we will define this
notion precisely in items~{1, 2, 3} of the next theorem), which
produces a new sequence of minimal surfaces, a subsequence of which
has a natural limit object being either a properly embedded minimal
surface in $\rth$, a minimal parking garage structure on $\rth$ (we
will study this notion in Section~\ref{seclt2}) or possibly, a
particular case of a singular minimal lamination of $\R^3$ with
restricted geometry, as stated in item~6 of the next result.

In the sequel, we will denote by $B_M(p,r)$ (resp. $\overline{B}_M(p,r)$)
the open (resp. closed) metric ball centered at a point $p$
in a Riemannian manifold $N$, with radius $r>0$. In the case $M$ is complete,
we will let $I_M\colon M\to (0,\infty ]$ be
the injectivity radius function of $M$, and given a subdomain $\Omega \subset
M$, $I_{\Omega }=(I_M)|_{\Omega }$ will stand for the restriction of $I_M$ to $\Omega $.
The infimum of $I_M$ is called the {\it injectivity radius} of $M$.

\begin{theorem}[Local Picture on the Scale of Topology]
\label{tthm3introd}
There exists a smooth decreasing function $\de\colon (0,\infty) \to (0,1/2)$
with $\lim_{r\to \infty} r\, \de(r)=\infty$
such that the following statements hold.
Suppose $M$ is a complete, embedded minimal
surface with injectivity radius zero in a homogeneously
regular
 three-manifold~$N$. Then, there exists a sequence
of points $p_n\in M$ (called ``points of almost-minimal injectivity radius'')
and positive numbers $\ve _n= n\, I_{M}(p_n)\to 0$ such that:
\begin{description}
\item[{\it 1}]  For all $n$, the closure $M_n$ of the component
of $M\cap B_N(p_n,\ve _n)$  that contains $p_n$ is  a compact surface with boundary in
$\partial B_N(p_n,\ve _n)$. Furthermore, $M_n$ is contained in the intrinsic open ball
$B_M(p_n,\frac{r_n}{2}I_M(p_n))$, where
$r_n>0$ satisfies $r_n\, \de (r_n)=n$.

\item[{\it 2}] Let $\lambda _n=1/I_{M}(p_n)$.
Then, $\lambda_nI_{M_n}\geq 1-\frac{1}{n}$ on $M_n$.

\item[{\it 3}] The metric balls $\lambda_nB_N(p_n,\ve_n)$
of radius $n=\lambda_n\ve_n$
converge uniformly as $n\to \infty $ to $\rth$ with
its usual metric (so that we identify $p_n$ with
$\vec{0}$ for all $n$).
\end{description}
Furthermore, exactly one of the following three possibilities occurs.
\begin{description}
\item[{\it 4}] The surfaces $\lambda_nM_n$ have
uniformly bounded Gaussian curvature
on compact subsets\footnote{As $M_n\subset B_N(p_n,\ve _n)$, the convergence
$\{ \lambda_nB_N(p_n,\ve _n)\} _n\to \R^3$ explained in item~{3} allows us to view
the rescaled surface $\lambda_nM_n$ as a subset of $\R^3$. The uniformly bounded
property for the Gaussian curvature of the induced metric on $M_n\subset N$ rescaled by
$\lambda_n$ on compact subsets of $\R^3$ now makes sense.} of
$\rth$ and there exists a connected, properly
embedded minimal surface $M_{\infty}\subset \R^3$
with $\vec{0}\in M_{\infty }$, $I_{M_{\infty}}\geq 1$
and $I_{M_{\infty}}(\vec{0})=1$,
 such that for any
$k \in \N$, the surfaces $\lambda_nM_n$ converge $C^k$
on compact subsets of $\rth$  to $M_{\infty}$ with
multiplicity one as $n \to \infty$.
\item[{\it 5}] After a rotation in $\rth$, the surfaces $\lambda_nM_n$ converge to
a minimal parking garage structure\footnote{For a description
of a minimal parking garage structure, see Section~\ref{seclt2}.} on $\rth$,
consisting of a foliation $\mathcal{L}$ of $\R^3$ by
horizontal planes, with columns forming a locally finite
 set $S(\mathcal{L})$ of vertical straight lines (at least two lines).
Moreover, if there exists a bound on the genus of the surfaces $\lambda_nM_n$, then
$S(\mathcal{L})$ consists  of just two lines $l_1, \,l_2$,
the associated limiting pair of multivalued graphs\footnote{This
means that for $i=1,2$
and $k\in \N$ fixed ($k\geq 2$), and for $n$ large
enough depending on $k$, $\lambda_nM_n$ contains around $l_i$ a pair of $k$-valued
graphs (See Definition~\ref{kgraph} for this concept)
with opposite orientations, both spiraling together, and the handedness of these
$k$-graphs nearby $l_1$ is opposite to the related one around $l_2$.}
in $\lambda_nM_n$ nearby $l_1,l_2$ are oppositely handed
and given $R>0$, for $n\in \N$ large depending on $R$, the surface
$(\lambda_nM_n)\cap B_{\lambda_nN}(p,\frac{R}{\lambda_n})$ has genus zero.
\item[{\it 6}]
\label{i6}
There exists a nonempty, closed set $\mathcal{S}\subset
\R^3$, a minimal lamination $\mathcal{L}$ of \mbox{$\R^3-\mathcal{S}$} and a
subset $S(\mathcal{L})\subset \mathcal{L}$ which is closed in the subspace topology,
such that the surfaces $(\lambda_nM_n)-\mathcal{S}$ converge to $\mathcal{L}$
outside of $S(\mathcal{L})$ and $\mathcal{L}$ has at least one nonflat leaf. Furthermore,
if we let $\Delta (\mathcal{L})=\mathcal{S} \cup S(\mathcal{L})$ and let
$\mathcal{P}$ be the sublamination of flat leaves in $\mathcal{L}$, then the following holds.
$\mathcal{P}\neq \mbox{\rm \O}$, the closure of every such flat leaf is a horizontal plane, and
if $L\in \mathcal{P}$ then the plane
$\overline{L}$ intersects $\Delta (\mathcal{L})$ in a set containing at least two points,
each of which are at least distance 1 from each other in $\overline{L}$, and either
$\overline{L}\cap \Delta (\mathcal{L})\subset \mathcal{S}$
or $\overline{L}\cap \Delta (\mathcal{L})\subset S(\mathcal{L})$.
\end{description}
\end{theorem}

For a more detailed description of cases 5 and 6 of Theorem~\ref{tthm3introd},
see Propositions~\ref{ass4.17} and~\ref{proposnew} below.

The results in the series of papers
\cite{cm21,cm22,cm24,cm23,cm35,cm25} by Colding and Minicozzi
and the minimal lamination closure theorem by Meeks and
Rosenberg~\cite{mr13}  play important roles in
deriving the above compactness result. We conjecture that
item~{6} in Theorem~\ref{tthm3introd} does not actually occur.

A short explanation of the organization of the paper is as follows.
In Section~\ref{sec2} we introduce some notation and
recall the notion and language of laminations, as well as a chord-arc
property for embedded minimal disks previously proven by Meeks and Rosenberg
and based on a similar one by Colding and Minicozzi. In Section~\ref{seclt2}
we develop the theory of parking garage surfaces and limit
parking garage structures, a notion that appears in item~5 of the
main Theorem~\ref{tthm3introd}. Section~\ref{sec4}, the bulk of this paper,
is devoted to the proof of Theorem~\ref{tthm3introd}. Section~\ref{secap}
includes some applications of Theorem~\ref{tthm3introd}.
We refer the reader to~\cite{mpe3,mpe17,mpr8,mpr9,mpr11,mr13,mt13,mt9} for further
applications of Theorem~\ref{tthm3introd}.

\section{Preliminaries.}
\label{sec2}
Let $M$ be a Riemannian manifold. Let
$B_M(p,r)$ be the open ball centered at a point $p\in M$ with radius
$r>0$, for the underlying metric space structure of $M$ associated to its
Riemannian metric. When $M$ is complete, the injectivity radius $I_M(p)$ at a point $p\in M$
 is the supremum of the radii $r>0$ of the open balls $B_M(p,r)$ for which
the exponential map at $p$ is a diffeomorphism. This defines
the {\it injectivity radius function,} $I_M\colon M\to (0,\infty ]$, which is continuous
on $M$ (see e.g., Proposition~88 in Berger~\cite{ber1}). The infimum of
$I_M$ is called the {\it injectivity radius} of $M$.

\begin{definition}
\label{deflamination}
{\rm
 A {\it codimension-one
lamination} of a Riemannian three-manifold $N$ is the union of a
collection of pairwise disjoint, connected, injectively immersed
surfaces, with a certain local product structure. More precisely, it
is a pair $(\mathcal{L},\mathcal{A})$ satisfying:
\begin{enumerate}[1.]
\item $\mathcal{L}$ is a closed subset of $N$;
\item $\mathcal{A}=\{ \varphi _{\be }\colon \D \times (0,1)\to
U_{\be }\} _{\be }$ is an atlas of coordinate charts of $N$ (here
$\D $ is the open unit disk in $\R^2$, $(0,1)$ is the open unit
interval and $U_{\be }$ is an open subset of $N$); note that
although $N$ is assumed to be smooth, we
only require that the regularity of the atlas (i.e., that of
its change of coordinates) is of class $C^0$, i.e., $\mathcal{A}$
is an atlas for the topological structure of $N$.
\item For each $\be $, there exists a closed subset $C_{\be }$ of
$(0,1)$ such that $\varphi _{\be }^{-1}(U_{\be }\cap \mathcal{L})=\D \times
C_{\be}$.
\end{enumerate}
}
\end{definition}

We will simply denote laminations by $\mathcal{L}$, omitting the
charts $\varphi _{\be }$ in $\mathcal{A}$.
A lamination $\mathcal{L}$ is said to be a {\it foliation of $N$} if $\mathcal{L}=N$.
Every lamination $\mathcal{L}$ naturally decomposes into a
collection of disjoint, connected topological surfaces (locally given by $\varphi
_{\be }(\D \times \{ t\} )$, $t\in C_{\be }$, with the notation
above), called the {\it leaves} of $\mathcal{L}$. As usual, the
regularity of $\mathcal{L}$ requires the corresponding
regularity on the change of coordinate charts $\varphi _{\be }$.
A lamination $\mathcal{L}$ of $N$ is said to be a {\it minimal lamination}
if all its leaves are (smooth) minimal surfaces. Since the leaves of $\mathcal{L}$ are pairwise
disjoint, we can consider the norm of the second fundamental form
$|\sigma _\mathcal{L}|$ of $\mathcal{L}$, which is the function defined at every point $p$ in
$\mathcal{L}$ as $|\sigma _L|(p)$, where $L$ is the
unique leaf of $\mathcal{L}$ passing through $p$ and $|\sigma _L|$ is the norm of the
second fundamental form of $L$.

\begin{definition}
\label{def2.2}
{\rm
If $\{ \S_n\} _n$ is a sequence of complete embedded minimal surfaces in a Riemannian three-manifold
$N$, consider the closed set
$A\subset N$ of points $p\in N$ such that for every neighborhood $U_p$ of $p$ and every subsequence
of $\{ \S_{n(k)}\} _k$, the sequence of norms of the second fundamental forms of $\S_{n(k)}\cap U_p$
is not uniformly bounded.
By the arguments in Lemma~1.1 of Meeks and Rosenberg~\cite{mr8} (see also
 Proposition~B.1 in~\cite{cm23}),
after extracting a subsequence, the $\S_n$ converge
on compact subsets of $N-A$ to a minimal lamination $\mathcal{L}'$ of $N-A$ that extends to a minimal lamination
$\mathcal{L}$ of $N-\mathcal{S}$,
where $\mathcal{S}\subset A$ is the (possibly empty) {\it singular set} of $\mathcal{L}$, i.e., $\mathcal{S}$ is
the closed subset of $N$
such that $\mathcal{L}$ does not admit a local lamination structure around any point of $\mathcal{S}$.
We will denote by $S(\mathcal{L})=A-\mathcal{S}$ the {\it singular set of convergence} of the $\S_n$
to $\mathcal{L}$, i.e., those points of $N$ around which $\mathcal{L}$ admits a lamination structure
but where the second fundamental forms of the $\S_n$ still blow-up.
}
\end{definition}

In this paper we will apply the {\it Minimal Lamination Closure Theorem} in~\cite{mr13},
which insures that if  $M$ is a complete, embedded minimal surface of positive injectivity radius
in a Riemannian three-manifold $N$ (not necessarily complete), then the closure $\overline{M}$ of $M$ in
$N$ has the structure of a $C^{0,1}$-minimal lamination $\mathcal{L}$ with the components of $M$ being
leaves of $\mathcal{L}$. We will also use the following technical result from~\cite{mr13}, which generalizes
to the manifold setting some of the results in~\cite{cm35}.
\begin{definition}
\label{def2.3}
{\em
Given a surface $\Sigma $ embedded in a Riemannian three-manifold $N$, a point $p\in \Sigma $ and
$R>0$, we denote by $\Sigma (p,R)$ the closure of the component of $\Sigma \cap B_N(p,R)$
that passes through $p$.
}
\end{definition}

\begin{theorem}[Theorem 13 in~\cite{mr13}]
\label{thm2.2}
Suppose that $\Sigma $ is a compact, embedded minimal disk in a homogeneously regular
three-manifold $N$ whose injectivity radius function
$I_{\Sigma }\colon \Sigma \to [0,\infty )$
equals the distance to the boundary function
$d_{\Sigma }(\cdot ,\partial \Sigma )$
($d_{\Sigma}$ denotes intrinsic distance in $\Sigma$).
Then, there exist numbers $\de '\in (0,1/2)$ and $R_0>0$, both depending only on $N$, such that
if $\overline{B}_{\Sigma }(x,R)\subset \Sigma -\partial \Sigma $ and $R\leq R_0$, then
\[
\Sigma (x,\de' R)\subset \overline{B}_{\Sigma }(x,R/2).
\]
Furthermore, $\Sigma (x,\de' R)$ is a compact,
embedded minimal disk in $\overline{B}_N(x,\de' R)$ with
$\partial \Sigma (x,\de' R)\subset \partial B_N(x,\de' R)$.
\end{theorem}

\section{Parking garage structures in $\R^3$.}
\label{seclt2}
For a Riemannian surface $M$, $K_M$ will stand for  its
Gaussian curvature function.
In our previous paper~\cite{mpr20} we proved the Local Picture
Theorem on the Scale of Curvature, which is a tool that applies to
any complete, embedded minimal surface $M$ of unbounded absolute
Gaussian curvature in a homogeneously
regular three-manifold $N$, and produces via a blowing-up process a
nonflat, properly embedded minimal surface $M_{\infty } \subset
\R^3$ with normalized curvature (in the sense that $|K_{M_{\infty
}}|\leq 1$ on $M_{\infty }$ and $\vec{0}\in M_{\infty }$,
$|K_{M_{\infty }}|(\vec{0})=1$). The key ingredient to do this is to
find points $p_n\in M$ of {\it almost-maximal curvature} and then
rescale exponential coordinates in $N$ around these points $p_n$ by
$\sqrt{|K_M|(p_n)}\to \infty $ as $n\to \infty $. We will devote
the next section to obtain a somehow similar result for an $M$ whose
injectivity radius is zero, by exchanging the role of $\sqrt{|K_M|}$
by $1/I_M$, where $I_M\colon M\to (0,\infty ]$ denotes the injectivity radius
function on $M$. We will consider this rescaling ratio after
evaluation at points $p_n\in M$ of {\it almost-minimal injectivity radius},
in a sense to be made precise in the first paragraph of Section~\ref{sec4}. One of the
difficulties of this generalization is that the limit objects that
we can find after blowing-up might be not only properly embedded
minimal surfaces in $\R^3$, but also new objects, namely limit
minimal parking garage structures which we study below, and certain
kinds of singular minimal laminations of $\R^3$.

Roughly speaking, a
{\it minimal parking garage structure} is
a limit object for a
sequence of embedded minimal surfaces which converges $C^\a$, $\a \in (0,1)$,
to a minimal foliation $\mathcal{L}$ of $\rth$ by parallel
planes, with singular set of convergence $S(\mathcal{L})$ being a
locally finite set of lines orthogonal to
$\mathcal{L}$, called the {\it columns} of the limit parking garage
structure, along which the limiting surfaces have
the local appearance of a highly-sheeted double
staircase. For example, the sequence of homothetic
shrinkings $\frac{1}{n} \, H$ of a vertical helicoid
$H$ converges to a minimal parking garage structure
that consists of the minimal foliation $\mathcal{L}$ of $\rth$
 by horizontal planes with singular set of convergence
$S(\mathcal{L})$ being the $x_3$-axis.

We remark that some of the language associated to
minimal parking garage structures, such as columns,
 appeared first in a paper of Traizet and
Weber~\cite{tw1}, and the first important application of this
type of structure  appeared in~\cite{mpr1} where we applied it to derive
curvature estimates for certain complete embedded minimal
planar domains in $\rth$. In~\cite{tw1}, Traizet and
Weber produced an
analytic method for constructing a one-parameter
family of properly embedded, periodic minimal surfaces in
$\rth$, by analytically untwisting via the implicit
function theorem a limit configuration given by a
finite number of regions on vertical helicoids in
$\rth$ that have been glued together in a consistent
way.  They referred to the limiting configuration as a
parking garage structure on $\rth$ with columns
corresponding to the axes of the helicoids that they
glued together. Most of the area of these surfaces,
just before the limit, consists of very flat horizontal
levels (almost-horizontal densely packed horizontal
planes) joined by the vertical helicoidal columns.
One can travel quickly up and down the horizontal
levels of the limiting surfaces only along the
helicoidal columns in much the same way that some
parking garages are configured for traffic flow;
hence, the name parking garage structure.
Parking garage structures also appear as natural objects
in the main results of the papers~\cite{cm23,cm25,mt14}.

We now describe in more detail the notion of a
{\it parking garage surface.} Consider a possibly infinite,
nonempty, locally finite set of points $P \subset
\re^2$ and a collection $\mathcal{D}$ of open
round disks centered at the points of $P$ such that the
closures of these disks form a pairwise disjoint
collection.  Let $\mu \colon H_1 (\re^2 - \mathcal{D})
\rightarrow \Z$ be a group homomorphism such that
 $\mu$ takes the values $\pm 1$ on the
homology classes represented by the boundary circles
of the disks in $\mathcal{D}$.  Let $\Pi \colon M
\rightarrow \re^2- \mathcal{D}$ be the infinite
cyclic covering space associated to the kernel
of the composition of the
natural map from $\pi_1(\re^2-\mathcal{D})$ to
$H_1(\re^2- \mathcal{D})$ with $\mu$.  It is
straightforward to embed $M$ into $\rth$ so
that under the natural identification of $\re^2$
with $\R^2
\times \{0\}$, the map $ \Pi$ is the restriction to
$M$ of the orthogonal projection of $\rth$ to
$\re^2 \times \{0\}$. Furthermore, in this embedding,
we may assume that the covering transformation of $M$
corresponding to an $n \in \Z$ is given geometrically
by translating $M$ vertically by $(0,0,n)$.  In
particular, $M$ is a singly-periodic surface with boundary
in $\partial \mathcal{D} \times \re$. $M$
has exactly one boundary curve $\G $ on each vertical cylinder
over the boundary circle of each disk in $\mathcal{D}$.
We may assume that every such curve $\G $ is a
helix, see Figure~\ref{parkinggarage6}.
\begin{figure}
\begin{center}
\includegraphics[width=8cm]{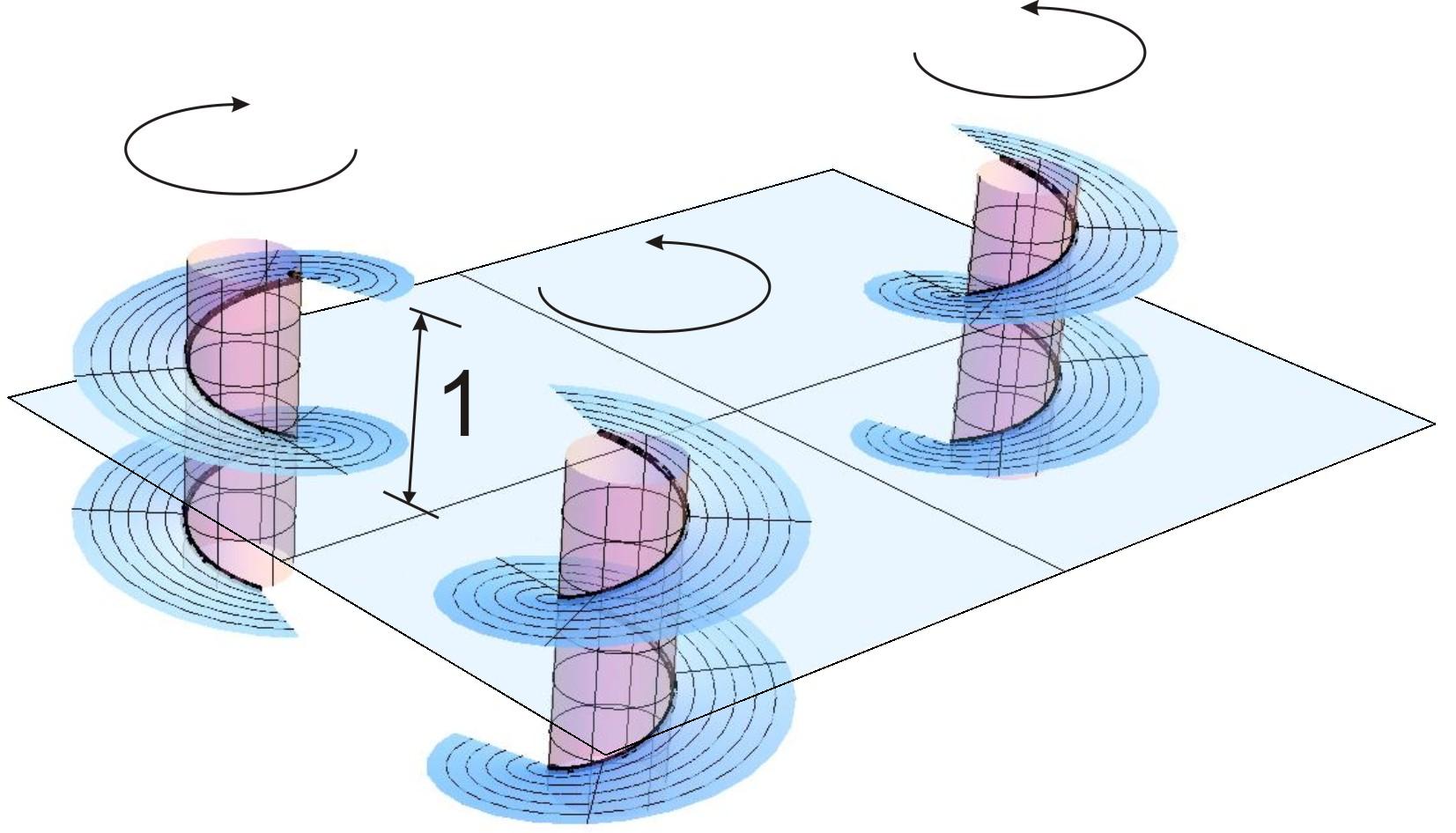}
\caption{Schematic representation of one of the
``two halves'' $M$ of a parking
garage surface $G$ with three columns, two right-handed and one
left-handed. The entire surface $G$, not represented in
the figure, is obtained after gluing $M$ with $M+(0,0,\frac{1}{2})$
and with an infinite helicoidal strip inside each of the columns.}
\label{parkinggarage6}
\end{center}
\end{figure}
Let $M(\frac{1}{2})$ be the vertical translation
of $M$ by $(0, 0,\frac{1}{2})$. $M \cup
M(\frac{1}{2})$ is an embedded, disconnected periodic
surface in $(\R^2-\mathcal{D})\times \R $ with a double
 helix on each boundary cylinder in $\partial \mathcal{D}
 \times \re$.

 \begin{definition}
{\rm
In the above situation, we will call a (periodic) {\it parking garage surface}
 corresponding to the surjective homomorphism $\mu$ to
 the connected topological surface $G\subset \R^3$ obtained after
attaching to $M \cup M(\frac{1}{2})$ an
infinite helicoidal strip in each of the solid
cylinders in $\mathcal{D} \times \re$. Note that by choosing $M$
appropriately, the resulting surface $G$ can be made smooth.
}
\end{definition}

Since in minimal surface theory we only see the
parking garage structure in the limit of a sequence of minimal surfaces, when the
helicoidal strips in the cylinders of $\mathcal{D}\times
 \R $ become arbitrarily densely packed, it is
useful in our construction of $G$ to consider parking
 garages $G(t)$ invariant under translation by
 $(0,0,t)$ with $t \in (0, 1]$ tending to zero. For
$t \in (0, 1]$, consider the affine transformation
$F_t(x_1, x_2, x_3) = (x_1, x_2, t x_3).$  Then
$G(t)=F_t(G)$. Note that our previously defined
surface $G$ is $G(1)$ in this new setup.
As $t \rightarrow 0$, the $G(t)$ converge to the
foliation $\mathcal{L}$ of $\rth$ by horizontal planes
 with singular set of convergence $S(\mathcal{L})$
consisting of the vertical lines in $P \times \re$.
Also, note that $M$ depends on the epimorphism
$\mu$, so to be more specific, we could also denote
$G(t)$ by $G(t, \mu)$.

\begin{definition}
  {\rm In the sequel, we will call the above limit object
$\lim _{t\to 0}G(t, \mu )$ a {\it limit parking garage
  structure} of $\R^3$ associated to the surjective homomorphism
  $\mu $.
}
\end{definition}

Next  we remark on the topology of the ends of the
periodic parking garage surface $G$ in the case that
$P$ is a finite set, where $G=G(t,\mu)$ for some
$t$ and $\mu$. Suppose $\mathcal{D}=\{D_1, \ldots, D_n\}$.
Then we associate to $G$ an integer index:
\[
I(G) = \sum^n_{i=1} \mu([\partial D_i])\in \Z .
\]
Note that the index $I(G)=I(G(t))$ does not depend on the parameter
$t$, and thus, it makes sense to speak about this index for a limit parking
garage structure $\lim _{t\to 0}G(t)$.

Consider the quotient orientable surface
$G/\Z$ in $\rth/\Z$, where $\Z$ is
generated by translation by $(0,0,t)$.  The ends of
$G/\Z$ are annuli and there are exactly two of them.
If $I(G) =0$, then these annular ends of
$G/\Z$ lift to graphical annular ends of $G$.
If $I(G)  \neq 0$, then the universal cover of an end
of $G/\Z$ has $|I(G)|$ orientation preserving
lifts to $G$, each of which gives rise to an
infinite-valued graph over its projection to the end of
$\re^2 \times \{0\}$.
\begin{lemma}
\label{ass2.1}
Suppose that $G$ is a periodic parking garage surface in $\R^3$ with a finite set
$P\times \R $ of $n$ columns. Then, the following properties are
equivalent.
\begin{enumerate}
\item $G$ has genus zero.
\item $G$ has finite genus.
\item $n=1$ {\rm (}in which case $G$ is
simply connected and $I(G) =\pm 1${\rm )},
or $n=2$ and
$I(G) =0$ {\rm (}in which case $G$ has an
infinite number of annular ends with two
limit ends{\rm )}.
\end{enumerate}
\end{lemma}
\begin{proof}
The equivalence between items~{1} and {2} holds
since $G$ is periodic. Clearly item~{3} implies
item~{1}. Finally, if item~{3} does not hold then
either $n\geq 3$ or $G$ has two columns with the
same handedness. In any of these cases there exist at
least two points $x_1,x_2\in P$ with associated values
$\mu ([\partial D_1])= \mu ([\partial D_2])$ for
the corresponding disks $D_1,D_2$ in $\mathcal{D}$ around
$x_1,x_2$ (up to reindexing). Consider an embedded
arc $\g $ in $\R^2-P$ joining $x_1$ to $x_2$. Then one can
lift $\g $ to two arcs in consecutive levels of the
parking garage $G$ joined by short vertical segments
on the columns over $x_1$ and $x_2$. Let
$\widetilde{\g }$ denote this associated simple closed
curve on $G$. Observe that if $\widetilde{\g }'$ is
the related simple closed curve obtained by
translating $\widetilde{\g }$ up exactly one level in
$G$ (this means that $\widetilde{\g}'=\widetilde{\g }+(0,0,t/2)$ if $G=G(t,\mu )$),
then $\widetilde{\g }$ and
$\widetilde{\g }'$ have intersection number one.
Thus, a small regular neighborhood of
$\widetilde{\g }\cup \widetilde{\g }'$ on $G$ has
genus one, which implies that item~{1} does not hold.
\end{proof}
\subsection{Examples of parking garage structures.}

\begin{enumerate}[(B1)]
\item Consider the limit of homothetic
shrinkings of a vertical helicoid. One obtains in
this way the foliation $\mathcal{L}$ of $\rth$ by
horizontal planes with a single column, or
singular curve of convergence $S(\mathcal{L})$, being
the $x_3$-axis. The related limit  {\it minimal}
parking garage surface $G$ has invariant
$I(G)=\pm 1$; the word ``minimal'' is used because
in this case, the surface $G$ is a minimal surface.

\item Let $R_t$, $t>0$, be the classical Riemann minimal examples. These
are properly embedded, singly-periodic minimal surfaces with genus zero and
infinitely many planar ends asymptotic to horizontal planes. Consider
the limit of the $R_t$ when the flux vector of $R_t$ along a compact
horizontal section converges to $(2,0,0)$. Note that the surfaces $R_t$ are invariant under
a translation that only becomes vertical in the limit; in spite of this slight difference with
the theoretical framework explained above, where the surface $G(t,\mu )$ is invariant
under a vertical translation, we still consider the limit minimal parking structure
in this case. This limit minimal parking garage structure $G$ has two columns
with opposite handedness (see~\cite{mpr1} for a proof of these
properties) and so, $G$ has invariant $I(G)=0$, see
Figure~\ref{figurepg}. These examples (B1), (B2)
correspond to Case~{3} of Lemma~\ref{ass2.1}.
\begin{figure}
\begin{center}
\includegraphics[width=9cm]{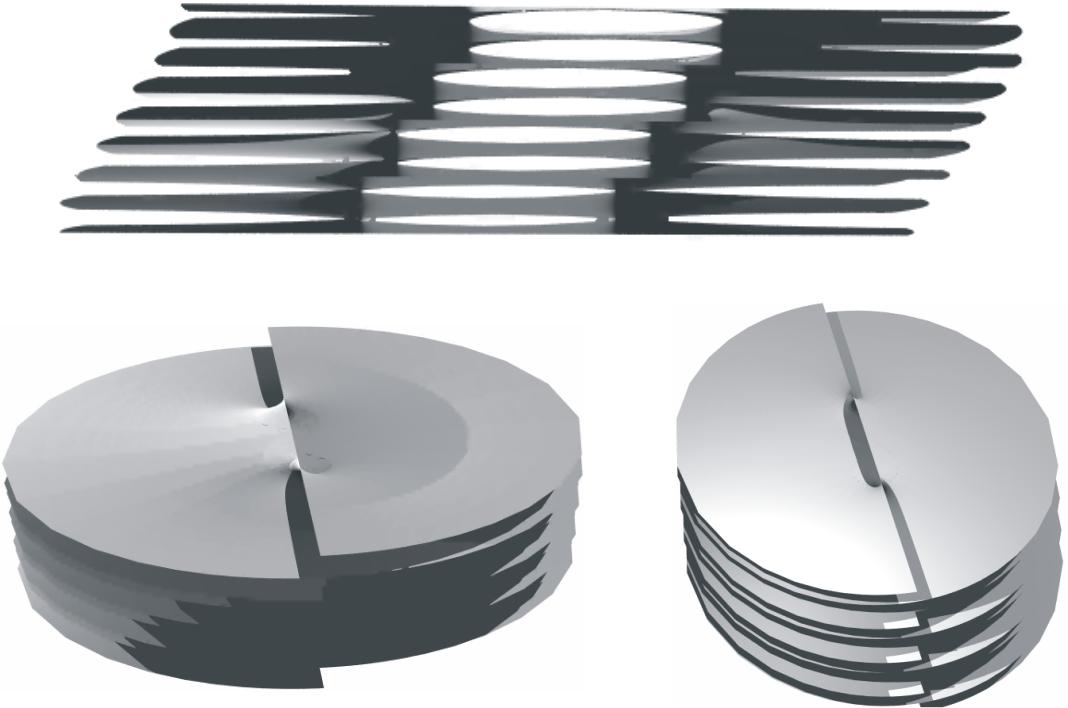}
\caption{Three views of a minimal parking
garage surface, constructed on
a Riemann minimal example.}
\label{figurepg}
\end{center}
\end{figure}

\item Consider the Scherk
doubly-periodic minimal surfaces $S_{\theta}$,
$\theta \in (0,\frac{\pi}{2}]$, with horizontal
lattice of periods $\{((m+n)\cos \theta, (m-n)\sin
\theta, 0) \mid m,n \in \Z \}$. The limit as $\theta
\rightarrow 0$ of the surfaces $S_{\theta }$
is a foliation of $\rth$ by planes
parallel to the $(x_1,x_3)$-plane, with columns
of the same orientation being the horizontal
lines parallel to the $x_2$-axis
and passing through $\Z \times \{0\} \times \{0\}$.
The related minimal parking garage structure
of $\rth$ has an infinite number of columns,
all of which are oriented the same way.
\end{enumerate}

We refer the interested reader to~\cite{tw1} for
 further details and more examples of parking garage
structures that occur in minimal surface theory.

\begin{lemma}
Every parking garage surface in $\R^3$ with a
finite number of columns is recurrent for Brownian motion.
\end{lemma}
\begin{proof}
Note that if $G\subset \R^3$ is a parking garage surface and
we consider the natural action of $\Z $ over $G$ by vertical
orientation preserving translations, then the quotient surface $G/\Z
$ has finite topology, exactly two annular ends and quadratic area
growth. In particular, $G/\Z $ is conformally a twice punctured
compact Riemann surface. On the other hand, since the covering $\Pi
\colon G\to G/\Z $ is a normal covering, then there is a natural
homomorphism $\tau $ from the fundamental group of $G/\Z $ onto the
group of automorphisms Aut$(\Pi )$ of this covering. Since Aut$(\Pi
)$ is abelian, then $\tau $ factorizes through the first homology group
$H_1(G/\Z )$ to a surjective
homomorphism $\widetilde{\tau }\colon H_1(G/\Z )\to \mbox{Aut}(\Pi )$.
Furthermore, each one of the two homology classes in $H_1(G/\Z )$
given by loops on $G/\Z $ around the ends applies via
$\widetilde{\tau }$ on a generator of Aut$(\Pi )$. In this setting,
Theorem~2 in Epstein~\cite{ep2} implies that $G$ is recurrent; see
the first paragraph in Section~4 of~\cite{mpr13} for details on this application
of the result by Epstein.
\end{proof}

\begin{remark}
{\rm
In Example (B3) above, the surface $S_{\theta }$ for any $\theta $
is not recurrent for Brownian motion, but it
is close to that condition, in the sense that
it does not admit positive nonconstant
harmonic functions, see~\cite{mpr13}.
}
\end{remark}

We have already introduced the notation $B_N(p,r)$
for the open metric ball centered at the
point $p$ with radius $r>0$ in a Riemannian three-manifold $N$.
In the case $N=\R^3$, we will simplify $\B (p,r)=B_{\R^3}(p,r)$
and $\B (r)=\B (\vec{0},r)$.

Note that it also makes sense for a sequence
of compact, embedded minimal surfaces
$M_n\subset \B(R_n)$ with boundaries
$\partial M_n\subset \partial \B(R_n)$ such that
$R_n\to \infty $ as $n\to \infty $, to converge
on compact subsets of $\rth$ to a minimal parking garage
structure on $\rth$
consisting of a foliation $\mathcal{L}$ of $\rth$ by planes
 with a locally finite set of lines $S(\mathcal{L})$
  orthogonal to the planes in $\mathcal{L}$, where $S(\mathcal{L})$
corresponds to the singular set of convergence of the
$M_n$ to $\mathcal{L}$.  We note that each of the lines
in $S(\mathcal{L})$ has an associated $+$ or $-$ sign
corresponding to whether or not the associated
forming helicoid in $M_n$ along the line is right or left handed.
For instance, Theorem~0.9 in~\cite{cm25} illustrates a particular
case of this convergence to a limit parking garage structure on $\R^3$
when $S(\mathcal{L})$ consists of two lines with associated
double staircases of opposite handedness.

\begin{remark}
{\rm
To study other aspects of how minimal parking garage structures
appear as the limit of a sequence of minimal surfaces in $\rth$, see
Meeks~\cite{me25,me30}.
}
\end{remark}

\section{The Proof of Theorem~\ref{tthm3introd}.}
\label{sec4}
Let $M\subset N$ be a complete, embedded minimal surface with
injectivity radius zero in a homogeneously regular three-manifold
$N$. As $N$ is homogeneously regular, its injectivity radius is positive.
After a fixed constant scaling of the metric of $N$, we may
assume that the injectivity radius of $N$ is greater than 1.
The first step in the proof of Theorem~\ref{tthm3introd} is to obtain
special points $p'_n\in M$, called {\it points of almost-minimal injectivity radius.}
To do this, first consider a sequence of
points $q_n\in M$ such that $I_M(q_n)\leq \frac{1}{n}$
(such a sequence $\{ q_n\} _n$ exists since the injectivity radius of $M$ is
zero). Consider the continuous function $h_n\colon \ov{B}_{M}(q_n,1)\to \R $ given by
\begin{equation}
\label{eq:hn}
h_n(x)=\frac{d_M(x,\partial B_M(q_n,1))}{I_M(x)},\quad x\in
\ov{B}_M(q_n,1),
\end{equation}
where $d_M$ is the distance function associated to its Riemannian metric.
As $h_n$ is continuous and vanishes on $\partial B_M(q_n,1)$, then there
exists $p'_n\in B_{M}(q_n,1)$ where $h_n$ achieves its maximum value.

We define $\lambda'_n=I_M(p'_n)^{-1}$. Note that
\begin{equation}
\label{eq:hn2}
\lambda'_n\geq \lambda'_nd_M(p'_n,\partial B_M(q_n,1)) =h_n(p'_n)\geq
h_n(q_n)=I_M(q_n)^{-1}\geq n.
\end{equation}
Fix $t>0$. Consider exponential coordinates centered at $p'_n$ in the extrinsic
ball $B_N(p_n',\frac{t}{\lambda_n'})$ (this can be done if $n$ is sufficiently large).
After rescaling the ambient metric by the factor $\lambda_n'\to \infty $ and identifying
$p'_n$ with the origin $\vec{0}$, we conclude that the
sequence $\{ \lambda'_nB_N(p'_n,\frac{t}{\lambda'_n}) \} _n$ converges to
the open ball $\B (t)$ of $\R^3$ with its usual metric. Similarly, we can consider $\{ 
\lambda'_n\overline{B}_M(p'_n,\frac{t}{\lambda'_n})\} _n$ to be a sequence of embedded
minimal surfaces with boundary, all passing through $p'_n= \vec{0}$ with
injectivity radius $1$ at this point.
\begin{lemma}
The injectivity
radius function of $\lambda_n'M$ (i.e., of $M$ endowed with the rescaled
metric by the factor $\lambda_n'$) restricted to $\lambda'_nB_M(p'_n,\frac{t}
{\lambda'_n})$ is greater than some positive constant independent of~$n$
 large.
\end{lemma}
\begin{proof}
Pick a point $z_n\in B_M(p'_n,\frac{t}{\lambda'_n})$. 
Since for $n$ large enough, $z_n$ belongs to $B_M(q_n,1)$,
we have
\begin{equation}
\label{eq:*} \frac{1}{\lambda'_nI_M(z_n)}=\frac{h_n(z_n)}{\lambda'_n
d_M(z_n,\partial B_M(q_n,1))}\leq
\frac{h_n(p'_n)}{\lambda'_n d_M(z_n,\partial B_M(q_n,1))}
=\frac{d_M(p'_n,\partial B_M(q_n,1))}{d_M(z_n,\partial B_M(q_n,1))}.
\end{equation}
By the triangle inequality, $d_M(p'_n,\partial B_M(q_n,1)) \leq
\frac{t}{\lambda'_n}+d_M(z_n,\partial B_M(q_n,1))$ and so,
\[
\frac{d_M(p'_n,\partial B_M(q_n,1))}{d_M(z_n,\partial B_M(q_n,1))}
\leq 1+\frac{t}{\lambda'_nd_M(z_n,\partial B_M(q_n,1))}
\]
\[
\leq 1+\frac{t}{\lambda'_n\left( d_M(p'_n,\partial
B_M(q_n,1)) -\frac{t}{\lambda'_n}\right) }
=1+\frac{t}{\lambda'_n d_M(p'_n,\partial
B_M(q_n,1)) -t}
\]
\begin{equation}
\label{eq:**}
\stackrel{(\ref{eq:hn2})}{\leq }1+\frac{t}{n-t},
\end{equation}
which tends to $1$ as $n\to \infty $, thereby proving the lemma.
\end{proof}

\subsection{A chord-arc property and the proof of
items {1, 2, 3} of Theorem~\ref{tthm3introd}.}
With the notation above, we define
\begin{equation}  \label{eq:item1}
t_n=\frac{\sqrt{n}}{2},\quad
{\textstyle \widetilde{M}(n)=\lambda_n'\overline{B}_M(p_n',\frac{t_n}{\lambda_n'}).}
\end{equation}
Since by (\ref{eq:hn2})
$h_n(p'_n)\geq n>t_n$, then $\frac{t_n}{\lambda_n'}
< \frac{h_n(p'_n)}{\lambda_n'}=d_M(p_n',\partial B_M(q_n,1))$.
Therefore, given any $z\in B_M(p_n',\frac{t_n}{\lambda_n'})$, we have
\[
d_M(z,q_n)\leq d_M(z,p_n')+d_M(p_n',q_n)<\frac{t_n}{\lambda_n'}+d_M(p_n',q_n)
\]
\[
\qquad<d_M(p_n',\partial B_M(q_n,1))+d_M(p_n',q_n)=1,
\]
 that is, $z\in B_M(q_n,1)$.
This last property lets us apply (\ref{eq:*}) and (\ref{eq:**}) to
conclude that for $n$ large and $z\in B_M(p_n',\frac{t_n}{\lambda_n'})$, we have
\begin{equation} \label{eq:item2}
\frac{1}{\lambda_n'I_M(z)}\leq 1+\frac{t_n}{n-t_n}\leq 2;
\end{equation}
hence the injectivity radius function of the complete surface $\lambda_n'M$ is
greater than  $\frac{1}{2}$ at any point in $\widetilde{M}(n)$.
This clearly implies that
\begin{enumerate}[(Inj)]
\item $\widetilde{M}(n)$ has injectivity
radius at least $\frac{1}{2}$ at points of distance greater than
$\frac{1}{2}$ from its boundary.
\end{enumerate}

\begin{proposition}
\label{lemma4.2}
Given $R_1>0$, there exists $\wt{\de} =\wt{\de} (R_1)\in (0,\frac{1}{2})$ such that for any $R
\in (0,R_1]$ and for $n$ sufficiently large, the closure $\Sigma (n,\wt{\de} R)$ of the component
of $\widetilde{M}(n)\cap B_{\lambda_n'N}(p_n',\wt{\de} R)$ passing through $p_n'$ has $\partial \Sigma
(n,\wt{\de }R)\subset \partial B_{\lambda_n'N}(p_n',\wt{\de} R)$ and satisfies
\begin{equation}
\label{eq:chordarc}
\Sigma (n,\wt{\de} R)\subset B_{\widetilde{M}(n)}(p_n',{\textstyle {\frac{R}{2}}}).
\end{equation}
Furthermore, the function $r\in (0,\infty )\mapsto \wt{\de}(r)\in (0,\frac{1}{2})$ can be chosen so that
$\wt{\de}(r)$ is nonincreasing and $r\, \wt{\de}(r)\to \infty $ as $r\to \infty $.
\end{proposition}
\begin{proof}
We will start by proving the following property.
\begin{enumerate}[(C)]
\item Given $R_1>0$, there exists $\wt{\de}=\wt{\de }(R_1) \in (0,\frac{1}{2})$ such that
with the notation of the lemma, the inclusion in (\ref{eq:chordarc}) holds for all $R\in(0,R_1]$.
\end{enumerate}
Arguing by contradiction, suppose there exists a sequence $\de _n\searrow 0$
so that $\Sigma (n,\de _nR_n)$ intersects the boundary of
$B_{\widetilde{M}(n)}(p_n',{\textstyle {\frac{R_n}{2}}})$ for some $R_n\leq R_1$.
Observe that we can assume that the number $R_0>0$ appearing in Theorem~\ref{thm2.2}
is not greater than $\frac{1}{2}$. For $n$ sufficiently large, all points
in $B_{\widetilde{M}(n)}(p_n',R_0)$ are at intrinsic distance greater than $\frac{1}{2}$ from
the boundary of $\widetilde{M}(n)$, and thus, the injectivity radius property (Inj)
implies that $B_{\widetilde{M}(n)}(p_n',R_0)$ is topologically a disk whose injectivity
radius function coincides with the distance to its boundary function.
This property together with the fact that $\de _n<\de'$ for $n$ large (here $\de'$ is the
positive constant that appears in Theorem~\ref{thm2.2}) allow us to apply Theorem~\ref{thm2.2}
to conclude that $R_n> R_0$.

As $\Sigma (n,\de _nR_n)\cap \partial B_{\widetilde{M}(n)}(p_n',{\textstyle {\frac{R_n}{2}}})
\neq \mbox{\O }$, then there exists a curve $\g _n\colon [0,1]\to \Sigma (n,\de _nR_n)$
such that $\g _n(0)=p_n'$,
$\g _n(1)\in \partial B_{\widetilde{M}(n)}(p_n',{\textstyle {\frac{R_n}{2}}})$
and $\g _n(t)\in B_{\widetilde{M}(n)}(p_n',{\textstyle {\frac{R_n}{2}}})$
for all $t\in [0,1)$. In particular, the length of $\g _n$
is at least $\frac{R_n}{2}$.

Consider the positive numbers $\displaystyle \tau_n=\frac{1}{\de_n R_n}\geq \frac{1}{\de_n R_1}\to \infty$
and note that the scaled surfaces
$\tau_n \Sigma (n,\de _nR_n)$ can be viewed as the closure of the component of
\[
\tau_n [\wt{M}(n) \cap B_{\lambda_n'N}(p_n',\de _nR_n)]
=\tau_n \wt{M}(n) \cap B_{\tau_n\lambda_n'N}(p_n',1)
\]
that passes through $p_n'$;
recall that $B_{\tau_n\lambda_n'N}(p_n',1)$ can be taken arbitrarily close to
$\B(1)$ with its standard flat metric.
Let $\wt{\g}_n\subset \tau_n \Sigma (n,\de _nR_n)$ be the related scaling of $\g_n$.
Since the intrinsic distance from $\wt{\g}_n(0)=\vec{0}=p_n'$ to $\wt{\g}_n(1) $
in $\tau_n \Sigma (n,\de _nR_n)$ is $\frac{\tau _nR_n}{2}=\frac{1}{2\de _n}\to \infty $,
then for $n$ large there exists a collection of points
$Q_n=\{q^n_1,\ldots,q^n_{k(n)}\}\subset  \wt{\g}_n \subset \B(1)$ whose intrinsic distances from
each other in $\tau_n\wt{M}(n)$ are diverging to infinity and with
$k(n)\to \infty$ as $n\to \infty$; here we are viewing $\B(1)$ as being
exponential coordinates for $B_{\tau_n\lambda_n'N}(p_n',1)$
and so, we can consider all of the curves $\wt{\g}_n$ to lie in
the open unit ball in $\rth$ with metrics converging to the usual flat one.  In particular,
there are positive numbers $r_n \to \infty$ such that
\[
\{B_{\tau_n \wt{M}(n)}(q^n_k,r_n)\mid k\in \{1,\ldots,k(n)\}\}
\]
forms a pairwise disjoint collection of intrinsic balls
contained in the interior of  $\tau_n \wt{M}(n)$, and property (Inj) implies that
\begin{enumerate}[\mbox{(Inj1)}]
\item The intrinsic distance from each $B_{\tau_n \wt{M}(n)}(q^n_k,r_n)$ to the boundary of
$\tau_n \wt{M}(n)$ is at least 1  for $n$ sufficiently large (this holds because $\tau_n \to \infty$
and $t_n\to \infty$), and the injectivity radius of $\tau _n\wt{M}(n)$ is at least 2 at points
of distance at least 2 from the boundary of $\tau _n\wt{M}(n)$.
\end{enumerate}

Since the number of points in $Q_n$ is diverging to infinity as $n\to \infty$,
then after replacing by a subsequence,
there exists a sequence of pairs of points $q^n_j\neq q^n_{j'}\in Q_n$ such that
$\{ q_j^n\} _n,\{ q_{j'}^n\} _n$ converge to the same point $q_\infty \in \ov{\B}(1)$.
By property (Inj1), $B_{\tau_n \wt{M}(n)}(q^n_j,1)$, $B_{\tau_n \wt{M}(n)}(q^n_{j'},1)$
are minimal disks that satisfy the hypotheses of Theorem~\ref{thm2.2}.
Therefore,
the closures $D_n, D_n'$ of the components of
$B_{\tau_n \wt{M}(n)}(q^n_j,1)\cap B_{\tau_n\lambda_n'N}(q^n_j,\de' R_0)$,
$B_{\tau_n \wt{M}(n)}(q^n_{j'},1)\cap B_{\tau_n\lambda_n'N}(q^n_{j'},\de' R_0)$ passing respectively
through $q^n_j , \, q^n_{j'}$
are disks with their boundaries in the respective ambient boundary spheres,
where $\de',R_0$ are defined in Theorem~\ref{thm2.2}.
For $n$ large enough, we may assume that the boundaries of extrinsic balls
of radius at
most~$1$ in $\tau_n\lambda_n'N$ are spheres of  positive mean curvature with respect to the inward
pointing normal vector. The
mean curvature comparison principle implies that
the respective components $\wt{D}_n, \wt{D}_{n}'$
of $D_n\cap B_{\tau_n\lambda_n'N}(q_\infty, \de' R_0/2),\, D_{n}'\cap B_{\tau_n\lambda_n'N}(q_\infty, \de ' R_0/2)$
passing through the points $q_j^n, q_{j'}^{n}$ are
disks with their boundary curves in  $B_{\tau_n\lambda_n'N}(q_\infty, \de' R_0/2)$,
see Figure~\ref{fig1}.
\begin{figure}
\begin{center}
\includegraphics[width=10.2cm]{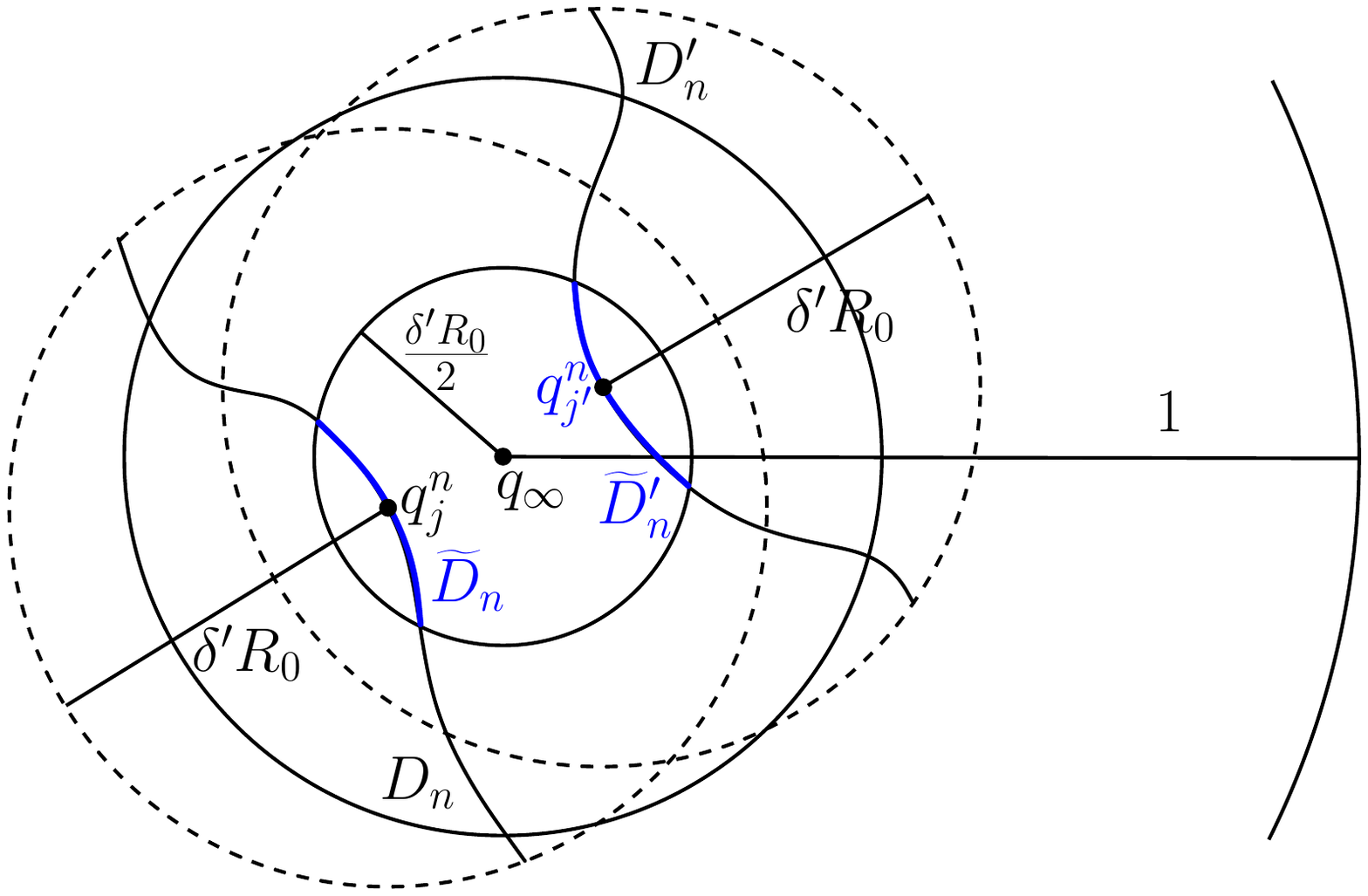}
\caption{The disks $\wt{D}_n, \wt{D}_{n}'$ are disjoint and become close to $q_{\infty }$.}
\label{fig1}
\end{center}
\end{figure}

As described in the proof of the minimal lamination closure theorem
in~\cite{mr13}, the extrinsic\footnote{One could instead use the intrinsic version
of the one-sided curvature estimates (Corollary~0.8 in Colding and Minicozzi~\cite{cm35}) to shorten this argument.}
one-sided curvature estimates for minimal disks of Colding-Minicozzi
(Corollary~0.4 in~\cite{cm23}) imply that there exists a  constant $C>0$ only depending on $N$ such that
the norm of the second fundamental forms of the subdisks of
 $\wt{D}_n, \wt{D}_{n}'$ in the smaller ball
 $B_{\tau_n\lambda_n'N}(q_\infty, \de' R_0/4)$ containing the respective points $q_j^n, q_{j'}^{n}$,
 are bounded by $C$ (see Theorem~7 in~\cite{mr13} for an exact statement of this result).
Since these subdisks have uniformly bounded second fundamental forms, a subsequence
of these subdisks  converges to a compact minimal disk $D(q_\infty)$ passing through $q_\infty$
with boundary in the boundary of the ball  $\B(q_\infty, \de' R_0/4)$ and
 the norm of the second fundamental form of $D(q_\infty)$ is everywhere
bounded from above by $C$. 
A prolongation argument (see for instance the proof of Theorem~4.37 in
P\'erez and Ros~\cite{pro2})
implies that $D(q_\infty)$ lies in a complete, embedded minimal surface
$M(\infty)$ in $\rth$ with its flat metric and with
the norm of the second fundamental form of $M(\infty)$ bounded from above by $C$;
furthermore, $M(\infty)$  must be proper in $\rth$ by Theorem~2.1 in~\cite{mr7}.

The above arguments also prove that for any fixed $T>0$, for $n$ sufficiently large,
the norms of the second fundamental forms of the intrinsic balls
$B_{\tau_n \wt{M}(n)}(q^n_j,T), \, B_{\tau_n \wt{M}(n)}(q^n_{j'},T)$,
are bounded from above by $2C$.  By Lemma~3.2 in~\cite{mpr20}, for $T$
sufficiently large, the boundary of the component of
$B_{\tau_n \wt{M}(n)}(q^n_j,T)\cap B_{\tau_n\lambda_n'N}(q_{j}^n,3)$ that passes through
$q^n_j$ lies on the boundary of the ball $B_{\tau_n\lambda_n'N}(q^n_j,3)$.
This is a contradiction since for $n$ sufficiently large,
the curve $\wt{\g}_n$ intersects this component, $\wt{\g}_n$ does not intersect
the boundary of the component and $\wt{\g}_n$ passes through the point
$q_{j'}^{n}\not \in B_{\tau_n \wt{M}(n)}(q^n_j,T)$. This contradiction
proves Property (C). 

Note that given $R_1>0$, we can assume that the number $t_n$ given by~(\ref{eq:item1})
satisfies $t_n>\frac{R_1}{2}$ for $n$ sufficiently large;
this implies that given $R\in (0,R_1]$, by definition of $\Sigma (n,\wt{\de }R)$, no points in
$\Sigma (n,\wt{\de }R)$ are in the boundary of $\wt{M}(n)$. Thus, the boundary
$\partial \Sigma (n,\wt{\de }R)$ is contained in the boundary of $B_{\lambda_n'N}(p_n',\wt{\de }R)$, as
stated in the first sentence of Proposition~\ref{lemma4.2}.

To finish the proof of Proposition~\ref{lemma4.2}, it remains to show that
$\wt{\de }=\wt{\de }(R_1)$ can be chosen so $\wt{\de }(r)$ is nonincreasing and
that $r\, \wt{\de }(r)\to \infty $ as $r\to \infty $.
Property (C) lets us define for each $r\in (0,\infty)$, $\wh{\de}(r)$ as
 the supremum of
the values $\wt{\de }\in (0,\frac{1}{2})$ such that  for $n$ sufficiently large,
the inclusion in (\ref{eq:chordarc}) holds for this value of $\wt{\de}$,  for all $R\leq r$.
Note that the function $r\mapsto \wh{\de}(r)$ is nonincreasing.
In order to complete the proof of the proposition it suffices to show that
$\lim_{r\to \infty} r \,\wh{\de}(r) =\infty$.

Arguing by contradiction,
suppose that
there exists a sequence $r_n\to \infty$ such that
$ r_n \,\wh{\de}(r_n) \leq K$, for some $K\in (1,\infty)$; in particular $\lim_{r\to \infty}\wh{\de}(r)=0$.
By the definition of $\wh{\de}(r)$, it follows that after choosing a subsequence,
\begin{equation}
\label{eq:chordarc2}
\Sigma (n,2\wh{\de}(r_n) R_n) \not \subset B_{\widetilde{M}(n)}(p_n',{\textstyle {\frac{R_n}{2}}}),
\;\mbox{\rm for some } R_n\in (R_0,r_n].
\end{equation}

We claim that
\begin{equation}
\label{eq:chordarc3}
\lim_{n\to \infty} R_n =\infty.
\end{equation}
Otherwise, after choosing a subsequence,
we have $R_n\leq C'$ for some  $C'>1$. Note that we may also assume  that
$\wh{\de}(r_n)<\frac12 \wh{\de}(C')$ for all $n\in \N$. But then,
\begin{equation}
\label{eq:chordarc4}
\Sigma (n,2\wh{\de}(r_n) R_n) \subset \S(n,\wh{\de}(C') R_n)
\subset \S(n,\wh{\de}(R_n) R_n)\stackrel{(\ref{eq:chordarc})}{\subset }
B_{\widetilde{M}(n)}(p_n',{\textstyle {\frac{R_n}{2}}}),
\end{equation}
which contradicts \eqref{eq:chordarc2}. Therefore, (\ref{eq:chordarc3}) holds.

Now define $\displaystyle \tau_n=\frac{1}{\wh{\de}(r_n) R_n}$. By property (Inj),
the rescaled surfaces $\tau_n \wt{M}(n)$ have injectivity radius bounded from below
by
\[
\frac{\tau _n}{2}=\frac{1}{2\wh{\de }(r_n)R_n}\geq \frac{1}{2\wh{\de }(r_n)r_n}\geq \frac{1}{2 K}>0
\]
at points of distance at least $\tau _n/2$ from its boundary.
Furthermore, the scaled surfaces
\[
\tau_n\Sigma (n,2\wh{\de}(r_n) R_n)
\]
can be viewed to be contained  in the ball $\ov{\B}(2)$.
As in the previous case where $r=R_1$ was fixed, one can define  curves $\g_n$ in
$\Sigma (n,2\wh{\de}(r_n) R_n)$ such that the associated scaled curves $\wt{\g}_n$
in $\tau_n\Sigma (n,2\wh{\de}(r_n) R_n)$ have  intrinsic distances between the end points
of $\wt{\g}_n$ diverging to infinity and $n\to \infty$.
From straightforward modifications of  the arguments in the first part of the proof of this
proposition, one arrives to a contradiction; for
these modifications one does not need that the injectivity
radii of the surfaces $\tau_n \wt{M}(n)$ are at least 2 but
just that they are bounded from below by a uniform constant. This contradiction proves that
$\lim_{r\to \infty} r \,\wh{\de}(r) =\infty$ and
completes the proof of the proposition after redefining $\wt{\de }(r)$ by $\wh{\de }(r)$.
\end{proof}

\begin{remark}
{\rm In Proposition~\ref{lemma4.2}, the value of $r\mapsto \wt{\de }(r)$ might depend {\it a priori}
on the homogeneously regular ambient manifold $N$ where the blow-up process
on the scale of topology was performed or
on the complete minimal surface $M$ with injectivity radius zero. In fact, this $\wt{\de }(r)$ can be
chosen independent upon $N$ because the inclusion in (\ref{eq:chordarc}) is invariant
under rescaling once the metric on the scaled manifold $\lambda_n' N$ is sufficiently
$C^2$-close  to a flat metric and the injectivity radius of $\lambda_n' N$ is at least 1.
A similar argument shows that $\wt{\de}(r)$ can be also chosen independent of the minimal surface $M$.
}
\end{remark}

We next continue with the proof of Theorem~\ref{tthm3introd}. Consider the nonincreasing function $\wt{\de}\colon (0,\infty) \to (0,\frac{1}{2})$ given by Proposition~\ref{lemma4.2}.  The function
$\de\colon (0,\infty) \to (0,\frac{1}{2})$
described in the statement of Theorem~\ref{tthm3introd} can be defined as any smooth decreasing
function such that $\frac{1}{2}\wt{\de}(r)\leq \de (r)\leq  \wt{\de}(r)$ for any $r>0$.
In particular, (\ref{eq:chordarc}) holds true after replacing $\wt{\de }$ by $\de$,
and $r\, \de (r)\to \infty $ as $r\to \infty $.

\begin{lemma}\label{lemma4.3}
 Items {1, 2, 3} of Theorem~\ref{tthm3introd} hold.
\end{lemma}
\begin{proof}
By the last statement in Proposition~\ref{lemma4.2}, for $k\in \N$, we can pick
values $r_k>0$ such that $r_k\, \de (r_k)=k$. In particular, Proposition~\ref{lemma4.2}
ensures that $\partial \S (n,k)\subset \partial B_{\lambda'_{n}N}(p_{n}',k)$ and
\begin{equation}
\label{eq:chordarc'}
\Sigma (n,k)\subset B_{\wt{M}(n)}(p_n',\frac{r_k}{2})
\end{equation}
for all $n\geq n(k)$, where $n(k)\in \N$ that can be assumed to tend to infinity as $k\to \infty$.
Furthermore, we can also assume $n(k)$ is chosen so that $\frac{k}{n(k)}\to 0$ as $k\to \infty $.
Defining
\[
\ve _k=\frac{k}{\lambda'_{n(k)}}=k\, I_M(p'_{n(k)}) 
\stackrel{(\ref{eq:hn2})}{\leq }
\frac{k}{n(k)},
\]
then $\ve _k\to 0$ as $k\to \infty$. Finally, we define for every $k\in \N$
\begin{equation}
\label{eq:Mk}
p_k=p'_{n(k)}\quad \mbox{ and }\quad M_k=\frac{1}{\lambda'_{n(k)}}\Sigma (n(k),k).
\end{equation}
Then, item~{1} of Theorem~\ref{tthm3introd} follows directly from (\ref{eq:chordarc'}).
Item~{2} of the theorem also holds from (\ref{eq:item2}) after replacing by a further subsequence,
and item~{3} of the theorem hold trivially since $N$ is homogeneously regular
and $\lambda_k=1/I_M(p_k)=\lambda'_{n(k)}$
tends to $\infty $ as $k\to \infty $. This completes the proof of the lemma, after replacing $k$ by $n$.
\end{proof}

From this point on  in the proof, we will assume that the first
three items of Theorem~\ref{tthm3introd} hold for
the sequence of points $p_n\in M$ and we will discuss two
cases in distinct subsections, depending on whether or not
a subsequence of the surfaces
$\lambda_nB_M(p_n,\frac{t}{\lambda_n})$ has uniformly bounded Gaussian
curvature (the bound could depend on $t>0$).
Before doing this, we will state a property which will be useful in both cases.

\begin{assertion} \label{beta-n}
For $n$ large, there exists an embedded geodesic loop $\be_n\subset \lambda_n M_n$
of length two based at $p_n$ (smooth except possibly at $p_n$) which is
homotopically nontrivial in $\lambda_n M_n$.
\end{assertion}
\begin{proof}
Since the surfaces $\lambda_n M_n$ are minimal and the sectional curvatures of the
ambient spaces $\lambda_n N$ are converging uniformly to zero,
then the Gauss equation implies that the exponential map $\exp _{p_n}$ of $ T_{p_n}(\lambda_n M_n)$
restricted to the closed metric ball of radius $2$ centered at the origin
is a local diffeomorphism. As the injectivity radius of $\lambda_nM_n$ at $p_n$ is 1, then
$\exp _{p_n}$ is a diffeomorphism when restricted to the open disk of radius $1$, and it fails
to be injective on the boundary circle of radius $1$.
Now it is standard to deduce the existence of a geodesic loop $\be_n$
as in the statement of the assertion, except for the property that $\be _n$
is homotopically nontrivial which we prove next.

Arguing by contradiction and after extracting a subsequence, assume that
$\be_n$ is homotopically trivial in $\lambda_n M_n$. Thus, $\be_n$ bounds a
disk $D_n\subset \lambda_n M_n$.
Observe that $\lambda_nM_n$ is contained in $B_{\lambda_nN}(p_n,n)$
(this follows from~(\ref{eq:Mk}) and from Definition~\ref{def2.3}).
Since the extrinsic spheres $\partial B_{\lambda_n N}(p_n,R)$ have positive mean curvature
for $n$ large with respect to the inward pointing normal vector
(because they are produced by rescaling of extrinsic balls of radius $R/\lambda_n\to 0$ in
the homogeneously regular manifold $N$)
and $D_n\subset B_{\lambda_n N}(p_n,n)$
with $\partial D_n=\be _n\subset B_{\lambda_nM_n}(p_n,2)\subset B_{\lambda_n N}(p_n,2)$,
then the mean curvature comparison principle implies
$D_n\subset B_{\lambda_n N}(p_n,2)$.  As the balls $B_{\lambda_n N}(p_n,2)$ are converging uniformly to the
flat ball $\B(2)$, and $\B(2)$ contains no closed minimal surfaces without boundary,
then the isoperimetric inequality in~\cite{wh13} implies that there exists an
upper bound $A_0$ for the areas of the disks $D_n$
(here we are using Theorem~2.1 in~\cite{wh13} on the mean convex balls $B_{\lambda_nN}(p_n,2)$ for $n$ large, hence the constant $A_0$ does in principle depend on~$n$;
the fact that $A_0$ can be taken independently
of $n$ large follows from the upper semicontinuous dependence of $A_0$ on the ambient
Riemannian manifold with mean convex boundary, see the sentence just before Corollary~2.4 in~\cite{wh13}).
As the limsup of the sequence of numbers
$\max _{D_n}K_{\lambda_nM_n}$ is nonpositive,
then the Gauss-Bonnet formula gives
\[
2\pi =\int_{D_n} K_{\lambda_nM_n} +\a \leq \int_{D_n} K_{\lambda_nM_n} +\pi ,
\]
where $\a $ is the angle of $\be _n$ at $p_n$. The above inequality is impossible, since as $n\to \infty$,
\[
\int_{D_n} K_{\lambda_nM_n} \leq \max _{D_n}(K_{\lambda_nM_n})\, A_0\leq
\limsup \left( \max _{D_n}K_{\lambda_nM_n}\right) \, A_0\leq 0.
\]
 This contradiction proves
that the embedded geodesic loop $\be_n$ is homotopically nontrivial in $\lambda_n M_n$.
\end{proof}

\subsection{The case of uniformly bounded Gaussian curvature on compact subsets of $\rth$.}
\begin{proposition}
\label{propos4.2}
In the situation above, suppose that for every $t>0$, the surfaces
$\lambda_nB_M(p_n,\frac{t}{\lambda_n})$
have uniformly bounded Gaussian curvature.
Then, item~4 of Theorem~\ref{tthm3introd} holds.
\end{proposition}
\begin{proof}
In the special case that there exists $C>0$ so that for every $t>0$
there exists $n(t)\in \N$ such that the surfaces
$\lambda_nB_M(p_n,\frac{t}{\lambda_n})$, $n\geq n(t)$,
have absolute Gaussian curvature bounded by $C$, then a complete proof of
this proposition can be  found in Section~3 of~\cite{mpr20}. We will next
modify some of those arguments in order to deal with the more general current
situation, where the bound $C$ might depend on $t>0$.

Fix $t>0$. As by hypothesis the surfaces $\lambda_nB_M(p_n,\frac{t}{\lambda_n})$
have uniformly bounded Gaussian curvature, then they also have uniformly
bounded area by comparison theorems in Riemannian geometry (Bishop's second theorem, see e.g.,
Theorem III.4.4 in Chavel~\cite{ch2}).
After extracting a subsequence, the compact surfaces
$\lambda_n\overline{B}_M(p_n,\frac{t}{\lambda_n})$ converge on compact sets of $\R^3$
(possibly with integer nonconstant multiplicities)
to an embedded, compact minimal surface with boundary
$M_{\infty }(t)\subset \overline{\B}(t)$, with bounded Gaussian curvature, such that $\vec{0}$
lies in the interior of $M_{\infty }(t)$.  We claim that
the intrinsic distance from $\vec{0}$ to
$\partial M_{\infty }(t)$ is $t$. To see this, first note that this
intrinsic distance is clearly at most $t$, as $t$ is the radius of
$\lambda_nB_M(p_n,\frac{t}{\lambda_n})$.  Let $\a $ be a
 minimizing geodesic from $\vec{0}$ to $\partial M_{\infty
}(t)$.  For $n$ large, one can lift  $\a$ normally to
nearby arcs
$\a_n$ on $\lambda_nB_M(p_n,\frac{t}{\lambda_n})$, each of which
starts at $\vec{0}$ and has one end point in
$\partial [\lambda_nB_M(p_n,\frac{t}{\lambda_n})]$, so that their lengths converge as
$n\to \infty $ to the length of $\a $. Clearly the length of $\a _n$
is at least $t$; hence after taking limits we deduce that the length
of $\a $ is at least~$t$, and our claim is proved.

Consider an increasing sequence $1=t_1<t_2<\ldots $ with $t_m\to \infty $ as $m\to \infty $.
For $m=1$, consider the compact surface $M_{\infty }(1)$ together with the
sequence of surfaces $\lambda_n\overline{B}_M(p_n,\frac{1}{\lambda_n})$ that converges to it
(after passing to a subsequence). For this sequence, consider the corresponding
intrinsic balls $\lambda_n\overline{B}_M(p_n,\frac{t_2}{\lambda_n})$. After extracting a subsequence,
these surfaces converge to $M_{\infty }(t_2)$; in particular, $M_{\infty }(1)\subset
M_{\infty }(t_2)$. Repeating this argument and using a diagonal subsequence, one can construct the surface
\[
M_{\infty }=\bigcup _{m=1}^{\infty }M_{\infty }(t_m).
\]
As the intrinsic distance from $\vec{0}$ to
$\partial M_{\infty }(t_m)$ is $t_m$ for every $m\in \N$,
then
$M_{\infty }$ is a complete, injectively immersed minimal surface in $\R^3$ without boundary.
Observe that for every $m$,
the convergence of the limit of the surfaces
$\lambda_n\overline{B}_M(p_n,\frac{{t_m}}{\lambda_n})$ to
$M_{\infty }({t_m})$ is one, as follows,
for example, from
the arguments in the proof of Lemma~3.1 in~\cite{mpr20} (higher multiplicity produces
a positive Jacobi function on $M_{\infty }$, hence $M_{\infty }$ is stable and so,
$M_{\infty }$ is a plane, which contradicts the following assertion).

\begin{assertion}
  $M_{\infty }$ is not a plane.
\end{assertion}
\begin{proof}
Consider for each $n\in \N $ large, the embedded geodesic loop $\be_n\subset \lambda_n M_n$
of length two based at $p_n$ given by Assertion~\ref{beta-n}. Clearly, $\be _n\subset \lambda_n\overline{B}_M(p_n,\frac{1}{\lambda_n})$.
As the surfaces $\lambda_n\overline{B}_M(p_n,\frac{1}{\lambda_n})$ converge on compact subsets of $\R^3$ to $M_{\infty }(1)$,
then the $\be _n$ converge after passing to a subsequence to an embedded geodesic loop $\be _{\infty }\subset M_{\infty }(1)$,
which is impossible if $M_{\infty }$ were a plane. Hence the assertion follows.
\end{proof}
We next analyze the injectivity radius function of $M_{\infty }$.
Fix $m\in N$. Since $M_{\infty }(t_m)$ is compact and injectively immersed, then
there exists $\mu ({t_m})>0$ such that $M_{\infty }({t_m})$ admits a regular
neighborhood $U({t_m})\subset \R^3$ of radius $\mu (t_m)$
and we have a related normal projection
\[
{\Pi _{t_m}\colon U(t_m)\to M_{\infty }(t_m).}
\]
In this setting, Lemma~3.1 in~\cite{mpr20} applies\footnote{In Section~3
of~\cite{mpr20} we had the additional hypothesis that $M_{\infty }$ has
globally bounded Gaussian curvature, hence it is proper; nevertheless, the proof
of Lemma~3.1 in~\cite{mpr20} only uses that $M_{\infty }({t_m})$ is compact
for each $m\in \N$ and that $M_{\infty }$ is not a plane
(or more precisely, that
the convergence of the limit $\{ \lambda_n\overline{B}_M(p_n,\frac{t_m}{\lambda_n})\} _n
\to M_{\infty }(t_m)$ is
 one, which in turn follows from the fact that $M_{\infty }$ is not a plane),
conditions which are
satisfied in our current setting.}
to give the following property:
\begin{assertion}
\label{ass4.3}
Given $m\in \N$, there exists $k\in \N$ such that if $n\geq k$,
then $\lambda_nB_M(p_n,\frac{{t_m}}{\lambda_n})$ is  contained in $U(t_{m+1})$
and $\lambda_nB_M(p_n,\frac{{t_m}}{\lambda_n})$ is a small normal
graph over its projection to $M_{\infty }(t_{m+1})$, i.e.,
\[
{\textstyle
(\Pi _{t_{m+1}})|_{\lambda_nB_M(p_n,\frac{t_m}{\lambda_n})}\colon
\lambda_nB_M(p_n,\frac{t_m}{\lambda_n})
\to \Pi _{t_{m+1}}\left( \lambda_nB_M(p_n,\frac{t_m}{\lambda_n})\right) \subset M_{\infty }(t_{m+1})
}
\]
is a diffeomorphism.
\end{assertion}
\par
\vspace{.2cm}
We now remark on some properties of the minimal
surface $M_{\infty }\subset \R^3$.  By Assertion~\ref{ass4.3},
given $m\in \N$ there exists $k\in \N$ such that we can induce the metric
of $\lambda_n\overline{B}_M(p_n,\frac{{t_m}}{\lambda_n})$ to its projected
image $\Pi _{t_{m+1}}\left( \lambda_nB_M(p_n,\frac{t_m}{\lambda_n})\right) $
through the diffeomorphism $\Pi _{t_{m+1}}$. As the sequence
$\{ \lambda_n\overline{B}_M(p_n\frac{{t_m}}{\lambda_n})\} _n$ converges to
$M_{\infty }({t_m})$ with multiplicity one
as $n\to \infty $, then using
the continuity of the injectivity radius function with respect
to the Riemannian metric on a given compact surface
(see Ehrlich~\cite{ehr1} and Sakai~\cite{sa2})
and inequality (\ref{eq:**}),
we deduce that
\begin{equation}
\label{eq:a}
\left( I_{M_{\infty }}\right) |_{M_{\infty }({t_m})}=
\lim _{n\to \infty }
\left( I_{\lambda_nM}\right) |_{\lambda_n\overline{B}_M(p_n,\frac{{t_m}}{\lambda_n})}
\geq \lim _{n\to \infty }\frac{1}{1+\frac{{t_m}}{n-{t_m}}}=1.
\end{equation}

Hence,
we conclude that $I_{M_{\infty }}\geq 1$ everywhere on $M_{\infty }$,
with $I_{M_{\infty }}(\vec{0})=1$.

Since $M_{\infty }\subset \rth$ is a complete embedded
minimal surface in $\R^3$ with positive
injectivity radius, the minimal lamination closure
theorem~\cite{mr13} insures
that $M_{\infty }$ is properly embedded
in $\R^3$. This finishes the proof of Proposition~\ref{propos4.2}.
\end{proof}

\subsection{The case of Gaussian curvature not uniformly bounded.}
Suppose now that the uniformly bounded Gaussian curvature
hypothesis in Proposition~\ref{propos4.2} fails to hold.  It follows,
after extracting a subsequence, that for some fixed positive number
$t_1>0$, the maximum absolute Gaussian curvature of the surfaces 
$\lambda'_nB_M(p_n,\frac{t_1}{\lambda_n})$ diverges to infinity as $n\to
\infty$. To finish the proof of Theorem~\ref{tthm3introd}, it remains to show that
under this condition, items 5 or 6 hold.

\begin{definition}
{\em
A sequence of compact embedded minimal surfaces
$\Sigma_n$ in $\rth$ with boundaries diverging in space, is called
{\em uniformly locally simply connected}, if there is an $\ve >0$ such that
for any ambient ball $\B $ of radius $\ve >0$ and for $n$ sufficiently large, $\B $
intersects $\Sigma_n$ in compact disks with boundaries in the boundary of $\B $
(this definition is more restrictive than the similarly defined notion
in the introduction of~\cite{cm25}, where $\ve $ might depend on $\B $).
}
\end{definition}

We next return to the proof of Theorem~\ref{tthm3introd}.
By the discussion in the proof of Proposition~\ref{lemma4.2} (also see Theorem \ref{thm2.2}),
the sequence of minimal surfaces
$\lambda_nM_n$ can be considered to be uniformly locally
simply connected, as the metric  balls containing the surfaces are
converging to $\rth$ with the usual metric.
Thus,
\begin{enumerate}[(ULSC)]
\item There exists $\ve _1\in (0,1/2)$ such that for every ball
$\B \subset \R^3$ of radius $\ve _1$, and for $n$ sufficiently large, $(\lambda_nM_n)\cap \B $
consists of disks with boundaries in $\partial \B $.
\end{enumerate}
In this situation,
several results by Colding and Minicozzi~\cite{cm21,cm22,cm23,cm25}
apply to describe the nature of
both the surfaces $\lambda_nM_n$ in the sequence and their limit objects after
passing to a subsequence. We next briefly explain this description, which can also be
modified to work in the setting of a homogeneously regular manifold (see
for instance page 33 of~\cite{cm21} and~\cite{mr13}).

\begin{definition}
\label{kgraph}
{\rm
In polar coordinates $(\rho, \theta )$ on $\R^2-\{ 0\} $ with $\rho >0$ and
$\theta \in \R $, a {\it $k$-valued graph on an annulus of inner radius
$r$ and outer radius $R$,} is a single-valued graph of a real-valued function $u(\rho ,\theta )$
defined over
\begin{equation}
\label{eq:sectormultigraph}
S_{r,R}^{-k,k}=\{ (\rho ,\theta )\ | \ r\leq \rho \leq R,\ |\theta |\leq k\pi \} ,
\end{equation}
$k$ being a positive integer.
}
\end{definition}

By the one-sided curvature estimates for minimal disks
as applied in the proof of Theorem~0.1 in~\cite{cm23}
(also see the proof of Theorem~0.9 in~\cite{cm25}),
there exists a closed set $\mathcal{S}\subset \R^3$,
a nonempty minimal lamination $\mathcal{L}$ of $\R^3-\mathcal{S}$ which
cannot be extended across any proper closed subset of $\mathcal{S}$, and a subset $S(\mathcal{L})
\subset \mathcal{L}$ which is closed in the subspace topology, such that
after replacing by a subsequence,
$\{ \lambda_nM_n\} _n$ has uniformly bounded second fundamental form
on compact subsets of $\R^3-\Delta (\mathcal{L}) $ where $\Delta (\mathcal{L}) =
\mathcal{S}\cup S(\mathcal{L})$, and $\{ \lambda_nM_n\} _n$ converges $C^{\a}$,
$\a\in (0,1)$, on compact subsets of
$\R^3-\Delta (\mathcal{L}) $ to $\mathcal{L}$.

Around each point $p \in \Delta (\mathcal{L})$, the surfaces $\lambda_nM_n$
have the following local description. By (ULSC) and
Theorem~5.8 in~\cite{cm22}, there exists $\ve \in (0,\ve _1)$ such that
after a rotation of $\R^3$ and extracting a subsequence, each of the disks
$(\lambda_nM_n)\cap \overline{\B }(p, \ve)$ contains a  $2$-valued minimal graph defined on an annulus
$\{ (x_1,x_2,0)\ | \ r_n^2\leq x_1^2+x_2^2\leq R^2\} $
with inner radius $r_n\searrow 0$, for certain $R\in (r_n,\ve )$ small but fixed.

\begin{enumerate}[(D1)]
\item If $p \in S(\mathcal{L})$ (in particular, $\mathcal{L}$ admits a local lamination
structure around $p$), then after possibly choosing a smaller $\ve >0$,
there exists a neighborhood of $p$
 in $\overline{\B}(p, \ve)$ which is foliated by compact disks in $\mathcal{L}\cap
 \overline{\B }(p,\ve )$, and $S(\mathcal{L})$ intersects this family of disks transversely in a connected
Lipschitz arc. This case corresponds to case (P) described in Section~II.2 of~\cite{cm25}.
In fact, the Lipschitz curve
$S(\mathcal{L})$ around $p$ is a $C^{1,1}$-curve orthogonal to the local foliation~(Meeks~\cite{me25, me30}).
See Figure~\ref{CM} left.
\begin{figure}
\begin{center}
\includegraphics[width=11.4cm]{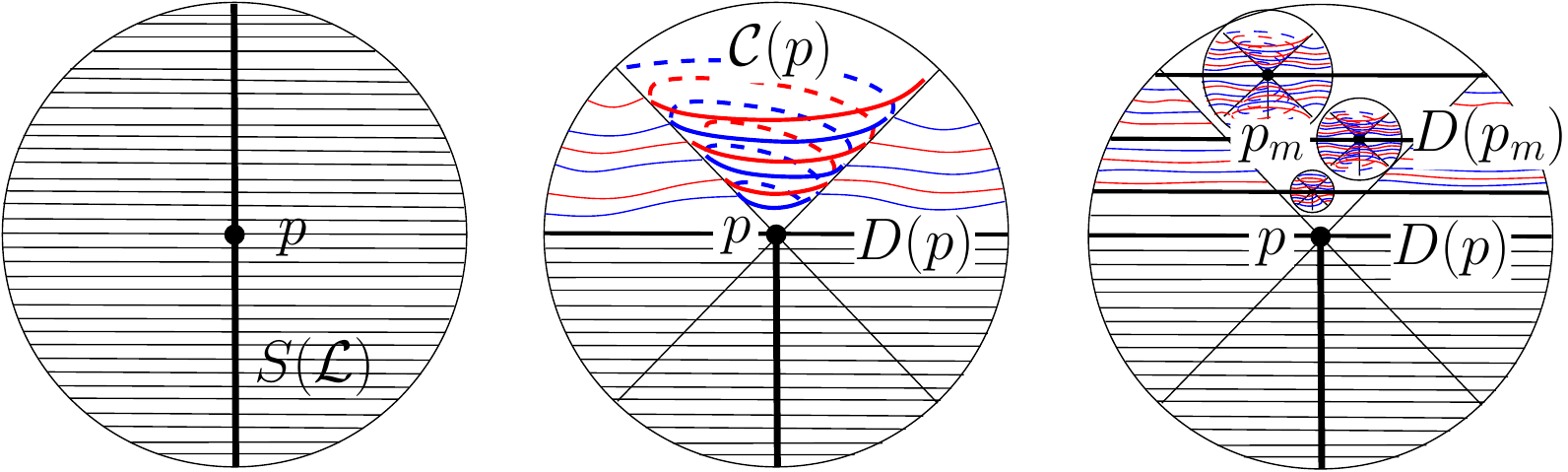}
\caption{Left: Case (D1), in a neighborhood of a point $p\in S(\mathcal{L})$.
Center: Case (D2-A) for an isolated point $p\in \mathcal{S}$; in the picture,
$p$ is the end point of an arc contained in $S(\mathcal{L})$, although $D(p,*)$ could
also be the limit of two pairs of multivalued graphical leaves, one pair on each side.
Right: Case (D2-B) for a nonisolated point $p\in \mathcal{S}$.}
\label{CM}
\end{center}
\end{figure}

\item  If $p\in \mathcal{S}$, then
after possibly passing to a smaller $\ve $,
a subsequence of the
surfaces $\{ (\lambda_nM_n)\cap \B (p,\ve )\}_n$
(denoted with the same indexes $n$) converges $C^\a$, $\a\in (0,1)$,
on compact subsets of $\B (p,\ve )-\Delta (\mathcal{L})$ to
the (regular) lamination $\mathcal{L}_p=\mathcal{L}\cap \B(p,\ve)$
of $\B(p,\ve )-\mathcal{S}_p$, where $\mathcal{S}_p =\mathcal{S}\cap  \B (p, \ve)$.
Furthermore, $\mathcal{L}_p$ contains a limit leaf $D(p,*)$ which
is a stable, minimal punctured disk with $\partial D(p,*)\subset
\partial \B (p, \ve)$ and $\overline{D(p,*)} \cap \mathcal{S}_p
=\{ p\}$, and $D(p,*)$ extends to a stable, embedded minimal
disk $D(p)$ in $\B(p, \ve)$
(this is Lemma II.2.3 in~\cite{cm25}). By
Corollary~I.1.9 in~\cite{cm23},
there is a solid double cone\footnote{A {\it solid double cone} in $\R^3$ is a set that after a
rotation and a translation, can be written as $\{ (x_1,x_2,x_3)\ | \ x_1^2+x_2^2
\leq \de ^2x_3^2\} $ for some $\de >0$. A solid double cone in a ball is the intersection
of a solid double cone with a ball centered at its vertex.}
$\mathcal{C}_p\subset \B (p, \ve)$ with vertex at
$p$ and axis orthogonal to the tangent plane $T_pD(p)$, that
intersects $D(p)$ only at the point $p$ (i.e., $D(p, *) \subset \B (p,\ve )-\mathcal{C}_p$)
and such that $[\B (p,\ve )-\mathcal{C}_p]\cap \Delta (\mathcal{L}) = \mbox {\O}$.
Furthermore for $n$ large, $(\lambda_nM_n)\cap \B (p, \ve)$ has the appearance
outside $\mathcal{C}_p$ of a highly-sheeted double multivalued graph over
$D(p,*)$, see Figure~\ref{fig2cone}.
\end{enumerate}
\begin{figure}
\begin{center}
\includegraphics[width=10cm]{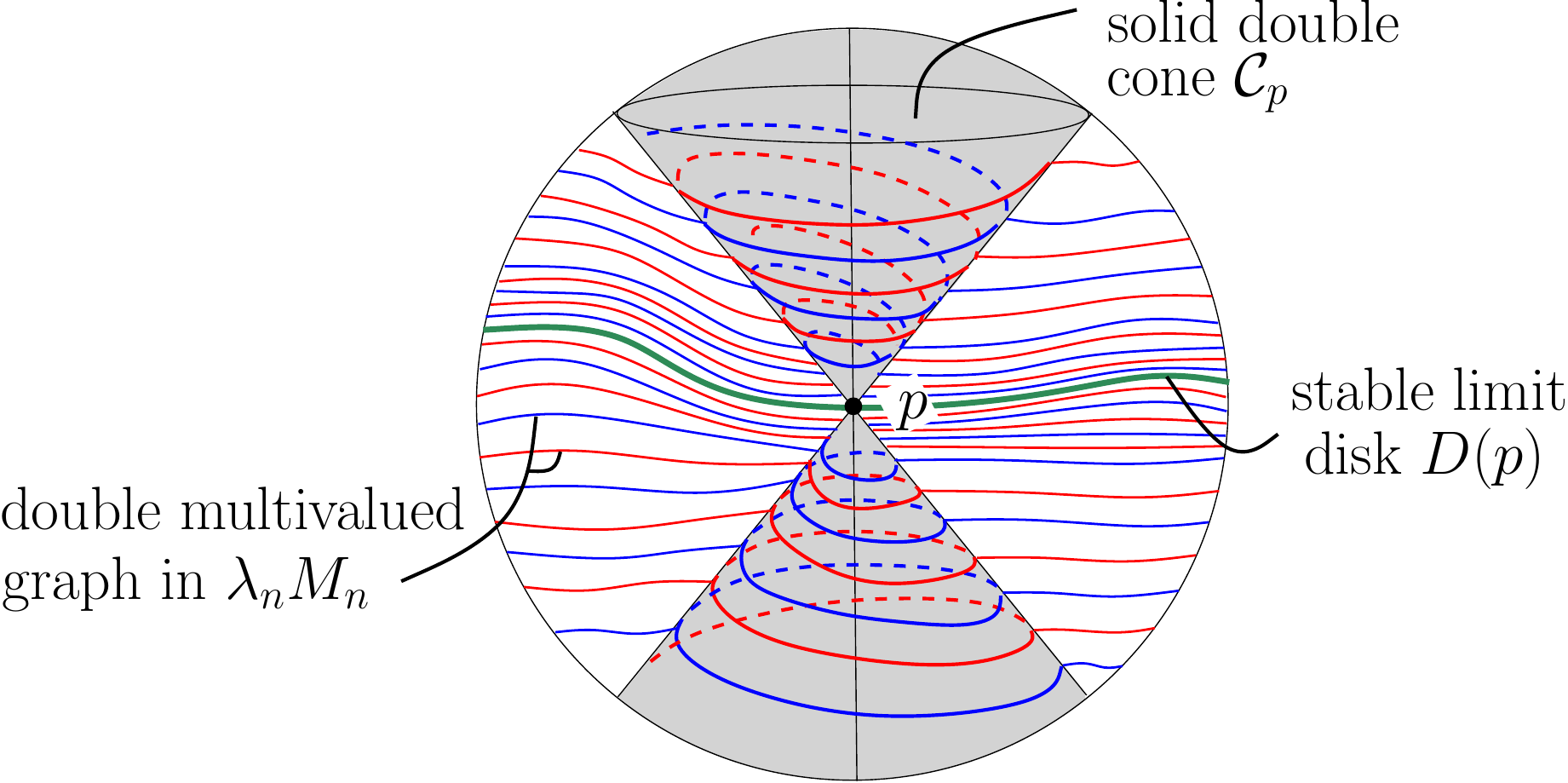}
\caption{The local picture of disk-type portions of $\lambda_nM_n$ around an
isolated point $p\in \mathcal{S}$. The stable minimal punctured disk
$D(p,*)$ appears in the limit lamination $\mathcal{L}_p$, and extends smoothly
through $p$ to a stable minimal disk $D(p)$ which is orthogonal at $p$ to the axis of
the double cone $\mathcal{C}_p$.}
\label{fig2cone}
\end{center}
\end{figure}

In this local description of this case (D2), it is worth distinguishing
two subcases:

\begin{description}
\item[{\rm (D2-A)}] If $p$ is an isolated point in $\mathcal{S}$, then the limit leaf $D(p,*)$
of $\mathcal{L}_p$ is either the limit of two pairs of multivalued graphical leaves in $\mathcal{L}_p$ (one
pair  on each side of $D(p,*)$),
or $D(p,*)$ is the limit on one side of just one pair of multivalued graphical leaves
in $\mathcal{L}_p$; in this last case, $p$ is the end point of an open arc $\G \subset
S(\mathcal{L})\cap \mathcal{C}_p$, and a neighborhood of $p$ in the closure of the component
of $\B(p,\ve )-D(p,*)$ that contains $\G $ is entirely foliated by disk leaves of $\mathcal{L}_p$.
See Figure~\ref{CM} center.

\item[{\rm (D2-B)}] $p$ is not isolated as a point in $\mathcal{S}$. In this case, $p$ is the limit
of a sequence $\{ p_m\} _m\subset \mathcal{S}\cap \mathcal{C}_p$. In particular, $D(p)$ is the limit
of the related sequence of stable minimal disks $D(p_m)$, and $D(p,*)$ is the limit of a sequence of
pairs of multivalued graphical leaves of $\mathcal{L}_p\cap [\B (p,\ve )-(\mathcal{C}_p\cup \{ D(p_m)\} _m]$.
Note that these singular points $p_m$ might be isolated or not in $\mathcal{S}$. See
Figure~\ref{CM} right.
\end{description}

We next continue with the proofs of items 5 and 6 of Theorem~\ref{tthm3introd}.
Since the maximum absolute Gaussian curvature of the surfaces $\lambda_nB_M(p_n,\frac{t_1}{\lambda_n})$
diverges to infinity as $n\to
\infty$ and $\lambda_nB_M(p_n,\frac{t_1}{\lambda_n})$ is contained in  $\lambda_nM_n$
for $n$ sufficiently large, then $\Delta (\mathcal{L})$ is nonempty and contains a
point which is at an extrinsic distance at most $t_1$ from the
origin in $\rth$.

\begin{lemma}
\label{lemma4.5}
  Let $p\in \Delta (\mathcal{L})$. Then, there exists a limit leaf $L_p$ of $\mathcal{L}$
whose closure $\overline{L_p}$ in $\R^3$ is a plane
passing through $p$. Moreover, the set
\[
\mathcal{P} '= \{ \overline{L_p} \mid p \in \Delta(\mathcal{L})\}
\]
is a nonempty, closed  set of planes.
\end{lemma}
\begin{proof}
The previous description (D1)-(D2) shows that there exists a minimal disk $D(p)$ passing
through $p$ such that $D(p,*)=D(p)-\{ p\} $ is contained in a limit leaf $L$ of $\mathcal{L}$.
As $L$ is a limit leaf of $\mathcal{L}$, then the two-sided cover $\wh{L}$ of $L$ is stable
(Meeks, P\'erez and Ros~\cite{mpr18,mpr19}). Consider the union $\wt{L}$ of $L$
with all points $q\in \mathcal{S}$ such that the related punctured disk
$D(q,*)$ defined in (D2) above is contained in $L$. Clearly, the two-sided
cover of $\wt{L}$ is stable and thus, the classification of complete stable minimal surfaces in
$\R^3$ (see e.g. do Carmo and Peng~\cite{cp1}, Fischer-Colbrie and Schoen~\cite{fs1} or Pogorelov~\cite{po1})
implies that to prove the lemma it remains to demonstrate
that $\wt{L}$ is complete.

Arguing by contradiction, suppose that there exists a shortest unit speed
geodesic $\g \colon [0,l)\to L$ such that $\g (0)\in L$ and $p:=\lim _{t\to l^-}\g (t)\in \mathcal{S}$.
Therefore, there exists $\de >0$ such that $\g (t)\in \overline{\B }(p,\ve )$ for every
$t\in [l-\de ,l)$, where $\overline{\B }(p,\ve )$ is the closed ball that appears in (D2).
Note that by construction, $\g (t)\notin D(p,*)$ for every $t\in [l-\de ,l)$.
As $D(p)$ separates $\overline{\B }(p,\ve )$, then $\g ([l-\de ,l))$ is contained in one of the
two half-balls of $\overline{\B }(p,\ve )-D(p)$, say in the upper half-ball (we can choose
orthogonal coordinates in $\R^3$ centered at $p$ so that $T_pD(p)$ is the $(x_1,x_2)$-plane).
In particular, there cannot exist a sequence
of $\{ p_m\} _m\subset \mathcal{S}$ converging to $p$ in that upper half-ball (otherwise $p_m$
produces a related disk $D(p_m)$ which is proper in the upper half-ball, such that
the sequence $\{ D(p_m)\} _m$ converges to $D(p)$ as $m\to \infty$; as $\g (l-\de )$ lies
above one of these disks $D(p_k)$ for $k$ sufficiently large, then $\g ([l-\de ,l))$ lies
entirely above $D(p_k)$, which contradicts that $\g $ limits to $p$).
Therefore,  after possibly choosing a smaller $\ve $, we can assume that there are no
points of $\mathcal{S}$ in the closed upper half-ball of  $\overline{\B }(p,\ve )-D(p)$ other than $p$.
Now consider the lamination $\mathcal{L}_1$ of $\overline{\B }(p,\ve )-\{ p\} $ given
by $D(p)$ together with the closure of $L\cap \overline{\B }(p,\ve )$ in
$\overline{\B }(p,\ve )-\{ p\} $. As the leaves of $\mathcal{L}_1$ are all stable (if $L$ is two-sided; otherwise
its two-sided cover is stable), then
Corollary~7.1 in~\cite{mpr10} implies that $\mathcal{L}_1$ extends smoothly across $p$,
which is clearly impossible. This contradiction proves the first statement in the lemma.

We now prove the second statement of the lemma. Suppose that
$\{ p_m\} _m\subset \Delta (\mathcal{L})$ and the related planes
$\overline{L_{p_m}}$ converge to a plane $P \subset \R^3$ as $m\to \infty $.
Arguing by contradiction, we suppose that $P\cap \Delta (\mathcal{L})= \mbox{\O }$.
Given $m\in \N$, consider a closed disk $D(q_m,\ve _1)$ in $P$ of radius
$\ve _1>0$ centered at the orthogonal projection $q_m$ of $p_m$ over $P$,
where $\ve _1\in (0,1/2)$ was defined in Property (ULSC).
As $P$ lies in $\mathcal{L}$ (because $\mathcal{L}$ is closed in $\R^3-\mathcal{S}$
and $P\cap \mathcal{S}=\mbox{\O }$) and $P$ does not contain points of $\Delta (\mathcal{L})$, then $q_m$
is at positive distance from $\Delta (\mathcal{L})$; in particular,
$D(q_m,\ve _1)$ can be arbitrarily well-approximated by
almost-flat closed disks $D_{n,m}$ in $\lambda_nM_n$ for $n$ large.
For $m$ large, the component $\Omega _n(p_m)$ of $(\lambda_nM_n)
\cap \overline{\B}(q_m,\frac{\ve _1}{2})$ that contains $p_m$ is a compact disk which is
disjoint from the almost-flat compact disk $D_{n,m}\cap \overline{\B}(q_m, \frac{\ve _1}{2})$
 and thus, $\Omega _n(p_m)$ lies at one side of $D_{n,m}\cap \overline{\B}(q_m, \frac{\ve _1}{2})$.
 As $p_m\in \Omega _n(p_m)$ is arbitrarily close to $D_{n,m}\cap \overline{\B}(q_m, \frac{\ve _1}{2})$
 (for $m$ large), then we contradict the one-sided curvature estimates for embedded minimal disks
(Corollary~0.4 in~\cite{cm23}). Now the proof of the lemma is complete.
\end{proof}

{\bf In the sequel, we will assume that the planes in $\mathcal{P}'$ are  horizontal.}

Recall that our goal in this section is to prove that items 5 or 6 of Theorem~\ref{tthm3introd} occur.
The key to distinguish which of these options occurs will be based on the singular set $\mathcal{S} $ of $\mathcal{L}$:
If $\mathcal{S} =\mbox{\O}$ (hence $\mathcal{L}$ is a lamination of $\R^3$) then item~5 holds, while if $\mathcal{S} \neq
\mbox{\O }$ then item~6 holds. This distinction is equivalent to $\mathcal{L} =\mathcal{P} $ or $\mathcal{L} \neq \mathcal{P}$.
The arguments to prove this dichotomy
are technical and delicate; we will start by adapting some of the arguments in the last paragraph
of the proof of Lemma~\ref{lemma4.5} to demonstrate the following result.

\begin{lemma}
\label{lemma4.6}
Every flat leaf of $\mathcal{L}$ lies in a plane in $\mathcal{P} '$, and no plane in $\R^3$ is disjoint from $\mathcal{L}$.
\end{lemma}
\begin{proof}
Arguing by contradiction, suppose $L$ is a flat leaf in $\mathcal{L}$ which
does not lie in a plane in $\mathcal{P}'$. Hence, $L$ does not intersect $\Delta(\mathcal{L})$.
This implies that $L$ is a plane and arbitrarily large disks in $L$ can be approximated
by almost-flat disks in the surfaces $\lambda_n M_n$. Since
these surfaces have injectivity radius greater than $1/2$ at points at distance at least $1/2$ from
their boundaries, then the one-sided curvature estimates for minimal  disks (Corollary~0.4 in~\cite{cm23})
imply that there are positive constants $\de $ and $C$, both independent of $L$, such that for $R>0$ and
for $n$ sufficiently large, the surface
\[
(\lambda_n M_n)\cap \{ |x_3-x_3(L)|<\de \} \cap \B
(\vec{0},R)
\]
 has Gaussian curvature less than $C $. From here, we deduce that
the leaves of 
$\mathcal{L}\cap \{ |x_3-x_3(L)|<\de \} $
 have uniformly
bounded Gaussian curvature. From this bounded curvature hypothesis,
the proof of Lemma 1.3 in~\cite{mr8} implies that
$\{ |x_3-x_3(L)|<\de \} \cap \mathcal{L}$
consists only of planes of $\mathcal{L}$; hence, the distance from $L$ to
$\Delta(\cL)$ is at least $\de$.
Let $\mathcal{L}'$ be the minimal lamination of $\R^3-\mathcal{S}$ obtained by
enlarging $\mathcal{L}$ by adding to it all planes which are disjoint from $\mathcal{L}$.
Note that by the one-sided curvature estimates in~\cite{cm23}, each of the
added on planes is also at a fixed distance at least $\de >0$ from
$\Delta(\mathcal{L})$ and from leaves of $\mathcal{L}$ which are not flat, where $\de$
is the same small number defined previously. Hence, the planes of $\mathcal{L}'$
which are not in $\mathcal{P}'$ form a both open and closed subset
of $\R^3$,
but $\rth$ is connected. Hence, this set is empty.

Note that the arguments in the last paragraph also prove that no
plane in $\rth - \mathcal{L}$ is disjoint from $\mathcal{L}$, which proves the lemma.
\end{proof}

Clearly, the closure of every flat leaf of $\mathcal{L}$ is an element of $\mathcal{P}'$ and vice versa, which
gives a bijection between $\mathcal{P}'$ and the collection $\mathcal{P}$ that appears in the statement of
item~6 of Theorem~\ref{tthm3introd}.

\begin{lemma}
\label{lemma4.7}
Consider a point $x\in \Delta (\mathcal{L})$ and the plane $\overline{L_x}\in \mathcal{P} '$
passing through $x$. Then, the distance between any
  two points in $\overline{L_x}\cap \Delta (\mathcal{L})$ is at least 1.
 \end{lemma}
\begin{proof}
Given any $p\in \overline{L_x}\cap \Delta (\mathcal{L})$,
$\overline{L_x}$ contains the disk $D(p)$ that appears in description (D2) above.
Since the sequence $\{ \lambda_nM_n\} _n$ is uniformly
locally simply connected, the Colding-Minicozzi
local picture (D1)-(D2) of $\lambda_nM_n$ near a point of $\Delta (\mathcal{L})$ implies that
there exists an $\eta >0$ so that for every
pair of distinct points $p,q\in \overline{L_x}\cap \Delta (\mathcal{L})$,
the distance between $p$ and $q$ is at least $\eta$.  Fix $p\in \overline{L_x}\cap \Delta (\mathcal{L})$ and take
a point $q\in \overline{L_x}\cap \Delta (\mathcal{L})$ closest to $p$. Hence, $\Delta (\mathcal{L})$ only intersects
the closed segment $[p,q]=\{ tp+(1-t)q\ | \ t\in [0,1]\} $ at the extrema $p,q$.
Using the plane $\overline{L_x}$ as a guide, one can produce homotopically nontrivial
simple closed curves $\g _n$ on the approximating surfaces $\lambda_nM_n$
which converge with multiplicity 2 outside $p,q$ to the segment $[p,q]$
as $n\to \infty $ (see for example the discussion just after Remark~2 in~\cite{mpr3}).
Our injectivity radius assumption implies that the length of the $\g _n$ is greater 
than or equal to 2, which
implies after taking $n\to \infty $ that the distance
between $p$ and $q$ is at least $1$. This finishes the proof of the lemma.
\end{proof}

\begin{definition}
\label{def4.12}
{\rm
Given $p\in \Delta (\mathcal{L})$, we assign an {\it orientation number} $n(p)=\pm 1$
according to the following procedure.
Outside a solid double vertical cone $\mathcal{C}_p\subset \B (p,\ve )$ based at $p$,
there exists a pair of multivalued graphs
contained in the approximating surface $\lambda_nM_n$ for $n$ large,
which, after choosing a subsequence, are both right or left handed
for $n$ sufficiently large (depending on $p$).  Assign a number
$n(p)= \pm 1$ depending on whether these multivalued graphs in $\lambda_nM_n$ occurring
nearby $p$ are right $(+)$ or left handed $(-)$.

We also define, given a plane $P\in \mathcal{P}'$,
\[
|I|(P)= \sum_{p \in P \cap \Delta(\mathcal{L})} |n(p)|=\mbox{ Cardinality}(P\cap \Delta (\mathcal{L}))\in \N \cup \{ \infty \} .
\]
By Lemma~\ref{lemma4.7}, the set $P \cap \Delta(\mathcal{L})$ is a closed, discrete countable set. After enumerating
this set and applying a diagonal argument, it is possible to choose a subsequence of the $\lambda_nM_n$ so that
locally around each point $p\in P \cap \Delta(\mathcal{L})$, there exists $n(p)\in \N$ such that $\lambda_nM_n$ contains
a pair of multivalued graphs around $p$ with a fixed handedness for all $n\geq n(p)$. This allows us to define
consistently the number
\[
I(P) = \sum_{p \in P \cap \Delta(\mathcal{L})}n(p)\in \Z,\quad \mbox{provided that $P \cap \Delta(\mathcal{L})$ is finite.}
\]
}
\end{definition}

Our next step will be to study the local constancy of the quantities $|I|(P)$ and
$I(P)$ when we vary the plane $P\in \mathcal{P}'$. Recall that the sequence
$\{ \lambda_nM_n\} _n$ is uniformly locally simply connected close, see Property (ULSC).
\begin{lemma}
\label{lemma4.9}
Given $P\in \mathcal{P}'$, there exists $\mu _0=\mu _0(P)$ such that if
$x\in P\cap \Delta (\mathcal{L} )$, then
for every $P'\in \mathcal{P}'$ with $|x_3(P)-x_3(P')|<\mu _0$
there exists a unique point $x'\in P'\cap \Delta (\mathcal{L} )\cap
\B (x,\ve )$ (this number $\ve >0$ was defined in Description (D1)-(D2)),
and the handedness of the two multivalued graphs
occurring in (a subsequence of) the $\lambda_nM_n$ nearby $x$ coincides with the
handedness of the two multivalued graphs occurring in $\lambda_nM_n$
nearby $x'$ (note that in particular, we do not need
to change the subsequence of the $\lambda_nM_n$ to produce a well-defined handedness at $x'$).
In particular, if $|I|(P)<\infty $, then $|I|(P')=|I|(P)$ and $I(P')=I(P)$.
\end{lemma}
\begin{proof}
We will start by proving the following simplified version
of the first sentence in the statement of the lemma.
\begin{claim}
Suppose $P=\{ x_3=0\} \in \mathcal{P}'$ and $x=\vec{0}\in P\cap \Delta (\mathcal{L})$. Then, there exists
$\mu _0>0$ such that if $P'\in \mathcal{P}'$ and $|x_3(P')|<\mu _0$, then there exists a unique point
$x'\in P'\cap \Delta (\mathcal{L} )\cap \B (\ve )$.
\end{claim}
{\it Proof of the Claim.}
Arguing by contradiction,  assume that there exists a sequence $P(m)$ of
horizontal planes in $\mathcal{P}'$ of heights $x_{3,m}$ converging to $0$ as $m\to \infty $,
so that for each $m\in \N$, the open disk $P(m)\cap \B (\ve )$ does not contain any point in $\Delta (\mathcal{L})$.
After passing to a subsequence, we can assume that the $P(m)$ converge to $P$ on one of its sides, say from above.
By definition of $\mathcal{P}'$, every plane $P(m)$ is the closure in $\R^3$ of a limit leaf $L_m$ of $\mathcal{L}$.
As for $m$ fixed the intersection $P(m)\cap \B (\ve )\cap \Delta (\mathcal{L})$ is empty,
then the convergence of portions of the surfaces
$\lambda_nM_n$ to $P(m)\cap \B (\ve )=L_m\cap \B (\ve )$ as $n\to \infty $
is smooth; in particular, for $n,m$ large
(but $m$ fixed), $(\lambda_nM_n)\cap \overline{\B }(\ve /2)$
contains a component $\Omega _1(m,n)$ which is a compact,
almost-horizontal disk of height arbitrarily close to $x_{3,m}$,
with $\partial \Omega _1(m,n)\subset \partial \B (\ve/2)$.
On the other hand, as $\vec{0}\in \Delta (\mathcal{L} )$ then
there exists a sequence of points $y_n\in \lambda_nM_n$ converging to the origin where the
absolute Gaussian curvature of $\lambda_nM_n$ tends to infinity. Let $\Omega _2(n)$ be the
component of $(\lambda_nM_n)\cap \overline{\B }(\ve /2)$ that contains $y_n$.
In particular, $\Omega _2(n)\cap \Omega _1(m,n)=\mbox{\O }$ for $m$ fixed and large,
and for all $n$ sufficiently large
depending on $m$. Note that as $\Omega _2(n)$ is topologically a disk and
Theorem~\ref{thm2.2} ensures that $\Omega _2(n)$ is compact
with boundary contained in $\partial \B (\ve /2)$.
As $x_{3,m}\to 0$ but $\ve $ is fixed, then the one-sided curvature
estimates for disks in~\cite{cm23} (see also
Theorem~7 in~\cite{mr13}) applied to $\Omega _1(m,n)$, $\Omega _2(n)$
gives a contradiction if $m,n$ are large enough. This contradiction proves the claim, as
the uniqueness of the point $x'\in P'\cap \Delta (\mathcal{L} )\cap \B (\ve )$
in the last part of the statement of the claim
follows directly from Lemma~\ref{lemma4.7}.
Now the claim is proved.

We next continue the proof of Lemma~\ref{lemma4.9}.
The existence part of the first sentence in the statement of the lemma
can be deduced from similar arguments as those in the proof of the last claim; the only difference
is that in the argument by contradiction, the plane $P$ cannot be assumed to be $\{ x_3=0\} $
but instead one assumes that there exist sequences $\{ P_m\} _m,\{ P'_m\} _m\subset \mathcal{P}'$,
$\{ \mu _m\} _m\subset \R^+$ with $\mu _m\searrow 0$ and $x_m\in P_m\cap \Delta (\mathcal{L})$
so that $|x_3(P_m)-x_3(P'_m)|<\mu _m$ and $P'_m\cap \Delta (\mathcal{L})\cap \B (x_m,\ve )=\mbox{\O }$.
The desired contradiction also appears in this case from the one-sided curvature estimates for embedded
minimal disks and we leave the details to the reader.
This finishes the proof of the first sentence in the statement
of the lemma.

It remains to show that if $|I|(P)<\infty $, then we can choose $\mu _0>0$ sufficiently small so that
$|I|(P')<\infty $ and $|I|(P)=|I|(P')$ and $I(P)=I(P')$ for
every $P'\in \mathcal{P}'$ with $|x_3(P)-x_3(P')|<\mu _0$.
As before, take $x\in P\cap \Delta (\mathcal{L})$. If $x\in S(\mathcal{L})$,
then Lemma~\ref{lemma4.5} and the description in (D1) of $S(\mathcal{L})$
around $x$ imply that after possibly choosing a smaller $\mu _0>0$,
$\mathcal{L}\cap \B (x,\mu _0)$ consists of a foliation by horizontal flat disks,
and $S(\mathcal{L})\cap \B(x,\mu _0)$ consists of a Lipschitz curve passing
through $x$ which is transverse to this
local foliation by flat disks. By the main theorem in Meeks~\cite{me25}, this Lipschitz curve is
in fact a vertical segment with $x$ in its interior. In particular,
every horizontal plane $P'$ that intersects
$\B (x,\mu _0)$ also intersects $S(\mathcal{L})\cap \B(x,\mu _0)$ at exactly one point $x'$,
and the handedness of the two multivalued graphs occurring in $\lambda_nM_n$ nearby $x$, $x'$ coincide.

If $x\in \mathcal{S}$, then similar arguments as in the last paragraph give the
same conclusion, following the local description in (D2) of $\mathcal{L} \cap \B (x,\mu _0)$
together with the uniform locally simply connected property of $\{
\lambda_nM_n\} _n$. This completes the proof of Lemma~\ref{lemma4.9}.
\end{proof}

\begin{lemma}
\label{lemma4.10}
For a plane $P\in \mathcal{P}'$, the following properties are equivalent:
\begin{enumerate}
\item $\mathcal{L} $ does not restrict to a foliation in any $\mu $-neighborhood of $P$, $\mu >0$.
\item $\Delta (\mathcal{L})\cap P=\mathcal{S}\cap P$.
\item $\mathcal{S}\cap P \neq \mbox{\rm \O }$.

\end{enumerate}
\end{lemma}
\begin{proof}
That statement 1 implies 2 follows from the observation that
if there exists a point $p$ in $S(\mathcal{L})\cap P$, then
there exists a small cylindrical neighborhood of $p$ which is entirely foliated by horizontal
disks contained in planes of $\mathcal{P}$, which in turn implies that
there exists a slab neighborhood of $P$ which is
foliated by planes in $\mathcal{L}$, a contradiction. Statement 2 implies 3 by definition of $\mathcal{P}'$.
Finally, the description in (D2-A), (D2-B) give that item~3 implies item~1.
\end{proof}

\begin{proposition}
\label{propos4.11}
Let $x$ be a point in
$\Delta (\mathcal{L})$ and let $P_x\in \mathcal{P}'$ be the horizontal plane that passes through $x$. If
$|I|(P_x)<\infty $ and  $I(P_x)=0$, then $\mathcal{L}$ is a foliation of $\R^3$ by horizontal planes.
\end{proposition}
\begin{proof}
Consider the largest closed horizontal slab $W$ containing $P_x$, so that
every plane $P$ in $W$ lies in $\mathcal{P}'$ and satisfies $|I|(P)=|I|(P_x)$ and $I(P)=I(P_x)$
(we allow $W$
to be just $P_x$, to be $\R^3$ or a closed halfspace of $\R^3$). If $W=\R^3$
then the proposition is proved. Arguing by contradiction, suppose $W$ has a boundary plane,
which we relabel as $P_x$ (since it has the same numbers $|I|(P_x)$, $I(P_x)$
as the original $P_x$ by Lemma~\ref{lemma4.9}). Without loss of generality,
we will assume $P_x$ is the upper boundary plane of $W$. Since $\R^3-\cup _{P'\in \mathcal{P}'}P'$
is a nonempty union of open slabs or halfspaces, then one of the following possibilities occurs:
\begin{enumerate}[(F1)]
\item $P_x$ is not a limit of planes in $\mathcal{P}'$ from above.
 In this case, $P_x$ is
the boundary of an open slab or halfspace in $\R^3-\cup _{P'\in \mathcal{P}'}P'$.
\item $P_x$ is the limit of a sequence of planes $P_m\in \mathcal{P}'$ from
above,
such that
for every $m$, $P_{2m}\cup P_{2m+1}$ is the boundary of an open slab component of
$\R^3-\cup _{P'\in \mathcal{P}'}P'$. Thus, Lemma~\ref{lemma4.9} implies that for $m$ sufficiently large,
$P_m$ satisfies $|I|(P_m)=|I|(P_x)$ and $I(P_m)=I(P_x)$.
\end{enumerate}
In either of the cases (F1), (F2), we can relabel $P_x$ so that $P_x$ is the bottom boundary
plane of an open component of $\R^3-\cup _{P'\in \mathcal{P}'}P'$.
By Lemma~\ref{lemma4.10}, $\Delta (\mathcal{L})\cap P_x=\mathcal{S}\cap P_x$.

After the translation by vector $-x$, we can assume $x=\vec{0}$ and $P_x=\{ x_3=0\} $.
By the local description of $\mathcal{L}$ in (D2-A), there is a nonflat leaf $L$ of $\mathcal{L}$
which has the origin in its closure. Note that this implies that $P_x$ is contained in the limit
set of $L$. As by hypothesis $P_x\cap \Delta (\mathcal{L})$ is finite, we can choose a large round open disk $D\subset P_x$
such that $P_x\cap \Delta(\mathcal{L})\subset D$.  Let \[
A=\{ (x_1,x_2)\in \R ^2\ | \ (x_1,x_2,0)\notin D\} .
\]
Choose $\mu >0$ small enough so that
\begin{equation}
\label{eq:A1}
(A\times [0,\mu ])\cap \mathcal{P} '=A\times \{ 0\} .
\end{equation}
\begin{claim}
\label{claim4.17}
The surface $L^{\mu} =\mathcal{L} \cap \{ (x_1,x_2,x_3) \ | \  0 < x_3 \leq \mu\} $
is connected. Furthermore, $(A\times (0,\ve ])\cap \mathcal{L}=(A\times (0,\ve ])\cap L$.
\end{claim}
{\it Proof of the claim.}
The equality (\ref{eq:A1}) together with
a standard barrier argument (see the proof of Lemma~1.3 in~\cite{mr8})
imply
that if the claim fails, then there exists a connected, nonflat, stable minimal
surface $\S$ in $\{ (x_1,x_2,x_3) \ | \  0 < x_3 \leq \mu\}$, such that
its boundary $\partial \Sigma $ is contained in $\mathcal{L}\cap \{x_3=\mu \}$,
$\Sigma $ is proper in $\{ (x_1,x_2,x_3) \ | \  0 < x_3 \leq \mu\}$, $\Sigma $ is not
contained in $\{ x_3=\mu \} $ and $\Sigma $
is complete outside the finite set $\mathcal{S}\cap \{x_3=0\}$
in the sense that any proper divergent arc $\a$ in $\S$ of finite length must have its
divergent end point contained  in $\mathcal{S}\cap \{x_3=0\}$.
Since $[ \Sigma \cap \{ (x_1,x_2,x_3) \ | \  0\leq x_3 < \mu\} ]\cup (P_x-\mathcal{S})$
is a minimal lamination of $\{ (x_1,x_2,x_3) \ | \  {-1}< x_3 < \mu \}$ outside of a finite
set of points and this lamination consists of stable leaves,
then item~1 of Corollary~7.1 in~\cite{mpr10} implies that
this lamination extends to a minimal lamination of $\{ (x_1,x_2,x_3) \ | \  -1< x_3 < \mu \}$,
which contradicts
the maximum principle for minimal surfaces at the origin. This finishes the proof of the claim.
\par
\vspace{.2cm}
We next continue the proof of Proposition~\ref{propos4.11}.
Clearly, we may also assume
\[
\partial L^{\mu }\cap \{ x_3=\mu \} \neq \mbox{\O }
\]
 by taking $\mu >0$
smaller. As $L^{\mu }$ is not flat and the injectivity
radius function of the surfaces $\lambda_nM_n$ is bounded from
below by 1/2 away from their diverging boundaries, then we can apply the
intrinsic version of the
one-sided curvature estimates for minimal disks in~\cite{cm23}
to $\lambda_nM_n$ to conclude that the norm of the second fundamental form  of
the possibly disconnected surfaces $(A\times [0,\mu ])\cap (\lambda_nM_n)$ is arbitrarily small if we
take $\mu $ sufficiently small. Thus by choosing $\mu >0$ small enough, it follows that each component
$G$ of $(A\times [0,\mu ]) \cap L$ is locally a graph over its projection to $A\times \{ 0\} $,
with boundary $\partial G$ contained in
$((\partial A )\times [0,\mu ]) \cup (A\times \{ \mu \})$.

Note that so far in the proof of the proposition we have only used that $|I|(P_x)<\infty $.
Next we will use the second hypothesis $I(P_x)=0$ to obtain the desired contradiction,
which will follow from the invariance of flux for $\nabla (x_3|_L)$.

Since $I(P_x)=0$, then $P_x \cap \Delta(\mathcal{L}) = P_x \cap
\mathcal{S}$ consists of an even positive number of points $p_1, q_1, p_2, q_2, \ldots, p_d, q_d$
with $n(p_i)=-n(q_i)$ for each $i=1,\ldots ,d$.  Consider a
collection $\{ \delta_1, \ldots, \delta_d\}$ of pairwise disjoint
embedded planar arcs in $P_x$ such that the end points of $\delta_i$ are $p_i, q_i$.
For $i\in \{ 1,\ldots ,d\} $ fixed, construct a sequence $\{ \g _i(m)\} _m$ of
connection loops in $L^{\mu }$, as indicated in the proof of Lemma~\ref{lemma4.7}
(note that now the $\de _i$ are not necessarily
straight line segments), i.e.,
each $\g _i(m)$ consists of two lifts of $\delta_i-[\B (p_i,\ve _i(m))\cup \B (q_i,\ve _i(m))]$
to adjacent sheets of $L^{\mu }$ over $\delta_i$ joined by short arcs of length at most $3\varepsilon_i(m)>0$
near $p_i$ and $q_i$, such that the $\g _i(m)$ converge as $m\to \infty $ with multiplicity 2
to $\delta_i$ and $\varepsilon_i(m) \to 0$ as $m \rightarrow \infty$.
It is possible to choose the indexing
of these curves $\g _i(m)$ so that for every $m\in \N$, the collection
$\Gamma_m = \{ \gamma_1(m), \gamma_2(m),
\ldots, \gamma_d(m) \}$ separates the connected surface
$L^{\mu}$ into two components (this property holds because
the portion of $L^{\mu }$ sufficiently close to $P_x$ is topologically equivalent to the intersection
of a periodic parking garage surface with a closed lower halfspace, hence it suffices to choose
all curves in the collection $\G_m$ corresponding to closed curves at the same level in the parking garage surface
(for instance, in the particular case $d=1$,
the surface $L^{\mu }$
is modeled by a suitable portion of a Riemann minimal example, and each of the
connection loops $\g_1(m)$ is a generator of the homology of the approximating Riemann
minimal example).

Recall that $\partial L^{\mu }\cap \{ x_3=\mu \} \neq \mbox{\O }$, and that $L^{\mu }$ is
separated by $\Gamma _m$ into two components; we will denote by $L^{\mu }_m$ the component of $L^{\mu }-\G _m$
whose nonempty boundary contains
$\partial L^{\mu }\cap \{ x_3=\mu \} $,
see Figure~\ref{fig2}.
We remark that $L_m^{\mu }$ lies above any horizontal plane $P'\subset \{ x_3>0\} $ which is strictly below
$\partial L_m^{\mu }$ (since otherwise we would contradict that the portion of $L$ below
$P'$ is connected by Claim~\ref{claim4.17},
as $L-L_m^{\mu }$ intersects the open slab bounded by $P_x$ and $P'$);
in particular, equation~(\ref{eq:A1}) ensures that
$L_m^{\mu }$ is properly embedded in $x_3^{-1}([0,\mu ])$; this is in contrast
with $L-L_m^{\mu }$, which is nonproper and limits to $P_x$.
\begin{figure}
\begin{center}
\includegraphics[width=11.4cm]{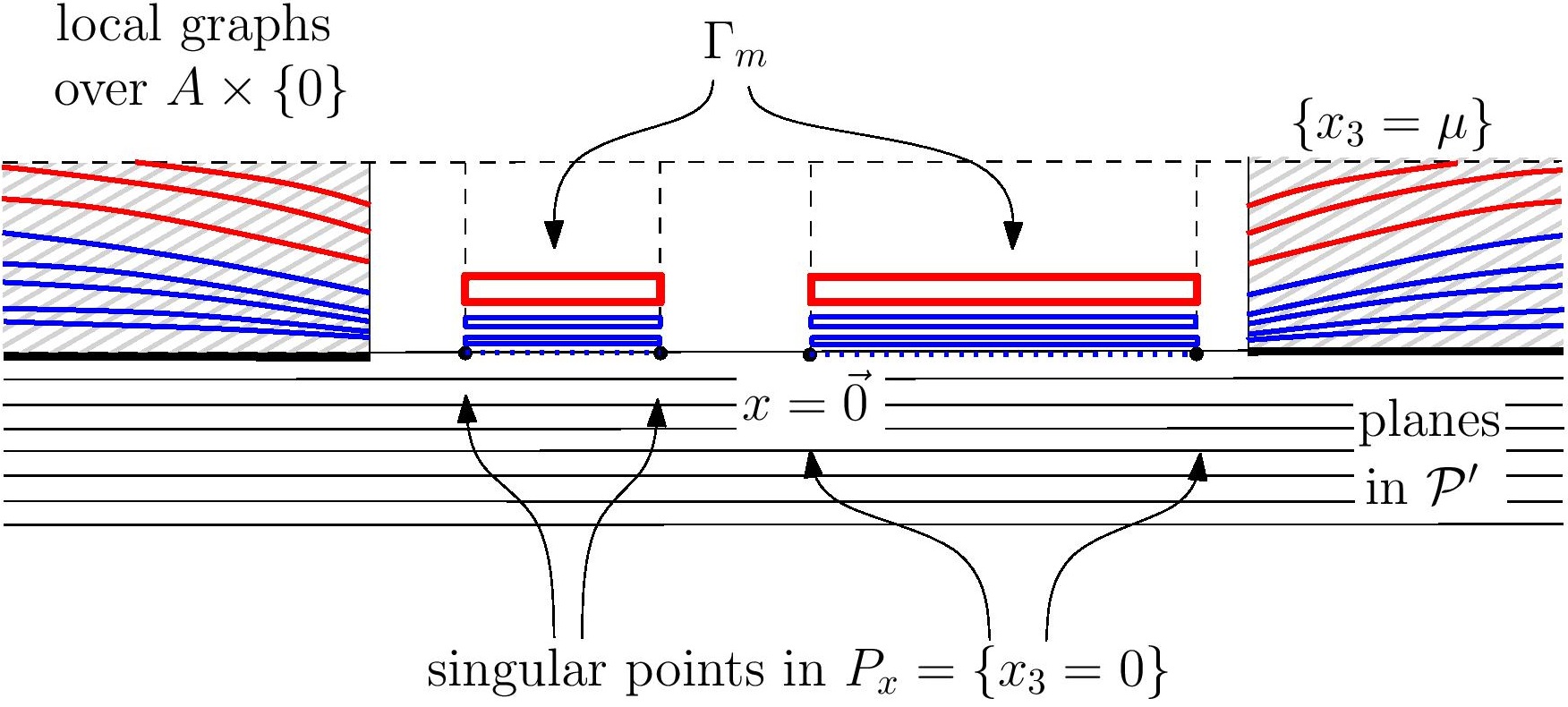}
\caption{The collection of connection loops $\G _m$ disconnects $L^{\mu }$ into two components,
one of which, denoted by $L_m^{\mu }$,
is proper in $x_3^{-1}([0,\mu ])$ (we have sketched $L_m^{\mu }$ in  red color).
Note that the boundary of $L^{\mu }_m$ contains curves lying in $\{ x_3=\mu \} $, and that
$\partial L^{\mu }_m\cap \{ x_3=\mu \} \subset
\partial L^{\mu }_{m+1}\cap \{ x_3=\mu \} $ for all $m\in \N$.}
\label{fig2}
\end{center}
\end{figure}

As $L_m^{\mu }$ is a properly immersed minimal surface with boundary (in fact, embedded) in a halfspace,
then Theorem~3.1 in Collin, Kusner, Meeks
and Rosenberg~\cite{ckmr1} implies that $L_m^{\mu }$ is a parabolic surface,
in the sense that the harmonic measure on its boundary is full.
Now consider
the scalar flux of a smooth tangent vector field $X$ to $L_m^{\mu }$
across a finite collection of compact curves or arcs
$\a \subset\partial L_m^{\mu }$, defined as
\[
F(X,\a)=\int _{\a}\langle X,\eta \rangle ,
 \]
 where $\eta $ is the exterior unit conormal vector to $L_m^{\mu }$ along $\a $.
Pick a compact arc $\sigma\subset\partial L_1^{\mu } \cap \{x_3 = \mu \}$. Since
$\sigma \subset \partial L^{\mu }_m\cap \{ x_3=\mu \} $ for all $m\in \N$, then we conclude that
\begin{equation}
\label{flux1}
F(\nabla x_3,\sigma )\leq F(\nabla x_3 ,\partial L^{\mu }_m\cap \{ x_3=\mu \} ),
\end{equation}
where the right-hand-side of (\ref{flux1}) must be understood as the limit of $F(\nabla x_3,\a )$ where $\a $ runs
along an increasing exhaustion of  $\partial L^{\mu }_m\cap \{ x_3=\mu \} $ by compact curves or arcs. Also
note that the left-hand-side of (\ref{flux1}) is a positive number by the maximum
principle, which does not depend on $m$.

For each $m\in \N$ fixed, let $\{ L_{m,k}^{\mu }\ | \ k\in \N ,k\geq k_0(m)\} $ be
a smooth increasing compact exhaustion of $L_m^{\mu }$
such that each $L_{m,k}^{\mu }$ contains all points of distance at most $k$
from some previously chosen point in $L_m^{\mu }$
and such that $\sigma \cup \G_m \subset \partial L_{m,k_0(m)}^{\mu }$
(hence $\sigma \cup \G_m\subset \partial L_{m,k}^{\mu }$
for all $k\geq k_0(m)$).  The boundary of $L_{m,k}^{\mu }$ is the disjoint union
of the following three pieces:
\[
\partial L_{m,k}^{\mu }=\partial_{\mu }(m,k)\cup \G_m \cup \partial_{*}(m,k),
\]
where
\[
\begin{array}{rcl}
\partial_{\mu }(m,k) & = & \partial L_{m,k}^{\mu }\cap \{x_3=\mu \} ,
\\
\partial_{*}(m,k) & = & \partial L_{m,k}^{\mu }-[\partial_{\mu }(m,k)\cup \G_m].
\end{array}
\]
Consider the nonnegative harmonic function $u_m=\mu-x_3$ on the surface
$L_m^{\mu }$. Let $u_{m,k}$ be the harmonic function on $L_{m,k}^{\mu }$
defined by the boundary values $u_{m,k}=0$ on $\partial _{*}(m,k)$,
$u_{m,k}=u_m$ on $\partial L_{m,k}^{\mu }-\partial _{*}(m,k)$.
Using that $u_{m,k}=0$ along $\partial _{*}(m,k)$,
$u_m=u_{m,k}=0$ along $\partial_{\mu }(m,k)$ and the double Green's formula, we have
\[
F(u_{m,k}\nabla u_m,\G _m)=F(u_{m,k}\nabla u_m,\partial L_{m,k}^{\mu })
=F(u_{m}\nabla u_{m,k},\partial L_{m,k}^{\mu })
\]
\begin{equation}
\label{eq:aa}
=F(u_{m}\nabla u_{m,k},\G_m)+F(u_{m}\nabla u_{m,k},\partial _{*}(m,k)).
\end{equation}
Also note that
\[
F(u_m\nabla u_{m,k},\G_m)=F((u_m-\mu )\nabla u_{m,k},\G_m)+\mu F(\nabla u_{m,k},\G_m)
\]
\begin{equation}
\label{eq:ab}
\stackrel{(\star)}{=}F((u_m-\mu )\nabla u_{m,k},\G_m)-\mu F(\nabla u_{m,k},\partial _{\mu }(m,k))
-\mu F(\nabla u_{m,k},\partial _{*}(m,k)),
\end{equation}
where in $(\star)$ we have used the Divergence Theorem applied to $u_{m,k}$ in $L_{m,k}^{\mu }$.
From (\ref{eq:aa}) and
(\ref{eq:ab}) we deduce that
\[
F(u_{m,k}\nabla u_m,\G _m)=
\]
\begin{equation}
\label{eq:ac}
F((u_m-\mu )\nabla u_{m,k},\G_m)-\mu F(\nabla u_{m,k},\partial _{\mu }(m,k))
+F((u_m-\mu )\nabla u_{m,k},\partial _{*}(m,k)).
\end{equation}

Equation (\ref{eq:ac}) leads to a contradiction, as the following properties hold:
\begin{enumerate}[(G1)]
\item The left-hand-side of (\ref{eq:ac}) tends to zero as $m,k\to \infty $.
\item The first term in the right-hand-side of
(\ref{eq:ac}) tends to zero as $m,k\to \infty $.
\item The second term in the right-hand-side of
(\ref{eq:ac}) is at least $\frac{1}{2}\mu \, F(\nabla x_3,\sigma )>0$ for $k$ large.
\item The third term in the right-hand-side of
(\ref{eq:ac}) is nonnegative.
\end{enumerate}

We next prove (G1)-(G4).
As $u_m-\mu =-x_3\leq 0$ and $\langle \nabla u_{m,k},\eta \rangle \leq 0$
along $\partial _{*}(m,k)$
(this last inequality follows from the facts that $u_{m,k}\geq 0$ in
$L^{\mu }_{m,k}$, $u_{m,k}=0$ along $\partial _{\mu }(m,k)$ and $\eta $ is
exterior to $L_m^{\mu }$ along its boundary), then (G4) holds.
Similarly, the fact that $\langle \nabla u_{m,k},\eta \rangle \leq 0$ along $\partial _{\mu }(m,k)$ implies
that $F(\nabla u_{m,k},\partial _{\mu }(m,k))\leq F(\nabla u_{m,k},\sigma )$.
As $L^{\mu }_m$ is a parabolic
surface, then the functions $u_{m,k}$ converge as $k\to \infty $ uniformly
on compact subsets of $L^{\mu }_m$ to the bounded
harmonic function $u_m$ (hence their gradients converge as well). This implies
that $F(\nabla u_{m,k},\sigma )$ converges as $k\to \infty $ to $F(\nabla u_{m},\sigma )=
-F(\nabla x_3,\sigma )$, from where (G3) follows.
Property (G2) also holds because
\begin{enumerate}[(G2.a)]
\item $\lim _{k\to \infty }| (\nabla u_{m,k})|_{\G _m}| =| (\nabla u_m)|_{\G _m}| \leq 1$,
\item Length$(\G _m)$ is bounded independently of $m$, and
\item $u_m-\mu =-x_3$ is arbitrarily small along $\G _m$ for $m$ large,
\end{enumerate}
Finally, (G1) holds because
\[
F(u_{m,k}\nabla u_m,\G _m)\stackrel{(k\to \infty )}{\longrightarrow }F(u_m\nabla u_m,\G _m)
\stackrel{(m\to \infty )}{\longrightarrow }
-\mu \lim _{m\to \infty }F(\nabla x_3,\G _m)=0,
\]
where in the last equality we have used that the tangent plane to $L^{\mu }$
becomes arbitrarily horizontal along $\G _m$ except along
$2d$ subarcs of $\G _m$ whose total length goes to zero as $m\to \infty $.
Now Proposition~\ref{propos4.11} is proved.
\end{proof}

As announced above, we next show that if $\mathcal{S}=\mbox{\O }$, then item~5
of Theorem~\ref{tthm3introd} holds. In fact, we will obtain a more detailed
description in the following result.
\begin{proposition}
\label{ass4.17}
If $\mathcal{L}$ is a regular lamination of $\R^3$, then $\mathcal{L}$ is a foliation
of $\R^3$ by parallel planes and  item~5
of Theorem~\ref{tthm3introd} holds.
Furthermore:
\begin{enumerate}[(A)]
\item $S(\mathcal{L})$ contains a line $l_1$  which
passes through the closed ball of radius 1 centered at the origin, and another line $l_2$ at
distance one from $l_1$, and all of the lines in $S(\mathcal{L})$
have distance at least one from each other.
\item There exist oriented closed geodesics  $\g_n\subset \lambda_nM_n$
with lengths converging to~$2$, which converge to the line segment $\g$ that joins
$(l_1\cup l_2)\cap \{x_3=0\}$ and such that the integrals of the unit conormal vector
of $\lambda_nM_n$ along $\g_n$ in the induced exponential $\rth$-coordinates of $\lambda_nB_N(p_n,\ve _n)$
converge to a horizontal vector of length $2$ orthogonal to $\g$.
\end{enumerate}
\end{proposition}
\begin{proof}
Since we are assuming $\mathcal{S}=\O$ but the uniformly bounded Gaussian
curvature hypothesis in Proposition~\ref{propos4.2} fails to hold,
then $\Delta (\mathcal{L})=S(\mathcal{L})\neq \mbox{\O}$
and Lemma~\ref{lemma4.10} implies that the Lipschitz curves in $S(\mathcal{L})$
go from $-\infty $ to $+\infty $ in height and thus, $\mathcal{L}$ is a foliation
of $\R^3$ by planes.
By the $C^{1,1}$-regularity theorem for $S(\mathcal{L})$
in~\cite{me25}, $S(\mathcal{L})$ consists of vertical lines, precisely one
passing through each point in $P_z \cap S(\mathcal{L})$ (here $z$ is
any point in $S(\mathcal{L})$).
The surfaces $\lambda_nM_n$ are now seen to converge on
compact subsets of $\rth-S(\mathcal{L})$ to the minimal parking garage
structure on $\rth$ consisting of horizontal planes, with
vertical columns over the points $y \in P_z \cap S(\mathcal{L})$ and with
orientation numbers $n(y)= \pm 1$ in the sense of Definition~\ref{def4.12}.

We next show that item (A) of the proposition holds.
Note that since the geodesic loops $\be_n$ given in Assertion~\ref{beta-n} all pass through $p_n$,
then the limit set Lim$(\{ \be _n\} _n)$ of the $\be _n$ contains the origin in $\rth$.
Since the surfaces  $\lambda_n M_n$ converge to the foliation of $\rth$ by horizontal planes and
this convergence is $C^\a$ (actually $C^{\infty }$ tangentially to the leaves of the limit
foliation) away from $S(\mathcal{L})$, for any $\a \in (0,1)$, then  Lim$(\{ \be _n\} _n)$
consists of a connected, simplicial complex consisting of finitely many horizontal
segments joined by vertical segments contained in lines of $S(\mathcal{L})$; this
complex could have dimension zero, in which case it reduces to the origin.
Clearly there exists a point $q\in S(\mathcal{L})\cap \mbox{ Lim}(\{ \be _n\} _n)$,
since otherwise Lim$(\{ \be _n\} _n)$ contains a horizontal segment $l$ with
$\vec{0}\in l$, and in this case one of the end points of $l$ lie in $S(\mathcal{L})$. As
$q\in S(\mathcal{L})$, then the vertical line $l_q$ passing through $q$ lies in $S(\mathcal{L})$.

We claim that there exists a vertical line $l$ contained in $S(\mathcal{L})$ at
distance one from $l_q$. As $S(\mathcal{L})$ is a closed set, Lemma~\ref{lemma4.7}
implies that if our claim fails to hold
then there exists $\eta >0$ such that the vertical cylinder of radius $1+2\eta $
with axis $l_q$ only intersects $S(\mathcal{L})$ at $l_q$. Consider a sequence of points
$q_n\in \be _n$ limiting to $q$ and consider the related extrinsic balls
$B_{\lambda_nN}(q_n,1+\eta )$. By the triangle inequality, the loop $\be _n$ is contained
in $(\lambda_nM_n)\cap B_{\lambda_nN}(q_n,1+\eta )$ (as the length of $\be _n$ is two and
the intrinsic distance between any two points in $\be _n$ is at most one).
By the parking garage structure of the limit foliation, for $n$ large each of the surfaces
$(\lambda_nM_n)\cap B_{\lambda_nN}(q_n,1+\eta )$ contains a unique main component $\Delta _n$ which is
topologically a disk (this is the component that contains $q_n$). Therefore, $\be _n\subset
\Delta _n$, which implies that $\be _n$ is homotopically trivial. This contradiction proves
our claim. Note that these arguments also imply that $q$ lies in the closed ball of radius
1 centered at the origin. This claim together with Lemma~\ref{lemma4.7} imply that
item~(A) of the proposition holds.
Also observe that we have completed the proof of the first statement of
Theorem~\ref{tthm3introd} (the $\lambda_nM_n$ converge to a minimal
parking garage structure of $\R^3$
with at leat two vertical lines in $S(\mathcal{L})$).

Since $l_q$ and $l$ are at distance 1 apart, then
there exist connection loops $\g _n$ on $\lambda_nM_n$ of lengths converging to 2
which converge as $n\to \infty $ to a horizontal line segment of length 1
joining $l_q$ and~$l$. These connection loops
satisfy the properties in
item~(B) of the proposition
(see~\cite{mpr3} for details).

To complete the proof of Proposition~\ref{ass4.17}, it only remains to demonstrate
the last sentence of item~5 of Theorem~\ref{tthm3introd} assuming that there exists
a bound on the genus of the surfaces $\lambda_nM_n$. This
 follows from similar arguments as those in the proof of Lemma~\ref{ass2.1}, since
the surfaces $\lambda_nM_n$ approximate arbitrarily well the behavior of a
periodic parking garage surface in $\R^3$.
This finishes the proof of Proposition~\ref{ass4.17}.
\end{proof}

By Proposition~\ref{ass4.17}, to finish the proof of Theorem~\ref{tthm3introd} it remains to
demonstrate that item~6 holds provided that $\mathcal{S}\neq \mbox{\O }$, a hypothesis that
will be assumed for the remainder of this section. We will start
by stating a property to be used later. Recall that given $R>0$ and $n\in \N$ sufficiently large,
 $\Sigma (n,R)$ denotes the closure of the component of $[\lambda_n\overline{B}_M(p_n,\frac{\sqrt{n}}{2\lambda_n})]
 \cap B_{\lambda_nN}(p_n,R)$
that contains $p_n$, and that the surface $M_k$ that appears below was defined in equation (\ref{eq:Mk})
as a rescaling by
$1/\lambda_k=1/\lambda'_{n(k)}$
 of $\Sigma (n(k),k)$, where $n(k)$ was defined
in the proof of Lemma~\ref{lemma4.3}.
The purpose of the next result is to show that given a radius $k_0$,
all the components of $\lambda_kM_k$ in an extrinsic ball of that radius centered at the origin
(for $k$ sufficiently large depending on $k_0$)  can be joined within an extrinsic
ball of a larger radius $m(k_0)$ independent of $k$, see Figure~\ref{figprop18}.
\begin{figure}
\begin{center}
\includegraphics[width=6cm]{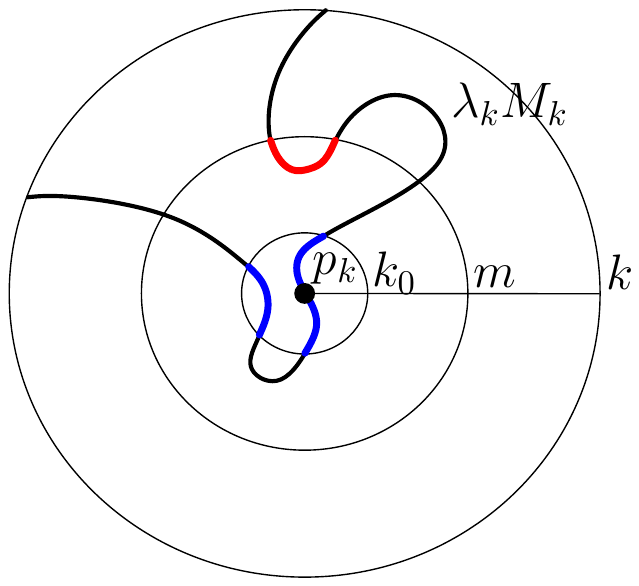}
\caption{The red portion of $\lambda_kM_k$ cannot enter into the ball of radius $k_0$,
since it does not join to the blue portions within the ball of radius $m$.}
\label{figprop18}
\end{center}
\end{figure}

\begin{proposition}
\label{propos4.18}
  Given $k_0\in \N$ there exists $m=m(k_0)\in \N$ such that for any $k\in \N$ sufficiently large, we have
\begin{equation}
\label{eq:4.5A}
(\lambda_{k}M_{k})\cap B_{\lambda_{k}N}(p_{k},k_0)\subset \Sigma (n(k),m).
\end{equation}
Furthermore, the intrinsic distance in $\lambda_kM_k=\Sigma (n(k),k)$ from any point in
$(\lambda_{k}M_{k})\cap B_{\lambda_{k}N}(p_{k},k_0)$ to $p_k$ is not greater than some number
independent on $k$, for all such $k$.
\end{proposition}
\begin{proof}
To prove (\ref{eq:4.5A}) we argue by contradiction. Suppose that for some $k_0\in \N$,
there exist sequences $\{ m_k\} _k, \{ a_k\} _k\subset \N$ so that $m_k\nearrow \infty $,
$m_k<a_k$ and for all $k$,
\begin{equation}
\label{fail}
\left( \lambda_{a_k}M_{a_k}\right)\cap B_{\lambda_{a_k}N}(p_{a_k},k_0)\not \subset \Sigma (n(a_k),m_k).
\end{equation}
Let $\Omega (k)$ be a component of $B_{\lambda_{a_k}N}(p_{a_k},m_k)-(\lambda_{a_k}M_{a_k})$
 that contains $\Sigma (n(a_k),m_k)$ in its boundary $\partial \Omega (k)$, and so that
$\partial \Omega (k)$ contains another component
$\Delta (k)$ different from $\Sigma (n(a_k),m_k)$,
with $\Delta (k)\cap B_{\lambda_{a_k}N}(p_{a_k},k_0)\neq \mbox{\O }$,
which can be done by (\ref{fail}).
Observe that the boundary of $\Omega (k)$
is a good barrier for solving Plateau problems in $\Omega (k)$ (here we are using that since
$\ve _k\to 0$, then the extrinsic balls $B_N(p_k,\ve _k)$ have
mean convex boundaries). Let $\Sigma '(k)$ be a surface of least area in $\Omega (k)$ homologous to
$\Sigma (n(a_k),m_k)$, and with $\partial \Sigma '(k)=\partial \Sigma (n(a_k),m_k)$.
As for all $k$ the surface $\Sigma '(k)$ intersects $B_{\lambda_{a_k}N}(p_{a_k},k_0)$,
then by uniform curvature estimates for stable
minimal surfaces away from their boundaries we conclude that after passing to a subsequence,
the $\Sigma '(k)$ converge as $k\to \infty $ to a nonempty (regular) minimal lamination
$\mathcal{L}'$ of $\R^3$ all whose leaves are complete,
embedded, stable minimal surfaces, and therefore all leaves of $\mathcal{L}'$ are planes.
To find the desired contradiction, note that the following properties hold.
\begin{enumerate}[(H1)]
\item The sequence $\{ \Sigma '(k)\} _k$ is locally simply connected in $\R^3$
(by uniform curvature estimates for stable minimal surfaces and by the uniform graph lemma,
see Colding and Minicozzi~\cite{cmCourant}
or P\'erez and Ros~\cite{pro2}), as well as the sequence $\{ \lambda_{a_k}M_{a_k}\} _k$.

\item For $k$ large, the surface $\lambda_{a_k}M_{a_k}$ is unstable, and thus $\Sigma'(k)$ and $\lambda_{a_k}M_{a_k}$
only intersect along $\partial \Sigma'(k)\subset \partial B_{\lambda_{a_k}N}(p_{a_k},m_k)$,
which diverges to $\infty $ as $k\to \infty $.
\end{enumerate}

By Property (H2), we deduce that the planes in $\mathcal{L}'$ are either disjoint from $\mathcal{L}$
or they are leaves of~$\mathcal{L}$. In fact, Lemma~\ref{lemma4.6} implies that
the first possibility cannot occur. Also note that as $\Sigma '(k)$ intersects $B_{\lambda_{a_k}N}(p_{a_k},k_0)$ for
all $k$, then $\mathcal{L}'$ contains a plane $\Pi $ that intersects the ball $\B (k_0)$.
Properties (H1), (H2) insure that we can apply Theorem~7 in~\cite{mr13}
(or the one-sided curvature estimates by Colding and Minicozzi~\cite{cm23})
to obtain uniform local curvature estimates for the surfaces $\lambda_{a_k}M_{a_k}$
in a fixed size neighborhood of $\Pi $.
As a consequence, $\Pi $ cannot lie in the collection $\mathcal{P}'$ defined in Lemma~\ref{lemma4.5} and thus,
Lemma~\ref{lemma4.6} gives a contradiction
(hence (\ref{eq:4.5A}) is proved).

As for the last sentence in the statement of Proposition~\ref{propos4.18}, take $r_{k_0}\in (0,\infty )$ such that
$m(k_0)=r_{k_0}\, \de (r_{k_0})$, where $\de (r)$ is the function that appears in the
first sentence of Theorem~\ref{tthm3introd}.
Applying Proposition~\ref{lemma4.2} to $R=R_1=r_{k_0}$ (recall that equation (\ref{eq:chordarc}) holds after
replacing $\wt{\de}(r)$ by $\de (r)$, see the paragraph just before Lemma~\ref{lemma4.3}),
for $k$ sufficiently large we have $\Sigma (k,m(k_0))\subset B_{\wt{M}(k)}(p_k,r_{k_0}/2)=B_{\lambda_kM_k}(p_k,r_{k_0}/2)$,
from where the last statement
of the proposition follows.
\end{proof}

\begin{remark}
  {\em
Proposition~\ref{propos4.18} still holds if the limit object of the surfaces $\l _nM_n$
is either a nonsimply connected, properly embedded minimal surface (case 4 of Theorem~\ref{tthm3introd})
or a minimal parking garage structure (case 5 of the theorem). To see why this generalization
holds, note that the arguments in the proof of
the last proposition still produce a plane $\Pi \subset \R^3$
which is either disjoint from the limit set of the sequence $\lambda_nM_n$, or it is contained in this limit set.
If case 4 of Theorem~\ref{tthm3introd} occurs for this sequence, then
then the halfspace theorem gives a contradiction. If
the $\lambda_nM_n$ converge to a minimal parking garage structure $\mathcal{L}$  of $\R^3$,
then it is clear that no plane in the complement of
$\mathcal{L}$ can exist.

Also note that the constant $m(k_0)$ in  Proposition~\ref{propos4.18} can be chosen so that it
does not depend on the homogeneously regular manifold or on the surface $M$ to which we apply it.
  }
\end{remark}

We next continue with the proof of item 6 of Theorem~\ref{tthm3introd}, provided that $\mathcal{S} \neq \mbox{\O }$.
Since the closures of the
set of flat leaves is $\mathcal{P}'$ and $\mathcal{P} '$ is a lamination of $\R^3$ with no singularities,
then there exists at least one leaf of $\mathcal{L}$ which is not flat, so the first statement of
item~6 holds. By Lemmas~\ref{lemma4.5} and~\ref{lemma4.6}, the sublamination $\mathcal{P}$ of flat leaves of $\mathcal{L}$ is nonempty,
the closure of every such planar leaf $L_1$ is a horizontal plane in the family $\mathcal{P} '$ defined in Lemma~\ref{lemma4.5}, and
hence by definition of $\mathcal{P} '$ we have that $\overline{L_1}$
intersects $\Delta (\mathcal{L})=\mathcal{S}\cup S(\mathcal{L})$.
By Lemma~\ref{lemma4.7}, the distance between any two points in $\overline{L_1}\cap \Delta (\mathcal{L})$ is at least 1. By
Lemma~\ref{lemma4.10}, $\overline{L_1}\cap \Delta (\mathcal{L})$ is either contained in $\mathcal{S}$ or in
$S(\mathcal{L})$. So in order to conclude the proof of item~6 
of Theorem~\ref{tthm3introd}
it remains to show the following property.
\begin{proposition}
\label{propos4.22}
Given of leaf $L_1$ of $\mathcal{P}$, the plane $\overline{L_1}$ intersects $\Delta (\mathcal{L})$ in at least two points.
\end{proposition}
We next prove the proposition by contradiction through a series of lemmas.
Suppose that $L_1\in \mathcal{P}$ satisfies that $\overline{L_1}\cap \Delta (\mathcal{L})$ consists
of a single point $x\in \Delta (\mathcal{L})$. If $x\in S(\mathcal{L})$, then Lemma~\ref{lemma4.10} implies that $\mathcal{L}$
restricts to a foliation of some $\ve $-neighborhood of $\overline{L_1}$.
Consider the largest open
horizontal slab or halfspace $W$ containing $\overline{L_1}$ so that $\mathcal{L}$ restricts to $W$ as a foliation by planes.
As $\mathcal{S}\neq \mbox{\O }$, then $W\neq \R^3$.
By Lemma~\ref{lemma4.9}, we can replace $L_1$ by a flat leaf in the boundary of $W$
and after this replacement, we have $|I|(\overline{L_1})=1$ and
$\overline{L_1}\cap \Delta (\mathcal{L})=\overline{L_1}\cap \mathcal{S}$. Without loss of generality,
we may assume that $\overline{L_1}$ is the top boundary plane of $W$. Arguing as in
the discussion of cases (F1) and (F2) in the proof of Proposition~\ref{propos4.11},
we can replace $\overline{L_1}$
by a bottom boundary plane $P$ of an open component $C$ of $\R^3-\cup _{P'\in \mathcal{P} '}P'$
so that the distance from
$\overline{L_1}$ to $P$ is less than the number $\mu _0
=\mu _0(\overline{L_1})$ given in Lemma~\ref{lemma4.9}
(note that $P$ may be equal
to $\overline{L_1}$).

Denote by $L$ the nonflat leaf of $\mathcal{L}$ directly above $P$. After the translation by $-x$,
we can assume that $P=\{ x_3=0\} $ and $x=\vec{0}\in \mathcal{S}$. We next analyze several aspects
of the geometry of $L$ in a neighborhood of $P$ in $C$.

Given a regular value $\mu \in (0,\ve)$ for $x_3$ restricted to $L$ (this number $\ve >0$
appears in description (D1)-(D2) above), let
\[
L^{\mu }:=L\cap \{ (x_1,x_2,x_3)\ | \ 0<x_3\leq \mu \} ,
\]
which is a connected surface with boundary for $\mu $ sufficiently small
by Claim~\ref{claim4.17}.

\begin{definition}
{\rm
Take a sequence of points $q_k\in L^{\mu }$ converging to $\vec{0}$
and numbers
$r_k\in (0,|q_k|/2]$. For each $k\in \N$,
consider the function $f_k\colon L^{\mu }\cap \B (q_k,r_k)\to \R $ given by
\begin{equation}
\label{eqw:fk}
 f_k(x)=\sqrt{|K_{L}|}(x)\cdot {\rm dist}_{\R^3}(x,\partial [L\cap \B(q_k,r_k)])
 \end{equation}
 Let $x_k\in \B(q_k,r_k)$ be a maximum of $f_k$
 (note that $f_k$ is continuous and vanishes at
$ \partial [L\cap B(q_k,r_k)]$). The sequence
$\{ x_k\} _k$ is called a {\it blow-up sequence on the scale of curvature}
if $f_k(x_k)\to \infty $ as $k\to \infty $.
}
\end{definition}

\begin{lemma}
\label{lem-blowup}
Suppose that $\{ x_k\} _k\subset L^{\mu }$ is a blow-up sequence. Then, after passing to
a subsequence, the surfaces $L(k)=\sqrt{|K_L|}(x_k)(L^{\mu }-x_k)$ converge
with multiplicity 1 to a vertical helicoid $\mathcal{H}\subset \R^3$ whose axis is the $x_3$-axis
and whose Gaussian curvature is $-1$ along this axis.
\end{lemma}
\begin{proof}
First observe that since the surfaces $M_n$ have injectivity radius function larger than or equal
to $1/2$, then the ball
$\B_k=\B \left( x_k,\frac{1}{2}\mbox{dist}_{\R^3}
(x_k,\partial \B (q_k,r_k))\right) $ intersects $L^{\mu }$ in disks when
$k$ is sufficiently large. Also note that the ball $\sqrt{|K_L|}(x_k)(\B_k-x_k)$ is centered
at the origin and its radius is $f_k(x_k)/2$, which tends to $\infty $ as
$k\to \infty $ because $\{ x_k\} _n$ is a blow-up sequence. Since the second fundamental
form of the surfaces $L(k)\cap \B \left( f_k(x_k)/2\right) $ is uniformly bounded, then a
subsequence of the $L(k)$ (denoted in the same way) converges to a minimal lamination
$\mathcal{L}'$ of $\R^3$ with a leaf $L'$ that is a complete embedded minimal surface that passes through
the origin with absolute Gaussian
curvature 1 at that point. Standard arguments then show that the multiplicity of the
convergence of portions of the $L(k)$ to $L'$ is one. Therefore, a lifting
argument of loops on $L'$ implies
that $L'$ is simply connected, hence $L'$ is a helicoid with maximal absolute
Gaussian curvature 1 at $\vec{0}$ and $L'$ is the only leaf of $\mathcal{L}'$. The
fact that $L'$ is a vertical helicoid with axis the $x_3$-axis
(so $L'=\mathcal{H}$) will follow from the description of the local
geometry of $L^{\mu }$ nearby $x_n$; to
see this, note that the blow-up points $x_k$ and the forming helicoids in $L^{\mu }$
on the scale of curvature near $x_k$ for $k$ large imply the existence of pairs of
highly sheeted almost-flat multivalued graphs $G_{n,k}^1,G_{n,k}^2\subset M_n$
extrinsically close to $x_k$ for $n$ sufficiently large (recall that
portions of the $M_n$ converge to $L^{\mu }$).  These multivalued graphs can be chosen to
have any fixed small gradient over the plane perpendicular to the axis of the helicoid $L'$.
For $n,k$ sufficiently large, these multivalued graphs in $M_n$ each contains a two-valued subgraph
that extends to an almost-flat two-valued graph
on a fixed scale (proportional to the number $\ve >0$ that appears in description (D1)-(D2) above)
and collapse to a punctured disk.
Since the punctured $(x_1,x_2)$-plane $P-\{\vec{0}\}$ is a leaf of the limit
minimal lamination $\mathcal{L}$, it then follows that the helicoid $L'$ must be vertical.
This completes the proof of the lemma.
\end{proof}

The next lemma implies that the same type of limit that appears in Lemma~\ref{lem-blowup}
at a blow-up sequence on the scale of curvature in $L^{\mu }$, also appears when using a
different notion of blow-up. Namely, when we rescale $L^{\mu }$ around points with heights
converging to zero where $L^{\mu }$ is vertical.

\begin{lemma}
\label{vert} Consider a sequence of points
$y_k\in L^{\mu }$ with $x_3(y_k)$
converging to zero where the tangent planes $T_{y_k}L$ to $L^{\mu }$ are vertical.
Then, $y_k$ converge to $\vec{0}$, the numbers $s_k:=\sqrt{|K_{L}|}(y_k)$ diverge to infinity and a subsequence
of the surfaces $L'(k)=s_k(L^{\mu }-y_k)$ converges on compact subsets of $\rth$
to a vertical helicoid $\mathcal{H}$ containing the $x_3$-axis and with maximal absolute
Gaussian curvature $1$ at the origin. Furthermore, the multiplicity of the convergence of the surfaces $L'(k)$
to $\mathcal{H}$ is one.
\end{lemma}
\begin{proof}
We first show that the points $y_k$ tend to the origin as $k\to \infty $. Arguing by contradiction, suppose
after choosing a subsequence that for $k$ large $y_k$ lies outside a ball $\B $ centered
at $\vec{0}$. Note that for $k$ large, the injectivity radius function of $L$ is bounded away
from zero at the $y_k$. As these points are arbitrarily close to $P$, then
the Gaussian curvature of $L$ at the $y_k$ blows up (otherwise $L$ could be written locally
as graphs over vertical disks in $T_{y_k}L$ of uniform size by the uniform graph lemma,
which would contradict that $L$ lies above $\{ x_3=0\} $)
and one obtains a contradiction to the one-sided curvature estimates of
Colding-Minicozzi (Corollary 0.4 in~\cite{cm23}). Therefore, $y_k\to \vec{0}$. Another consequence of
the one-sided curvature estimates is that
\begin{enumerate}[(J)]
\item There exists $\de >0$ such that if $\mu \in (0,\de )$, then
the tangent plane to $L^{\mu }$ at every point in $L^{\mu }\cap
\{ (x_1,x_2,x_3 ) \ | \ x_1^2+x_2^2\geq \de ^2x_3^2\} $
makes an angle less than $\pi /4$ with the horizontal.
\end{enumerate}
 For $k\in \N$ fixed, let $t_k>0$ be the largest radius such that all points in $L^{\mu }\cap \B (y_k,t_k)$ have
tangent plane making an angle less than $\pi /4$ with $T_{y_k}L$; existence of $t_k$ follows from the fact that
$L^{\mu }$ is proper in the slab $\{ 0<x_3\leq \mu \} $. Note that the following properties hold.
\begin{enumerate}[(K1)]
\item  $L^{\mu }\cap \B (y_k,t_k)\subset \{ (x_1,x_2,x_3 ) \ | \ x_1^2+x_2^2<\de ^2x_3^2\} $ (this follows from (J)),
\item $t_k\leq x_3(y_k)$ (otherwise, $\B (y_k,t_k)\cap P$ contains a disk $D\subset P-\{ \vec{0}\} $
which is the limit of a sequence of graphs inside $L^{\mu }\cap \B (y_k,t_k)$ over $D$;
this is clearly impossible by (K1)).
\end{enumerate}

By Property (K2), we have $t_k\to 0$ as $k\to \infty $.
Consider the sequence of translated and scaled surfaces
\begin{equation}
\label{eq:L(k)}
L(k)=\frac{2}{t_k}(L^{\mu }-y_k).
\end{equation}
We claim that to prove the lemma it suffices to demonstrate that
\begin{enumerate}[(N)]
\item {\it Every subsequence of the $L(k)$ has a subsequence that converges with multiplicity one
to a vertical helicoid containing the $x_3$-axis.
}
\end{enumerate}
We will prove the lemma, assuming that Property (N) holds. Let $\{ L(k_i)\} _i$ be a subsequence of the $L(k)$ that converges
with multiplicity one to a vertical helicoid $\mathcal{H}'$ containing the $x_3$-axis. Then,
$\sqrt{|K_{\mathcal{H}'}|}(\vec{0})\in (0,\infty)$ and
 \[
\lim_{i\to \infty}  \sqrt{|K_{L(k_i)}|}(\vec 0 )
 =\lim_{i\to \infty} \frac{t_{k_i}}{2} \sqrt{|K_{L}|}(y_{k_i}).
 \]
Since $\lim _{i\to \infty }t_{k_i}=0$, this implies that the numbers
$s_{k_i}:=\sqrt{|K_{L}|}(y_{k_i})$ diverge to infinity, and the sequence of surfaces
 \[
L'(k_i)=s_{k_i}(L^{\mu }-y_{k_i})=s_{k_i}\frac{t_{k_i}}{2}L(k_i).
 \]
converges with multiplicity one  to $\mathcal{H}=\sqrt{|K_{\mathcal{H}'}|}(\vec{0})
\mathcal{H}'$, and the proposition follows.
Thus, it suffices to prove Property (N); there are two cases to consider after choosing a subsequence. \vspace{.1cm}

\noindent{\underline{Case~(N.1).}} The sequence $\{ L(k)\} _k$ has
uniform local bounds of the Gaussian curvature in $\rth$.
\par
\vspace{.1cm}
In this case, standard arguments show that, after choosing a subsequence,
the $L(k)$ converge to a minimal lamination $\mathcal{L}_\infty$ of $\rth$.
Observe that if $L_1$ is a nonflat leaf of $\mathcal{L} _{\infty }$, then the multiplicity of convergence of portions
of the $L(k)$ to $L_1$ is one; in particular, $L_1$ is simply connected (since the injectivity radius
function of the $L(k)$ becomes arbitrarily large at points in any fixed compact set of $\R^3$ as $k\to
\infty $). On the other hand, if the multiplicity of convergence of portions
of the $L(k)$ to a leaf $L_2$ of $\mathcal{L}_{\infty }$ is greater than one, then $L_2$ is stable hence a plane.
By the classification of simply connected, complete embedded minimal surfaces in $\R^3$ (Meeks and Rosenberg~\cite{mr8}
and Colding and Minicozzi~\cite{cm35}), we conclude that every leaf of $\mathcal{L}_{\infty }$ is either a plane or a helicoid.
Clearly, if $\mathcal{L}_{\infty }$ contains a leaf which is a helicoid, then this is the only leaf of $\mathcal{L}_{\infty }$ and
Property (N) is proved in this case. Since the leaf of $\mathcal{L}'$ passing through $\vec{0}$ has a vertical tangent plane at $\vec{0}$
but at some point in the sphere $\partial \B (2)$ there exists a point on a leaf of $\mathcal{L}_{\infty }$ whose tangent plane makes
an angle at least $\pi /4$ (by definition of $t_k$), then $\mathcal{L}_{\infty }$ contains a leaf which is not a plane.
This finishes Case~(N.1).
\par
\vspace{.2cm}
\noindent{\underline{Case~(N.2).}}  There exists $x_{\infty }\in \R ^3$ and points $x_k\in L(k)$ converging to $x_{\infty }$ such that
$|K_{L(k)}|(x_k)\geq k$ for all $k\in \N$.
\par
\vspace{.1cm}
We will show that this case cannot occur, by dividing it into two subcases after replacing by a subsequence.
\begin{description}
\item[{\rm (N2.1)}] Suppose that $\frac{x_3(y_k)}{t_k}\to \infty $.
\item[{\rm (N2.2)}] Suppose that $\frac{x_3(y_k)}{t_k}$ converges to a number $D$ which is greater than or equal to 1 (by Property (K2) above).
\end{description}
In case (N2.1) holds, we consider the sequence of compact embedded minimal surfaces $\{ L(k)\cap \B (R_k)\} _k$, where
$R_k=\frac{x_3(y_k)}{t_k}$. For $k$ large, every component of $L(k)\cap \B (R_k)$ is a disk with boundary contained in $\partial \B (R_k)$. As the supremum of the norm of the second
 fundamental form of $L(k)\cap \B (2| x_{\infty }| )$ tends to $\infty $ as $k\to \infty $
(by assumption in this case (N.2)),
 then Theorem~0.1 in Colding and Minicozzi~\cite{cm23} and Meeks' regularity theorem~\cite{me25}
assure that after choosing a subsequence, the $L(k)$ converge as $k\to \infty $ to a limit
parking garage structure with one column.
Observe that by equation (\ref{eq:L(k)}), points of $L(k)\cap \B (1)$
correspond to points of $L^{\mu }\cap \B (y_k,\frac{t_k}{2})$, and thus, the tangent plane
to $L(k)$ at every point in $L(k)\cap \B(1)$ makes an angle less than $\pi /4$ with $T_{y_k}L$.
This property implies that the inner product of the Gauss map of $L(k)$ with
the unit normal vector to $T_{y_k}L$ is positive (up to sign)  in $L(k)\cap \B(1)$, hence
$L(k)\cap \B(1)$ is stable. Schoen's curvature estimates~\cite{sc3} now give that
the norm of the second fundamental form of $L(k)\cap \B (1)$ is uniformly bounded.
Since the tangent plane $\Pi $ to $L(k)$
at the origin is vertical, then we conclude that the planes in the limit parking garage structure are
parallel to $\Pi $. As for the line $l$ given by the
singular set of convergence of the $L(k)\cap \B (R_k)$ to the limit parking garage structure, note that by definition
of $t_k$, for large $k$ the tangent plane to $L(k)$ at some point $q_k$ in
the sphere $\partial \B (2)$ makes an angle at least $\pi /4$
with $\Pi $; this implies that $l$ is the straight line orthogonal to $\Pi $ that passes through
the limit of the $q_k$ (after passing to a subsequence).
By similar arguments as those at the end of the proof of Lemma~\ref{lem-blowup},
we can find pairs of highly sheeted almost-flat almost-vertical multivalued graphs
$G_{n,k}^1,G_{n,k}^2\subset M_n$ over portions of $\Pi $, and these multivalued graphs
in $M_n$ contain two-valued subgraphs that extend to two-valued almost-vertical
multivalued graphs on a fixed scale. By the arguments at the end of the proof of Lemma~\ref{lem-blowup},
these extended two-valued almost-vertical multivalued graphs must be almost-horizontal, which gives a contradiction.
This finishes the case (N2.1).

Finally, suppose case (N2.2) occurs. By the application of a diagonal-type
argument to the doubly indexed sequence of surfaces $\{M_n-y_k\}_{n,k\in \N}$
where $n$ is chosen to go to infinity sufficiently quickly in terms of $k$ that also goes to infinity, we can
produce a sequence
\[
\L_{n(k)}=\left\{ \frac{2}{t_k}(M_{n(k)}-y_k)\right\}_{k\in \N}
\]
such that the following properties hold.
\begin{enumerate}[(O1)]
\item The injectivity radius function of $\frac{2}{t_k}(M_{n(k)}-y_k)$ can be made arbitrarily large for $k$ large
at every point a any fixed ball in $\R^3$ (this follows from Property (Inj) just before the statement of
Proposition~\ref{lemma4.2} after rescaling by $2/t_k$).
\item There exists a (possibly empty) closed set $\mathcal{S}_{\infty }\subset \R^3$ and a minimal lamination $\mathcal{L} _{\infty }$ of
$\R^3-\mathcal{S} _{\infty }$ such that the surfaces $\frac{2}{t_k}(M_{n(k)}-y_k)-\mathcal{S} _{\infty }$ converge to $\mathcal{L} _{\infty }$
outside of some singular set of convergence $S(\mathcal{L} _{\infty })\subset \mathcal{L} _{\infty }$, and if we call $\Delta (\mathcal{L}_{\infty })=
\mathcal{S} _{\infty }\cup S(\mathcal{L} _{\infty })$, then $\Delta (\mathcal{L}_{\infty })\neq \mbox{\O }$ (this property holds by similar arguments
as those that prove the first part of item~\ref{i6} of Theorem~\ref{tthm3introd}, which are still valid since we have property
(O1)).
\item Through every point in $\Delta (\mathcal{L}_{\infty })$ there passes a plane which contains a planar leaf of $\mathcal{L} _{\infty }$
and which intersects $\Delta (\mathcal{L} _{\infty })$ in exactly one point (two or more points would produce connection loops in the
surfaces $\frac{2}{t_k}(M_{n(k)}-y_k)$ for $k$ large, in contradiction with property (O1) above).

\item By our hypotheses in Case (N2.2), we deduce that one of the leaves of $\mathcal{L}_\infty$ is contained in the plane $\{ x_3=-2D\} $.
In particular, the planes mentioned in property (O3) are horizontal.

\item $\mathcal{L} _{\infty }$ contains a sublamination $\wh{\mathcal{L}}_\infty$ which is a limit as $k\to \infty $ of the surfaces $L(k)$.

\end{enumerate}

By property (O1), it follows from our previous arguments in this paper that every nonflat leaf $Z$ of $\mathcal{L}_\infty$
is simply connected. Furthermore, the injectivity radius function of such a $Z$ at any point $z\in Z$ is equal to
the intrinsic distance from $z$ to boundary of the metric completion $\ov{Z}$ of $Z$, where the points of this metric completion
correspond to certain (singular) points in $\mathcal{S}_\infty$. Observe that such a $Z$ cannot be complete (otherwise,
by the discussion in Case (N.1), $Z$ would be a helicoid which is impossible by (O3) or (O4)).

We next show that there exists a nonflat $Z_1$ of $\mathcal{L}_\infty$ that passes through the origin.
Since the norms of the second fundamental forms of the surfaces $\frac{2}{t_k}(M_{n(k)}-y_k)$
are uniformly bounded in the ball of radius 1 centered at the origin, there is a
leaf $Z_1$ of  $\mathcal{L}_\infty$ passing through $\vec{0}$ with vertical tangent plane $T_{\vec{0}}Z_1$.
Since the flat leaves of $\mathcal{L}$ are horizontal, then $Z_1$ cannot be flat. By the last paragraph,
$Z_1$ is not complete, hence there exists $p_0\in \mathcal{S} _{\infty }$ in the metric completion of $Z_1$.
By Property (O3), the punctured horizontal plane $\Pi _0=\{ x_3=x_3(p_0)\} -\{ p_0\} $ is a leaf of
$\mathcal{L}_{\infty }$. By the same arguments and the connectedness of $Z_1$,
there is at most one other point $p_1$ in the metric completion
of $Z_1$, and in this case the plane $\Pi _1=\{ x_3=x_3(p_1)\} -\{ p_1\} $ is a leaf of $\mathcal{L}_{\infty }$
(if no such $p_1$ exists, then $Z_1$ is properly embedded in the upper open halfspace determined by
the plane $\Pi _0$). Without loss of generality, we can assume that if $\Pi _1$ exists, then
$x_3(p_1)>x_3(p_0)$.

Given $\de >0$, consider the conical region
\[
C^+(\de )=\{ (x_1,x_2,x_3)\ | \ (x_1-x_1(p_0))^2+(x_2-x_2(p_0))^2
< \de ^2(x_3-x_3(p_0))^2\}
\]
 with vertex $p_0$. Suppose that no $p_1$ exists. In this case, the injectivity radius function
Inj$_{Z_1}(x)$ at any point $x\in Z_1$ is equal to the intrinsic distance function in $Z_1$ from $x$ to $p_0$,
which in turn is at least $| x-p_0| $. Therefore, Inj$_{Z_1}$ grows at least linearly with the
extrinsic distance in $Z_1$ to $p_0$. If $p_1$ exists, the same property can be proven for $\de >0$ sufficiently small
with minor modifications, since for such $\de $, there exists $a=a(\de )>0$ such that
$C^+(\de)$ also contains $p_1$ (if $p_1$ exists) and
\begin{equation}
\label{dist}
\min \{ | x-p_0| ,| x-p_1| \} \geq a\ | x-p_0| ,
\end{equation}
for all $x\in x_3^{-1}([x_3(p_0),x_3(p_1)])-C^+(\de )$. This scale invariant lower bound on Inj$_{Z_1}$
together with the intrinsic version of the one-sided curvature estimate by Colding-Minicozzi
(Corollary~0.8 in~\cite{cm35})
imply that for $\de >0$ sufficiently small, the intersection of $Z_1$ with
$x_3^{-1}([x_3(p_0),x_3(p_1)])-C^+(\de )$ consists of two multivalued graphs whose gradient
can be made arbitrarily small (in terms of $\de $). The same scale invariant lower bound on Inj$_{Z_1}$
is sufficient to apply the arguments in page 45 of Colding-Minicozzi~\cite{cm25}; especially see
the implication that property (D) there implies property (D1). In our current setting, property (D) is the
scale invariant lower bound on Inj$_{Z_1}$, and property (D1) asserts that $Z_1-C^+(\de _1)$
consists of a pair of $\infty $-valued graphs for some $\de _1>0$ small, which can be connected by curves
of uniformly bounded length arbitrarily close to $p_0$.
The existence of such $\infty $-valued graphs over the punctured plane contradicts
Corollary~1.2 in~\cite{cm26}. This contradiction rules out Case (N2.2), which
finishes the proof of Lemma~\ref{vert}.
\end{proof}

The following corollary is an immediate consequence of Lemma~\ref{vert}.

\begin{corollary}
\label{lookslikehel}
Given $\ve_1,R>0$, there exists an $\ve_2\in(0,\ve_1)$  such that the following holds. Let
\[
\g =\{ y\in L^{\mu} \cap \B(\ve_2)\ | \  T_yL \mbox{ is vertical } \}.
\]
Then for any $y\in \g$, there exists a vertical helicoid $\mathcal{H}_y$ with
maximal absolute Gaussian curvature~$1$
at the origin\footnote{Observe that for $y_1,y_2\in \g$, the helicoids
$\mathcal{H}_{y_1}$, $\mathcal{H}_{y_2}$ coincide up to a rotation around the $x_3$-axis.}
such that the connected component of $\sqrt{|K_{L}|}(y)
[L^{\mu }-y]\cap \B(R)$ containing the origin is a normal graph $u$ over its projection
$\Omega\subset \mathcal{H}_y$, and
\[
\B(R-2\ve_1)\cap \mathcal{H}_y\subset \Omega \subset \B(R+2\ve_1)\cap
\mathcal{H}_y, \qquad \|u\|_{C^2(\Omega )}\leq \ve_1.
\]
\end{corollary}

Three immediate consequences of Corollary~\ref{lookslikehel} when $\ve_2$ is chosen sufficiently small are:
\begin{enumerate}[(P1)]
\item The set $\g$ in Corollary~\ref{lookslikehel} can be parameterized as a
connected analytic curve $\g(t) $ where $t\in (0,t_0]$ is its positive $x_3$-coordinate,
and $\lim_{t\to 0} \g'(t)=(0,0,1)$.
\item Given $t\in (0,t_0]$, let $f(t)\in \R$ be the angle that the vertical plane $T_{\g (t)}L$
makes with the positive $x_1$-axis, that is, $(-\sin f(t),\cos f(t),0)$ is the unit normal vector
to $T_{\g (t)}L$ up to a choice of orientation. Note that $T_{\g(t)}L$
rotates infinitely often as $t\searrow 0^+$, and consequently, the angle function $f(t)$
can be considered to be a smooth function
of the height that tends to $+\infty $ as $t\searrow 0^+$ if the forming
helicoid $\mathcal{H}_{\g(t)}$ is left-handed, or to $-\infty$ if $\mathcal{H}_{\g (t)}$ is right-handed.
In the sequel we will suppose that this last possibility occurs (after a possible reflection of
$L^{\mu }$ in the $(x_1,x_3)$-plane). Also observe that $f(t)$ is determined up to an additive multiple of $2\pi $,
and so, we do not loose generality by assuming that $f(t)<0$ for each $t\in (0,t_0]$. Since for
the right-handed vertical helicoid $\mathcal{H}$ the corresponding angle function $f_{\mathcal{H}}$ is negative
linear, then we conclude that after choosing $t_0>0$ small enough,
 $f'>0$ is bounded away from zero
and $f''/f'$ is bounded from above in $(0,t_0]$.

\item For any $t\in (0,t_0]$, $L^{\mu }\cap T_{\g (t)}L$ contains a small smooth arc $\a_{\g (t)}$ passing
through $\g (t)$ that is a graph over its projection to
the horizontal line  $x_3^{-1}(t)\cap T_{\g (t)} L$. Since for $t>0$ small
the point $\g (t)$ is a point of almost-maximal curvature in a certain ball centered at $\g (t)$,
then the forming double multivalued graph around $\a _{\g (t)}$ in $L^{\mu }$
extends sideways almost horizontally in $T_{\g (t)}L$
by the extension results in Colding and Minicozzi (Theorem II.0.21 in~\cite{cm21}, note also
that $L^{\mu }-\g $ consists of stable pieces), until exiting the solid vertical cone
$\mathcal{C}_{\vec{0}}$ given by description (D2) before Lemma~\ref{lemma4.5},
 whose vertex is the singular point $x=\vec{0}\in \mathcal{S}$.
 This allows us to extend $\a _{\g (t)}$ in the vertical plane
$T_{\g (t)}L$ until it exits $\mathcal{C}_{\vec{0}}$.
Once $\a _{\g (t)}$ exits $\mathcal{C}_{\vec{0}}$, then the almost-horizontal nature
of $L^{\mu }$ outside $\mathcal{C}_{\vec{0}}$ for $\mu >0$ small that comes from
curvature estimates, insures that $\a _{\g (t)}$ can be extended in $L^{\mu }\cap T_{\g (t)}L$
as an almost-horizontal arc, until it possibly intersects the plane $\{ x_3=\mu \}$.
The number of intersection points of this extended arc $\a _{\g (t)}$ with $\{ x_3=\mu \}$
is zero, one or two. If $\a _{\g (t)}$ never
intersects $\{ x_3=\mu \} $ (respectively, if $\a _{\g (t)}$ intersects $\{ x_3=\mu \} $ exactly once),
then $\a _{\g (t)}$ defines a proper open (respectively half-open) arc in
$\{ 0<x_3\leq \mu \} \cap T_{\g (t)}L^{\mu }$. Otherwise, $\a _{\g (t)}$ is a compact
arc with its two end points at height $\mu $.
\end{enumerate}

We next show that for $t_0>0$ sufficiently small and for all $t\in (0,t_0]$,
the number of intersection points of $\a _{\g (t)}$ with $\{ x_3=\mu \}$ is two.
Arguing by contradiction, suppose that for some $t_1\in (0,t_0]$ small,
$\a _{\g (t_1)}$ is not a compact arc. Then,
there exists $t_2\in (0,t_1)$ such that for all $t\in (0,t_2]$, $\a _{\g (t)}$ is
an open proper arc in $\{ 0<x_3<\mu \} \cap T_{\g (t)}L$,
which is a graph over the horizontal line $x_3^{-1}(t)\cap T_{\g (t)}L$.
It follows that the surface
\begin{equation}
\label{Sigma}
\Sigma (t_2)=\bigcup _{t\in (0,t_2]}\a _{\g (t)}
\end{equation}
is a proper subdomain of $L^{\mu}$, $\Sigma (t_2)$ is topologically a disk with
connected boundary and when intersected with the domain $x_3^{-1}([0,\mu ])\cap
\{ x_1^2+x_2^2\geq 1\} $, is an $\infty $-valued graph.
This minimal surface $\Sigma(t_2) $
cannot exist by the flux arguments in~\cite{cm26} (specifically see Corollary~1.2 and
the paragraph just after this corollary.
Thus, we may assume that $\a (t)$ is a compact arc for
all $t\in (0,t_2]$ and $t_2>0$ sufficiently small.
Observe that $\a (t)$ is transversal to $\{ x_3=\mu \} $
at the two end points of $\a (t)$ (because $\mu $ was
a regular value of $x_3$ in $L$), for all $t\in (0,t_2]$.

By the above discussion, for $t_2>0$ sufficiently small,
$\Sigma (t_2)$ is a union of the compact arcs $\a _{\g (t)}$, $t\in (0,t_2]$. Let $\G _1(t),\G_2(t)$
be the end points of $\a _{\g(t)}$, $t\in (0,t_2]$. Hence for $i=1,2$, $t\in (0,t_2]\mapsto \G _i(t)$
is an embedded proper arc in $L\cap \{ x_3=\mu \} $ that spins infinitely often, and
$\G_1,\G_2$ are imbricated (they rotate together). The boundary of $\Sigma (t_2)$ is connected and
consists of $\a _{\g (t_2)}\cup \G _1\cup \G _2$. Consider the
piecewise smooth surface $\wh{\Sigma }(t_2)$ obtained by adding to each $\a _{\g (t)}$ the
two disjoint halflines $l_{t,1},l_{t,2}$ in $x_3^{-1}
(\mu )\cap T_{\g (t)}L$ that start at $\G _1(t),\G _2(t)$, respectively, for all $t\in (0,t_2]$.
Observe that $\Sigma (t_2)$ is a subdomain of $\wh{\Sigma }(t_2)$, that $\wh{\Sigma }(t_2)$
fails to be smooth precisely $\G _1\cup \G _2$, and that
$\wh{\Sigma }(t_2)$ fails to be embedded since for certain values
$t<t'\in (0,t_2]$, the added halflines $l_{t,1},l_{t',1}$ satisfy $l_{t,1}\subset l_{t',1}$,
and similarly for the halflines $l_{t,2},l_{t',2}$. Both problems for $\wh{\Sigma }(t_2)$
can be easily overcome (actually embeddedness is not strictly
necessary in what follows) by slightly changing the construction,
as we now explain. For each $t\in (0,t_2]$, enlarge slightly $\a _{\g (t)}$ to a
compact arc $\wh{\a }_{\g (t)}\subset L\cap T_{\g (t)}L$, so that if we call
$\wt{\G }_1(t),\wt{\G }_2(t)\in L\cap x_3^{-1}((\mu ,\mu +1])$
to the end points of $\wh{\a }_{\g (t)}$,
then the following properties hold.
\begin{enumerate}[(Q1)]
\item For $i=1,2$, the correspondence $\G _i(t)\mapsto \wt{\G }_i(t)$ defines a smooth map
 that goes to zero as $t\searrow 0$. In other words, the curve $t\in (0,t_2]\mapsto \wt{\G }_i(t)$
is asymptotic to the planar curve $t\in (0,t_2]\mapsto \G _i(t)$ as $t\searrow 0$.
\item $x_3\circ \wt{\G }_1(t)=x_3\circ \wt{\G }_2(t)$ is strictly increasing as a function of $t\in (0,t_2]$.
\end{enumerate}
Now add to each $\wh{\a }_{\g (t)}$ the two disjoint halflines $\wt{l}_{t,1},\wt{l}_{t,2}$ in $x_3^{-1}
(x_3(\wt{\G }_1(t)))\cap T_{\g (t)}L$ that start at $\wt{\G }_1(t), \wt{\G }_2(t)$,
respectively, for all $t\in (0,t_2]$.
By property (Q2) above, the piecewise smooth surface
\[
\wt{\Sigma }(t_2)=\bigcup _{t\in (0,t_2]}\left[ \wh{\a }_{\g (t)}\cup \wt{l}_{t,1}\cup \wt{l}_{t,2}\right]
\]
is embedded and fails to be smooth precisely along $\wt{\G} _1\cup \wt{\G }_2$.
Now smooth $\wt{\Sigma }(t_2)$ by rounding off the corners along
$\wt{\G} _1\cup \wt{\G }_2$ in a neighborhood of these curves that is disjoint from
$\Sigma (t_2)$, and relabel the resulting smooth embedded surface as  $\wt{\Sigma }(t_2)$.
Furthermore, the above smoothing process can be done so that the tangent spaces to
$\wt{\Sigma }(t_2)$ form an angle less than $\pi /4$ with the horizontal. Observe that
$\Sigma (t_2)$ is a proper subdomain of $\wt{\Sigma }(t_2)$, see Figure~\ref{parabhelic1}.
\begin{figure}
\begin{center}
\includegraphics[width=11.4cm]{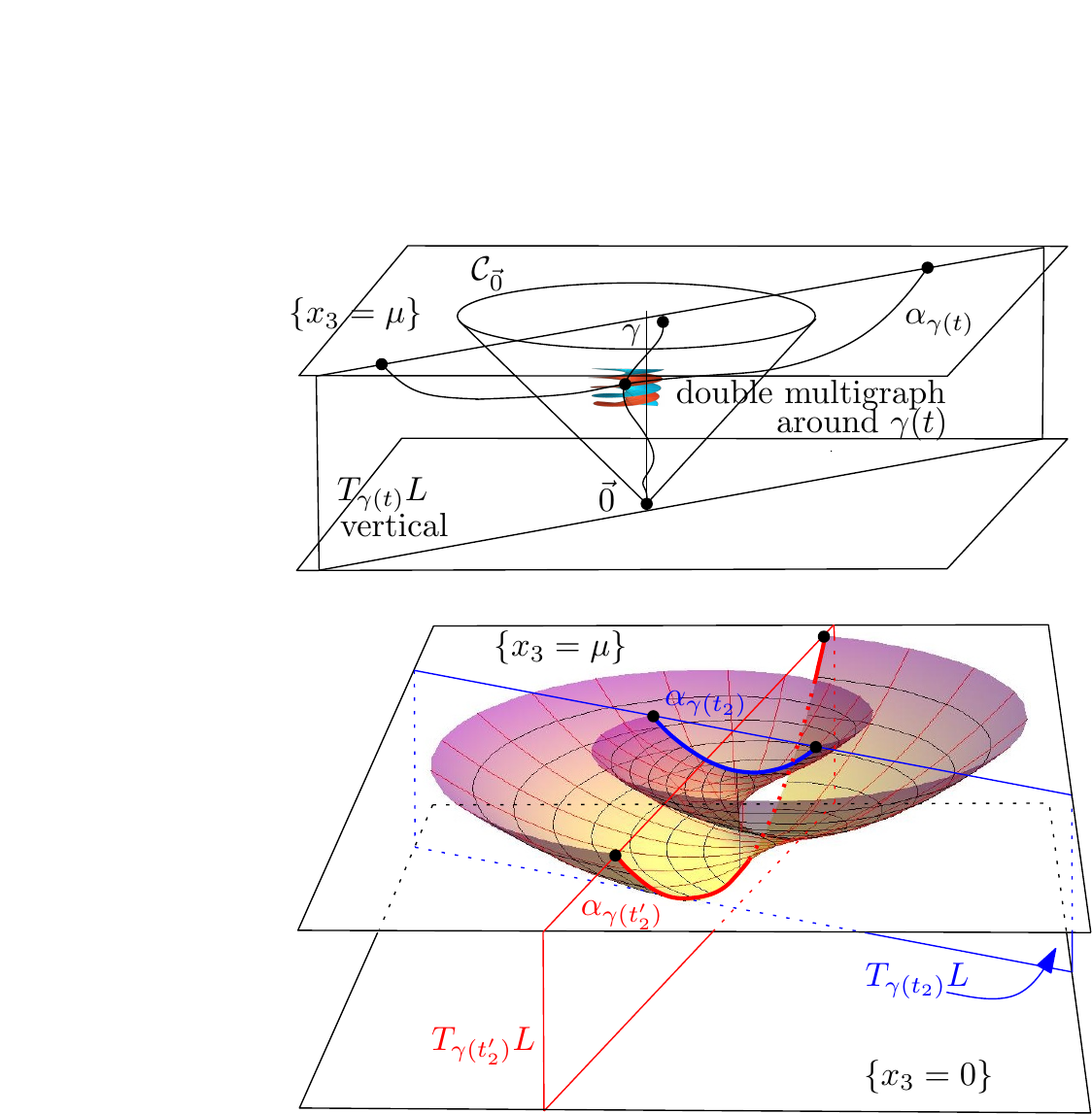}
\caption{Top: The arc $\a _{\g (t)}\subset L^{\mu }\cap T_{\g (t)}L$ starts forming in
the double multivalued graph around the point $\g (t)$, extends sideways until
exiting the solid cone $\mathcal{C}_{\vec{0}}$ and eventually intersects $\{ x_3=\mu \} $.
Bottom: Schematic representation of a compact portion of the surface $\Sigma (t_2)\subset L^{\mu }$,
foliated by
arcs $\a _{\g (t)}$ in a compact range $t\in [t_2',t_2]$, $0<t_2'<t_2$.}
\label{parabhelic1}
\end{center}
\end{figure}
We denote by $\wt{\a }_{\g(t)}\subset \wt{\Sigma }(t_2)\cap T_{\g (t)}L$
the smooth proper arc that extends $\a _{\g (t)}$. Thus, $\wt{\Sigma }(t_2)$ is foliated by
these arcs $\wt{\a }_{\g (t)}$, $t\in (0,t_2]$.

Now consider the ruled surface
\begin{equation}
\label{eq:ruled}
R(t_2)=\bigcup _{t\in (0,t_2]}\left( x_3^{-1}(t)\cap T_{\g (t)}L\right) .
\end{equation}
The vertical projection $\Pi \colon \wt{\Sigma }(t_2)\to R(t_2)$ defined by
$\Pi (x,y,z)=(x,y,t)$ if $(x,y,z)\in \wt{\a }_{\g (t)}$,
is a quasiconformal diffeomorphism; near $\g $, this property follows from the fact
that both $\wt{\Sigma }(t_2),R(t_2)$ can be rescaled around
$\g (t)$ to produce the same vertical helicoid, and away from $\g $ because
the tangent planes to both $\wt{\Sigma }(t_2),R(t_2)$ form a small angle
with the horizontal for $t_2$ sufficiently small.

The next lemma will show that the surface $R(t_2)$ is
quasiconformally diffeomorphic to a closed halfplane in $\C $.
Observe that the hypotheses of Lemma~\ref{lemma4.23} hold for $R(t_2)$,
see items (P1), (P2) above.

\begin{lemma}
\label{lemma4.23}
Let $\G \colon (0,1]\to \R^2$ be a smooth curve and $f\colon (0,1]\to (-\infty ,0)$
be a $C^2$ function. Consider the
ruled surface $R\subset \R^3$ parameterized by $X\colon \R\times (0,1]
\to \R^3$,
\begin{equation}
\label{param}
 X(\mu ,z)=(\G (z),0)+(\mu \cos f(z),\mu \sin f(z),z),\qquad  (\mu ,z)\in
\R \times (0,1].
\end{equation}
If $|\G'|$ is bounded, $\lim _{z\to 0^+}f(z)=-\infty $, $f'$ is bounded away from zero
and $f''/f'$ is bounded from above, then $R$ is quasiconformally diffeomorphic to the
closed lower half of a vertical helicoid.
\end{lemma}
\begin{proof}
Consider the diffeomorphism $\psi \colon x_3^{-1}((0,1])\to x_3^{-1}((0,1])$ given by
\[
\psi (p,z)=(p-\G (z),z),\qquad (p,z)\in \R^2\times (0,1].
\]
As $|\G'|$ is bounded, then $\psi $ is quasiconformal;
this means that there exists $\ve \in (0,1)$ such that
given two unitary orthogonal vectors $a,b\in \R^3$, we have
\[
\ve \leq \frac{|\psi _*(a)|}{|\psi _*(b)|}\leq \frac{1}{\ve },\qquad \frac{\langle \psi _*(a),
\psi _*(b)\rangle }{|\psi _*(a)|\, |\psi _*(b)|}\in [-1+\ve ,1-\ve ],
\]
where $\psi _*$ denotes the differential of $\psi $ at any point of $x_3^{-1}((0,1])$.
Therefore, after composing with $\psi $, we may assume in the sequel that $\G (z)=0$ for all
$z\in (0,1]$.

Let $\mathcal{H}=\{ (x,y,z)\in  \R^3\ | \ y=x\tan z\} $ be the standard vertical helicoid. Consider a map
$\phi \colon R\to \mathcal{H}$ of the form
\[
\phi (X(\mu ,z))=\left( \wh{\mu }\cos f(z),\wh{\mu }\sin f(z),f(z)\right) , \qquad (\mu ,z)\in
\R \times (0,1],
\]
where $\wh{\mu }=\wh{\mu }(\mu ,z)$ is to be defined later. We will find a choice of
$\wh{\mu }$ for which $\phi $ is a quasiconformal diffeomorphism
from $R$ onto its image (which is the lower half of $\mathcal{H}$ obtained after intersection of $\mathcal{H}$
with $x_3^{-1}((-\infty ,f(1)])$. Observe that a global choice of an orthonormal basis for
the tangent bundle to $R$ is $\{ X_{\mu }, \frac{1}{|X_z|}X_z\} $. By definition, $\phi $
is quasiconformal if the following two properties hold.
\begin{enumerate}[(R1)]
  \item ${\displaystyle
  \frac{|\phi _*(\frac{1}{|X_z|}X_z)|}{|\phi _*(X_\mu )|}}$ is
bounded and bounded away from zero (uniformly on $M$).
  \item ${\displaystyle \frac{\langle \phi _*(X_\mu ),
  \phi _*(\frac{1}{|X_z|}X_z)\rangle ^2}
  {|\phi _*(X_\mu )|^2\, |\phi _*(\frac{1}{|X_z|}X_z)|^2}\in [0,1-\ve ]}$
  uniformly on $M$,
 for some $\ve \in (0,1)$.
\end{enumerate}
A direct computation gives
\[
\begin{array}{rcl}
\phi _*(X_\mu ) &=&\wh{\mu }_{\mu }(\cos f,\sin f,0),
\\
\phi _*(X_{z})&=&\wh{\mu }_{z}(\cos f,\sin f,0)+
\wh{\mu }f'(-\sin f,\cos f,0)+(0,0,f'(z)),
\end{array}
\]
where $\wh{\mu }_{\mu }=\frac{\partial \wh{\mu }}{\partial \mu }$,
$\wh{\mu }_{z}=\frac{\partial \wh{\mu }}{\partial z}$.
Hence
\[
\begin{array}{rcl}
|\phi _*(X_{\mu })|^2&=&(\wh{\mu }_{\mu })^2,
\\
|\phi _*(X_{z})|^2&=&(\wh{\mu }_{z})^2+[1+\wh{\mu }^2](f')^2,
\\
\langle \phi _*(X_{\mu }), \phi _*(X_{z})\rangle &=&\wh{\mu }_{\mu }\wh{\mu }_{z}.
\end{array}
\]
Thus,
\begin{equation}
\label{eq:qc1}
 \frac{|\phi _*(\frac{1}{|X_z|}X_z)|^2}{|\phi _*(X_\mu )|^2}=
\frac{(\wh{\mu }_{z})^2+[1+\wh{\mu }^2](f')^2}{[1+\mu ^2(f')^2](\wh{\mu }_{\mu })^2},
\end{equation}
\begin{equation}
\label{eq:qc2}
\frac{\langle \phi _*(X_\mu ),\phi _*(\frac{1}{|X_z|}X_z)\rangle ^2}
{|\phi _*(X_\mu )|^2\, |\phi _*(\frac{1}{|X_z|}X_z)|^2}=
\frac{(\wh{\mu }_z)^2}{(\wh{\mu }_{z})^2+[1+\wh{\mu }^2](f')^2}.
\end{equation}
To simplify the last two expressions, we will take $\wh{\mu }(\mu ,z)=\mu f'(z)$ (note that
this choice of $\wh{\mu }$ makes $\phi $ a diffeomorphism onto its image, as $f'$ does
not vanish). The right-hand-side of (\ref{eq:qc1}) transforms into
\begin{equation}
\label{eq:Q1}
E(\mu ,z):=1+\frac{\mu ^2}{1+\mu ^2(f')^2}\left( \frac{f''}{f'}\right)^2,
\end{equation}
which is greater than or equal to 1 and bounded from above  under our hypotheses
on $f'$ and $f''/f'$, thereby giving
(R1). As for the right-hand-side of (\ref{eq:qc2}), a direct computation
shows that it equals
\begin{equation}
\label{eq:Q2}
\frac{\mu ^2(f'')^2}{(f')^2[1+\mu ^2(f')^2]+\mu^2(f'')^2}=1-\frac{1}{E(\mu ,z)}.
\end{equation}
As $E(\mu ,z)$ is bounded from above, we conclude that the last expression is bounded
away from 1 (below 1), and (R2) is also proved. Therefore, $\phi $ is a
quasiconformal diffeomorphism  from $R$ onto the lower half of a vertical helicoid,
with the choice $\wh{\mu }(\mu ,z)=\mu f'(z)$.
\end{proof}

\begin{lemma}
\label{lemma4.24}
Let $\Sigma $ be a simply connected surface which is quasiconformally diffeomorphic
to a closed halfplane in $\C $. Then, $\Sigma $ is conformally diffeomorphic to a closed halfplane.
\end{lemma}
\begin{proof}
Suppose the lemma fails. Since $\Sigma $ is simply connected, then $\Sigma $ can be
conformally identified with closed unit disk $\overline{\D }=\{ z\in \C \ | \ |z|\leq 1\} $
minus a closed interval $I\subset \partial \D $ that does not reduce to a point. Let
$\Sigma ^*=(\C \cup \{ \infty \} )-I$ be the simply connected Riemann surface
obtained after gluing $\Sigma $
with a copy of itself along $\partial \Sigma $ (by the identity function on $\partial \Sigma $).
Thus by the Riemann mapping theorem, $\Sigma ^*$ is conformally diffeomorphic
to the open unit disk $\D $. On the other hand,
the quasiconformal version of the Schwarz reflection principle (see e.g.,
Theorem~3.6 in~\cite{kri1}) implies that
 $\Sigma ^*$ is quasiconformally diffeomorphic to the surface obtained
after doubling a closed halfplane along its boundary, which is $\C $.
In particular, we deduce that $\C $ is quasiconformally diffeomorphic to $\D$, which
is a contradiction (see e.g., Corollary~1 in~\cite{zakzei}).
\end{proof}

{\it Proof of Proposition~\ref{propos4.22}.}
As the minimal surface $\Sigma (t_2)$ defined in (\ref{Sigma}) can be considered to be a proper
subdomain of the surface $\wt{\Sigma }(t_2)$ defined immediately before
(\ref{eq:ruled}), and $\wt{\Sigma }(t_2)$ is conformally diffeomorphic to
a closed halfplane (by Lemmas~\ref{lemma4.23} and \ref{lemma4.24}), then $\Sigma
(t_2)$ is a parabolic surface. The restriction of the $x_3$-coordinate function to $\Sigma (t_2)$
is a bounded harmonic function with boundary values greater than or equal to
$m=\min \{ x_3(q)\ | \ q\in \a _{\g (t_2)}\} >0$. In particular, the parabolicity of
$\Sigma (t_2)$ insures that
\[
m=\min _{\partial \Sigma (t_2)}x_3 \leq x_3\leq \max _{\partial \Sigma (t_2)}x_3,
\]
which contradicts that $\Sigma (t_2)$ contains points at
height arbitrarily close to zero. Now Proposition~\ref{propos4.22} is proved.

Note that Proposition~\ref{propos4.22} finishes the proof of item~6 of Theorem~\ref{tthm3introd}
(see the paragraph just after the statement of Proposition~\ref{propos4.22}).
Therefore, the proof of Theorem~\ref{tthm3introd} is complete.

We next prove some additional information about case~6 of Theorem~\ref{tthm3introd}.
\begin{proposition}
\label{proposnew}
Suppose that $\mathcal{S}\neq \mbox{\rm \O }$ (hence item~6 of Theorem~\ref{tthm3introd} holds).
Then:
\begin{enumerate}[(A)]
\item $\Delta (\mathcal{L})=\mathcal{S}\cup S(\mathcal{L})$ is a closed set of $\rth$ which is contained in the union
of planes  $\bigcup_{L \in \mathcal{P}} \overline{L}$. Furthermore,
    every plane in $\R^3$ intersects $\mathcal{L}$.
\item There exists $R_0>0$ such that the
sequence of surfaces $\left\{ M_n\cap B_M(p_n,\frac{R_0}{\lambda_n})\right\} _n$
does not have bounded genus.
\item There exist oriented closed geodesics  $\g_n\subset \lambda_nM_n$
with uniformly bounded lengths which converge to a line segment $\g$
in the closure of some flat leaf in $\mathcal{P}$, which
joins two points of $\Delta(\mathcal{L})$, and such that the integrals of $\lambda_nM_n$ along $\g_n$
in the induced exponential $\rth$-coordinates of $\lambda_nB_N(p_n,\ve _n)$
converge to a horizontal vector orthogonal to $\g$ with length $2\, \mbox{\rm Length}(\g)$.
\end{enumerate}
\end{proposition}
\begin{proof}
$\Delta (\mathcal{L})$ is closed in $\R^3$ since $\mathcal{S} $ is closed in
$\R^3$ and $S(\mathcal{L})$ is closed in $\R^3-\mathcal{S} $. Lemma~\ref{lemma4.5}
implies that $\Delta (\mathcal{L})$ is contained in $\bigcup _{L\in \mathcal{P}}\overline{L}$.
Every plane in $\R^3$ intersects $\mathcal{L}$ by Lemma~\ref{lemma4.6}, and so
item~(A) of the proposition holds.

We next prove item~(B).
Consider a plane $P\in \mathcal{P}'$ such that $P\cap \mathcal{S}\neq \mbox{\O }$. By
Proposition~\ref{propos4.22}, $P$ intersects $\Delta (\mathcal{L} )$ in at least two points.
If this intersection consists of exactly two points with opposite
orientation numbers, then Proposition~\ref{propos4.11} implies that $\mathcal{L}$ is a foliation of $\R^3$, which
contradicts that $\mathcal{S} \neq \mbox{\O }$. Therefore, there exist
two points $x_1,x_2$ in $P\cap \mathcal{S}$ with the same orientation
number. In this setting, we can adapt the arguments in the proof of Lemma~\ref{ass2.1} to conclude that
$B_{\lambda_nN}(x_1,2d_n)\cap (\lambda_nM_n)$ has unbounded genus for $n$ large enough, where $d_n$ is the extrinsic
distance in $\lambda_nN$ from $x_1$ to $x_2$ (note that $d_n$ converges as $n\to \infty $ to $| x_1-x_2| $).
Finally, take $k_0\in \N$ so that $B_{\lambda_nN}(x_1,2d_n)\cap (\lambda_nM_n)$ is contained in $B_{\lambda_nN}(p_n,k_0)$,
for all $n\in \N$.
By Proposition~\ref{propos4.18}, each point in $B_{\lambda_nN}(x_1,2d_n)\cap (\lambda_nM_n)$ is at an intrinsic distance
not greater than some fixed number $R_0$ (depending on $k_0$) from $p_n$, for all $n$ sufficiently large. After coming
back to the original scale, this implies that the surfaces $M_n\cap B_M(p_n,\frac{R_0}{\lambda_n})$ do not have bounded
genus. This finishes the proof of
item~(B) of the proposition.

Finally, item (C) follows from applying
to the plane $P$ that appears in the previous paragraph
the arguments in the proof of Lemma~\ref{lemma4.7}.
This concludes the proof of
the proposition.
\end{proof}

\begin{remark}
{\rm The techniques used to prove Theorem~\ref{tthm3introd} have
other consequences. For example, suppose $\{M_n \}_n$ is a sequence
of compact embedded minimal surfaces in $\rth$ with $\vec{0} \in M_n$ whose
boundaries lie in the boundaries of balls $\B(R_n)$, where $R_n
\rightarrow \infty$. Suppose that  there exists some $\varepsilon >
0$ such that for any ball $\B $ in $\R^3$ of radius $\varepsilon$, for
$n$ sufficiently large, $M_n \cap \B $ consists of disks,
and such that for some fixed compact set $C$, there exists a $d > 0$
such that for $n$ large, the injectivity radius function of $M_n$ is at
most $d$ at some point of $M_n \cap C$. Then the proof of
Theorem~\ref{tthm3introd} shows that, after replacing by a
subsequence, the $M_n$
converge on compact subsets of $\R^3$ to one of the following cases:
\begin{enumerate}[(4.29.a)]
\item A properly embedded, nonsimply connected minimal surface $M_{\infty }$ in
$\R^3$. In this case, the convergence of the surfaces $M_n$ to $M_{\infty }$
is smooth of multiplicity one on compact sets of $\R^3$.
\item A minimal parking garage structure of $\R^3$ with at least two columns.
\item A singular minimal lamination $\mathcal{L}$ of $\R^3$ with properties
similar to the minimal lamination described in item~{6} of
Theorem~\ref{tthm3introd} and in Proposition~\ref{proposnew}.
\end{enumerate}
}
\end{remark}

\begin{remark}
{\rm
In~\cite{mpr11}, we will apply Theorem~\ref{tthm3introd} under slightly weaker
hypotheses for the embedded minimal surface $M$ appearing in it, namely $M$
is not assumed to be complete, but instead we will suppose that $M$ satisfies the following condition.

Suppose $M$ is an embedded minimal surface, not necessarily complete and
possibly with boundary,
in a homogeneously regular three-manifold $N$. Observe that the injectivity radius $I_M(p)
\in (0,\infty ]$ at any interior point $p\in M$ still makes sense,
although the exponential map $\exp _p$ is no longer
defined in the whole $T_pM$.
We endow $M$ with the structure of a metric space with respect to the intrinsic distance $d_M$,
and let $\overline{M}$ be the metric completion of $(M,d_M)$. In the sequel we will identify $M$
with its isometrically embedded image in $\overline{M}$. Given an interior point $p\in M$,
we define $d_M(p,\partial M)>0$ to be the distance in $\overline{M}$ from
$p$ to $\partial M=\overline{M}-\Int(M)$.
Consider the continuous function
$f\colon \mbox{Int}(M)\to (0,\infty )$ given by
\[
f(p)=\frac{\min \{ 1,d_M(p,\partial M)\} }{I_M(p)}.
\]
Suppose that $f$ is unbounded. Then, the conclusions in Theorem~\ref{tthm3introd} hold, i.e.,
exist points $p_n\in \Int (M)$ and positive numbers $\ve _n=nI_M(p_n)\to 0$ such that
items 1,\ldots , 6 of Theorem~\ref{tthm3introd} hold.

To prove this version of Theorem~\ref{tthm3introd} in the case that either $M$ is incomplete
or $\partial M\neq \mbox{\O }$, then
one must replace the points $q_n\in M$ with $I_M(q_n)\leq \frac{1}{n}$ that appeared in the
first paragraph of Section~\ref{sec4} by points $q_n\in \mbox{Int}(M)$
such that $f(q_n)\geq n$, and
then change the function $h_n$ defined in (\ref{eq:hn}) by the expression
\[
h_n(x)=\frac{d_M(x,\partial B_M(q_n,\frac{1}{2}d_M(q_n,\partial M))}{I_M(x)},\quad
x\in \overline{B}_M(q_n,\textstyle{ \frac{1}{2}}d_M(q_n,\partial M)).
\]
From this point on, the above proof of Theorem~\ref{tthm3introd} works without changes.
}
\end{remark}

\section{Applications.} \label{secap}

\begin{definition}
  {\rm
  Given a complete embedded minimal surface $M$ with injectivity radius zero
  in a homogeneously regular
  three-manifold, a {\it local picture of $M$ on the scale of
  topology} is one of the blow-up limits that can occur when we apply
  Theorem~\ref{tthm3introd} to $M$, namely a nonsimply connected
properly embedded minimal surface $M_{\infty }
  \subset \R^3$ as in item~4 of that theorem, a minimal parking garage
structure in $\R^3$ with at least two columns as in item~5 or a minimal
lamination $\mathcal{L}$ of $\R^3-\mathcal{S}$ as in item~6,  obtained as
  a limit of $M$ under blow-up around points of almost-minimal injectivity
  radius.

  Similarly, given a complete embedded minimal surface $M$ with unbounded second fundamental form
  in a homogeneously regular three-manifold,  a {\it local picture of $M$ on the scale of
  curvature} is a nonflat properly embedded minimal surface $M_{\infty }
  \subset \R^3$ of bounded Gaussian curvature, obtained as a limit of $M$ under blow-up around
  points of almost-maximal second fundamental form, in the sense of Theorem~1.1 in~\cite{mpr20}.
  }
\end{definition}

An immediate consequence of Theorem~\ref{tthm3introd} in the Introduction
and of the uniqueness of the helicoid~\cite{mr8}
is the following statement.
\begin{corollary}
  Let $M$ be a complete embedded minimal surface with injectivity radius zero in a
  homogeneously regular  three-manifold. If a properly embedded minimal surface
$M_{\infty }\subset \R^3$ is a local picture of $M$
on the scale of topology (i.e., $M_{\infty }$ arises as a blow-up limit of $M$ as in
item~4 of Theorem~\ref{tthm3introd}) and $M_{\infty }$
does not have bounded Gaussian curvature,
  then every local picture of $M_{\infty }$ on the scale of curvature is a helicoid.
\end{corollary}

\subsection{The set of local pictures on the scale of topology.}

Given $0<a\leq b$, consider the set $\mathcal{B}_{a,b}$ of all complete
embedded minimal surfaces $M\subset \R^3$ with $|K_M|\leq b$ and
$|K_M|(p)\geq a$ at some point $p\in \overline{\B }(1)$. Since complete embedded nonflat minimal
surfaces in $\rth$ of bounded absolute Gaussian curvature are proper~\cite{mr7}
 and properly embedded nonflat minimal surfaces in $\rth$ are connected~\cite{hm10},
 then the surfaces in $\mathcal{B}_{a,b}$ are connected and properly embedded.
 The topology of uniform $C^k$-convergence on compact subsets
of $\R^3$ is metrizable on the set $\mathcal{B}_{a,b}$ (see Section~5 of~\cite{mpr20}
for a proof of this fact in a slightly different context, for $a=b=1$).
Sequential compactness (hence compactness)
of $\mathcal{B}_{a,b}$ follows immediately from uniform local curvature and area estimates
(area estimates come from the existence of a tubular neighborhood of fixed radius,
see~\cite{mr7}). By the regular neighborhood
theorem in~\cite{mr7} or~\cite{sor1}, the surfaces in $\mathcal{B}_{a,b}$
all have cubical area growth, i.e.,
\[
R^{-3}\mbox{ Area}(M\cap \B (R))\leq C
\]
for all surfaces $M\in \mathcal{B}_{a,b}$ and for all $R>1$, where $C=C(b)>0$ depends
on the uniform bound of the curvature.

The next corollary follows directly from the above observations,
the Local Picture Theorem on the Scale of Curvature (Theorem~1.1 in~\cite{mpr20})
and the Local Picture Theorem on the Scale of Topology (Theorem~\ref{tthm3introd}
in this paper).

\begin{corollary} \label{cor5.1}
Suppose $M$ is a complete, embedded minimal surface with injectivity
radius zero in a homogeneously regular three-manifold, and
suppose $M$ does not have a local picture on the scale of curvature
which is a helicoid. Then, there exist positive constants $a\leq b$ depending only on $M$,
such that every local picture of $M$ on the scale of
topology lies in $\mathcal{B}_{a,b}$ (in particular, every such local picture of $M$ on the
scale of topology arises from item~4 in Theorem~\ref{tthm3introd}). Furthermore, the set
\[
\mathcal{B}(M)=\{ \mbox{local pictures of $M$ on the scale of topology} \}
\]
is a closed subset of $\mathcal{B}_{a,b}$ (thus $\mathcal{B}(M)$ is compact),
and there is a constant $C=C(M)$ such that every local picture on the scale of topology has
area growth at most $C R^3$.
\end{corollary}

\begin{remark}
{\rm
With the notation of Theorem~\ref{tthm3introd} in this paper,
if $M$ has finite genus or if the sequence $\{ \lambda_nM_n\} _n$ has
uniformly bounded genus in fixed size intrinsic metric balls,
then item~{6} of that theorem does not occur, since 
item~(B) of Proposition~\ref{proposnew}
does not occur.
This fact will play a crucial role in our forthcoming paper~\cite{mpr8}, when
proving a bound on the number of ends for a complete, embedded minimal surface
of finite topology in $\R^3$, that only depends on its genus. Also
in~\cite{mpr11}, we will apply Theorem~\ref{tthm3introd} to give a
general structure theorem for singular minimal laminations of $\rth$
with a countable number of singularities. }
\end{remark}

\subsection{Complete embedded minimal surfaces in $\R^3$ with zero flux.}

Recall that a nonflat minimal immersion $f\colon M \to \rth$ has zero
flux if the integral of the unit conormal vector around any closed curve on $M$
is zero. By the Weierstrass representation, a nonflat minimal immersion
$f\colon M\to \rth$ has nonzero
flux if and only if $f\colon M \to \rth$ is the unique isometric
minimal immersion of $M$ into $\rth$ up to rigid motions.

The results described in the next corollary to Theorem~\ref{tthm3introd}
overlap somewhat with the rigidity
results for complete embedded constant mean curvature surfaces by
Meeks and Tinaglia described in~\cite{mt2}.

\begin{corollary}
\label{cor5.4}
  Let $M\subset \R^3$ be a complete, embedded minimal surface with zero flux.
  Suppose that $M$ is not a plane or a helicoid.
  Then, $M$ has infinite genus and one of the following two possibilities hold:
\begin{enumerate}
\item $M$ is properly embedded in $\rth$ with positive injectivity radius and one end.
\item $M$ has injectivity radius zero and every local picture of $M$ on the scale of topology
is a properly embedded minimal surface with one end, infinite genus and zero flux.
\end{enumerate}
\end{corollary}
\begin{proof}
First suppose that $M$ is properly embedded in $\R^3$. As $M$ has zero flux, then the main result in
Choi, Meeks and White~\cite{cmw1} insures that $M$ has one end.
We claim that $M$ has infinite genus. Otherwise,
by classification of properly embedded minimal surfaces with finite genus and one end,
then $M$ is a helicoid with handles, in particular, $M$ is asymptotic to the helicoid
(Bernstein and Breiner~\cite{bb2}).  We next show that in this case, $M$ has nonzero flux,
which contradicts the hypothesis: as $M$ is a helicoid with
handles, then after rotation we can assume that $M$ is asymptotic to a vertical helicoid whose
axis is the $x_3$-axis. Consider the intersection $\G_t$ of $M$ with the horizontal plane
$\{ x_3=t\} $. For every $t\in \R$, $\G_t$ is a proper 1-dimensional analytic set with two ends
which are asymptotic to the ends of a straight line; this result can be deduced from
the analytic results described in either~\cite{bb2} or \cite{mpe3}.
For $|t|$ large, $\G_t$ consists of a connected, proper planar arc asymptotic
to a straight line. However, since $M$ is not simply connected (because $M$ is not a plane or helicoid,
Meeks and Rosenberg~\cite{mr8}), there exists a lowest plane $\{ x_3=T\} $ such that $\G_T$ is not a
proper arc. In particular, $\G _T$ is a limit of proper arcs $\G_t$, $t\nearrow T$. By the maximum
principle, $\G _T$ is a connected, 1-dimensional analytic set asymptotic to a straight line. Since
$\G_T$ is not an arc, then $x_3\colon M\to \R$ has a critical point of negative index on $\G _T$.
As $\G_T$ only has two ends, then standard topological arguments imply that $\G_T$ contains a
piecewise smooth loop that bounds a horizontal disk on one side of $M$. By the maximum principle,
the unit conormal vector of $M$ along this loop lies in the closed upper (or lower) hemisphere
and it is not everywhere horizontal. Thus, the flux of $M$ along this loop is not zero, which
is a contradiction. Therefore, $M$ has infinite genus provided that it is proper.

Finally, suppose that $M$ satisfies the hypotheses of Corollary~\ref{cor5.4}.
If the injectivity radius of $M$ is positive, then $M$ is proper by Theorem~2 in~\cite{mr13}
and so, Corollary~\ref{cor5.4} holds by the arguments in the last paragraph.
Otherwise, the injectivity radius of $M$ is zero and thus, Theorem~\ref{tthm3introd}
applies. Note that cases 5 and 6 of Theorem~\ref{tthm3introd} cannot occur since in
those cases the flux of the approximating surface $\lambda_nM_n$
(with the notation of Theorem~\ref{tthm3introd}) is not zero
by item~(B) of Proposition~\ref{ass4.17} and item~(C) of
 Proposition~\ref{proposnew},
but $M$ has zero flux. Hence every local picture of $M$ on the scale of topology
is a properly embedded minimal surface $M_{\infty}\subset \R^3$.
Note that $M_{\infty }$ has zero
flux since $M$ has zero flux. By arguments in the first paragraph of this proof,
$M_{\infty }$ has
infinite genus. Finally, observe this last property together with a lifting argument
shows that $M$ also has infinite genus in this case.
\end{proof}

\begin{remark}
{\rm
If $M\subset \R^3$ is a complete embedded minimal surface that admits an intrinsic isometry
$I\colon M \to M$ which does not extend to an ambient isometry of $\rth$,
then $M$ must have zero flux (because the only isometric minimal immersions from $M$ into $\R^3$
are associated minimal surfaces to $M$ by Calabi~\cite{ca1}). Since the associated surfaces to
a helicoid which are not congruent to it are not embedded, then such an $M$ cannot admit a local
picture on the scale of curvature which is a helicoid. Therefore, either $M$ has positive injectivity
radius (so it is properly embedded in $\R^3$ by Theorem~2 in~\cite{mr13}) and its
Gaussian curvature is bounded (otherwise one could blow-up $M$ on the scale of curvature to
produce a limit helicoid by Theorem~1.1 in~\cite{mpr20}, which is a contradiction), or
$M$ has injectivity radius zero and Corollaries~\ref{cor5.1} and~\ref{cor5.4} imply that  every local picture of
$M$ on the scale of topology is a nonsimply connected,
properly embedded minimal surface with bounded Gaussian curvature
and zero flux.
The authors believe that this observation could
play an important role in proving
the classical conjecture that intrinsic isometries of complete embedded
minimal surfaces in $\rth$ always extend to ambient isometries, and more generally, to prove that a
complete embedded minimal surface in $\rth$ does not admit another noncongruent isometric minimal embedding
into $\rth$.
}
\end{remark}

\vspace{.2cm}

\center{William H. Meeks, III at profmeeks@gmail.com\\
Mathematics Department, University of Massachusetts, Amherst, MA 01003}
\center{Joaqu\'\i n P\'{e}rez at jperez@ugr.es\\
Department of Geometry and Topology, Institute of Mathematics (IEMath-GR),\\ 
University of Granada, Granada, Spain}
\center{Antonio Ros at aros@ugr.es\\
Department of Geometry and Topology, Institute of Mathematics (IEMath-GR), \\
University of Granada, Granada, Spain}

\bibliographystyle{plain}
\bibliography{bill}

\begin{thebibliography}{10}

\bibitem{ber1}
M.~Berger.
\newblock {\em A panoramic view of {R}iemannian {G}eometry}.
\newblock Springer-Verlag, 2003.
\newblock MR2002701,Zbl 1038.53002.

\bibitem{bb2}
J.~Bernstein and C.~Breiner.
\newblock Conformal structure of minimal surfaces with finite topology.
\newblock {\em Comm. Math. Helv.}, 86(2):353--381, 2011.
\newblock MR2775132, Zbl 1213.53011.

\bibitem{ca1}
E.~Calabi.
\newblock Quelques applications de l'analyse complexe aux surfaces d'aire
  minima.
\newblock In {\em Topics in Complex Manifolds}, pages 59--81. Les Presses de
  l'{U}niversit\'{e} de Montr\'{e}al, 1967.
\newblock H. Rossi, editor.

\bibitem{ch2}
I.~Chavel.
\newblock {\em Riemannian Geometry: a modern introduction}.
\newblock Cambridge University Press, 1993.
\newblock MR1271141, Zbl 0810.53001.

\bibitem{cmw1}
T.~Choi, W.~H. Meeks~III, and B.~White.
\newblock A rigidity theorem for properly embedded minimal surfaces in $\rth$.
\newblock {\em J. Differential Geom.}, 32:65--76, 1990.
\newblock MR1064865, Zbl 0704.53008.

\bibitem{cmCourant}
T.~H. Colding and W.~P. Minicozzi~II.
\newblock {\em Minimal surfaces}, volume~4 of {\em Courant Lecture Notes in
  Mathematics}.
\newblock New York University Courant Institute of Mathematical Sciences, New
  York, 1999.
\newblock MR1683966, Zbl 0987.49025.

\bibitem{cm26}
T.~H. Colding and W.~P. Minicozzi~II.
\newblock Multivalued minimal graphs and properness of disks.
\newblock {\em International Mathematical Research Notices}, 21:1111--1127,
  2002.
\newblock MR1904463, Zbl 1008.58012.

\bibitem{cm21}
T.~H. Colding and W.~P. Minicozzi~II.
\newblock The space of embedded minimal surfaces of fixed genus in a
  $3$-manifold {I}; {E}stimates off the axis for disks.
\newblock {\em Ann. of Math.}, 160:27--68, 2004.
\newblock MR2119717, Zbl 1070.53031.

\bibitem{cm22}
T.~H. Colding and W.~P. Minicozzi~II.
\newblock The space of embedded minimal surfaces of fixed genus in a
  $3$-manifold {I}{I}; {M}ulti-valued graphs in disks.
\newblock {\em Ann. of Math.}, 160:69--92, 2004.
\newblock MR2119718, Zbl 1070.53032.

\bibitem{cm24}
T.~H. Colding and W.~P. Minicozzi~II.
\newblock The space of embedded minimal surfaces of fixed genus in a
  $3$-manifold {I}{I}{I}; {P}lanar domains.
\newblock {\em Ann. of Math.}, 160:523--572, 2004.
\newblock MR2123932, Zbl 1076.53068.

\bibitem{cm23}
T.~H. Colding and W.~P. Minicozzi~II.
\newblock The space of embedded minimal surfaces of fixed genus in a
  $3$-manifold {I}{V}; {L}ocally simply-connected.
\newblock {\em Ann. of Math.}, 160:573--615, 2004.
\newblock MR2123933, Zbl 1076.53069.

\bibitem{cm35}
T.~H. Colding and W.~P. Minicozzi~II.
\newblock The {C}alabi-{Y}au conjectures for embedded surfaces.
\newblock {\em Ann. of Math.}, 167:211--243, 2008.
\newblock MR2373154, Zbl 1142.53012.

\bibitem{cm25}
T.~H. Colding and W.~P. Minicozzi~II.
\newblock The space of embedded minimal surfaces of fixed genus in a
  $3$-manifold {V}; {F}ixed genus.
\newblock {\em Ann. of Math.}, 181(1):1--153, 2015.
\newblock MR3272923, Zbl 06383661.

\bibitem{ckmr1}
P.~Collin, R.~Kusner, W.~H. Meeks~III, and H.~Rosenberg.
\newblock The geometry, conformal structure and topology of minimal surfaces
  with infinite topology.
\newblock {\em J. Differential Geom.}, 67:377--393, 2004.
\newblock MR2153082, Zbl 1098.53006.

\bibitem{cp1}
M.~{do Carmo} and C.~K. Peng.
\newblock Stable complete minimal surfaces in $\rth$ are planes.
\newblock {\em Bulletin of the AMS}, 1:903--906, 1979.
\newblock MR0546314, Zbl 442.53013.

\bibitem{ehr1}
P.~E. Ehrlich.
\newblock Continuity properties of the injectivity radius function.
\newblock {\em Compositio Math.}, 29:151--178, 1974.
\newblock MR0417977, Zbl 0289.53034.

\bibitem{ep2}
C.L. Epstein.
\newblock Positive harmonic functions on abelian covers.
\newblock {\em J. Funct. Anal.}, 82(2):303--315, 1989.
\newblock MR0987296, Zbl 0685.31003.

\bibitem{fs1}
D.~Fischer-{C}olbrie and R.~Schoen.
\newblock The structure of complete stable minimal surfaces in $3$-manifolds of
  nonnegative scalar curvature.
\newblock {\em Comm. on Pure and Appl. Math.}, 33:199--211, 1980.
\newblock MR0562550, Zbl 439.53060.

\bibitem{hm10}
D.~Hoffman and W.~H. Meeks~III.
\newblock The strong halfspace theorem for minimal surfaces.
\newblock {\em Invent. Math.}, 101:373--377, 1990.
\newblock MR1062966, Zbl 722.53054.

\bibitem{kri1}
H.~Kriete.
\newblock Herman's proof of the existence of critical points on the boundary of
  singular domains.
\newblock In {\em Progress in Holomorphic Dynamics}, pages 31--40. CRC Press,
  1998.
\newblock MR1643013, Zbl 0938.30018.

\bibitem{me25}
W.~H. Meeks~III.
\newblock The regularity of the singular set in the {C}olding and {M}inicozzi
  lamination theorem.
\newblock {\em Duke Math. J.}, 123(2):329--334, 2004.
\newblock MR2066941, Zbl 1086.53005.

\bibitem{me30}
W.~H. Meeks~III.
\newblock The limit lamination metric for the {C}olding-{M}inicozzi minimal
  lamination.
\newblock {\em Illinois J. of Math.}, 49(2):645--658, 2005.
\newblock MR2164355, Zbl 1087.53058.

\bibitem{mpe3}
W.~H. Meeks~III and J.~P\'{e}rez.
\newblock Embedded minimal surfaces of finite topology.
\newblock Preprint at http://arxiv.org/pdf/1506.07793v1.pdf.

\bibitem{mpe17}
W.~H. Meeks~III and J.~P\'{e}rez.
\newblock Finite topology minimal surfaces in homogeneous three-manifolds.
\newblock Preprint at http://arxiv.org/abs/1505.06764.

\bibitem{mpr8}
W.~H. Meeks~III, J.~P\'{e}rez, and A.~Ros.
\newblock Bounds on the topology and index of classical minimal surfaces.
\newblock Preprint available at https://arxiv.org/abs/1605.02501.

\bibitem{mpr9}
W.~H. Meeks~III, J.~P\'{e}rez, and A.~Ros.
\newblock The embedded {C}alabi-{Y}au conjectures for finite genus.
\newblock Work in progress.

\bibitem{mpr11}
W.~H. Meeks~III, J.~P\'{e}rez, and A.~Ros.
\newblock Structure theorems for singular minimal laminations.
\newblock Preprint at http://arxiv.org/pdf/1602.03197v1.pdf.

\bibitem{mpr1}
W.~H. Meeks~III, J.~P\'{e}rez, and A.~Ros.
\newblock Uniqueness of the {R}iemann minimal examples.
\newblock {\em Invent. Math.}, 133:107--132, 1998.
\newblock MR1626477, Zbl 916.53004.

\bibitem{mpr3}
W.~H. Meeks~III, J.~P\'{e}rez, and A.~Ros.
\newblock The geometry of minimal surfaces of finite genus {I}; curvature
  estimates and quasiperiodicity.
\newblock {\em J. Differential Geom.}, 66:1--45, 2004.
\newblock MR2128712, Zbl 1068.53012.

\bibitem{mpr13}
W.~H. Meeks~III, J.~P\'{e}rez, and A.~Ros.
\newblock Liouville-type properties for embedded minimal surfaces.
\newblock {\em Communications in Analysis and Geometry}, 14(4):703--723, 2006.
\newblock MR2273291, Zbl 1117.53009.

\bibitem{mpr19}
W.~H. Meeks~III, J.~P\'{e}rez, and A.~Ros.
\newblock Stable constant mean curvature surfaces.
\newblock In {\em Handbook of Geometrical Analysis}, volume~1, pages 301--380.
  International Press, edited by Lizhen Ji, Peter Li, Richard Schoen and Leon
  Simon, ISBN: 978-1-57146-130-8, 2008.
\newblock MR2483369, Zbl 1154.53009.

\bibitem{mpr18}
W.~H. Meeks~III, J.~P\'{e}rez, and A.~Ros.
\newblock Limit leaves of an {H} lamination are stable.
\newblock {\em J. Differential Geom.}, 84(1):179--189, 2010.
\newblock MR2629513, Zbl 1197.53037.

\bibitem{mpr20}
W.~H. Meeks~III, J.~P\'{e}rez, and A.~Ros.
\newblock The {D}ynamics {T}heorem for properly embedded minimal surfaces.
\newblock {\em Mathematische Annalen}, 365(3):1069--1089, 2016.
\newblock MR3521082, Zbl 06618524.

\bibitem{mpr10}
W.~H. Meeks~III, J.~P\'{e}rez, and A.~Ros.
\newblock Local removable singularity theorems for minimal laminations.
\newblock {\em J. Differential Geometry}, 103(2):319--362, 2016.
\newblock MR3504952, Zbl 06603546.

\bibitem{mr8}
W.~H. Meeks~III and H.~Rosenberg.
\newblock The uniqueness of the helicoid.
\newblock {\em Ann. of Math.}, 161:723--754, 2005.
\newblock MR2153399, Zbl 1102.53005.

\bibitem{mr13}
W.~H. Meeks~III and H.~Rosenberg.
\newblock The minimal lamination closure theorem.
\newblock {\em Duke Math. Journal}, 133(3):467--497, 2006.
\newblock MR2228460, Zbl 1098.53007.

\bibitem{mr7}
W.~H. Meeks~III and H.~Rosenberg.
\newblock Maximum principles at infinity.
\newblock {\em J. Differential Geom.}, 79(1):141--165, 2008.
\newblock MR2401421, Zbl 1158.53006.

\bibitem{mt13}
W.~H. Meeks~III and G.~Tinaglia.
\newblock Limit lamination theorem for ${H}$-disks.
\newblock Preprint available at http://arxiv.org/pdf/1510.05155.pdf.

\bibitem{mt14}
W.~H. Meeks~III and G.~Tinaglia.
\newblock Limit lamination theorem for ${H}$-surfaces.
\newblock To appear in J. Reine Angew. Math. Preprint available at
  http://arxiv.org/pdf/1510.07549.pdf.

\bibitem{mt9}
W.~H. Meeks~III and G.~Tinaglia.
\newblock One-sided curvature estimates for ${H}$-disks.
\newblock Preprint available at http://arxiv.org/pdf/1408.5233.pdf.

\bibitem{mt2}
W.~H. Meeks~III and G.~Tinaglia.
\newblock The rigidity of embedded constant mean curvature surfaces.
\newblock {\em J. Reine Angew. Math.}, 660, 2011.
\newblock MR2855824, Zbl 1235.53008.

\bibitem{pro2}
J.~P\'{e}rez and A.~Ros.
\newblock Properly embedded minimal surfaces with finite total curvature.
\newblock In {\em The Global Theory of Minimal Surfaces in Flat Spaces}, pages
  15--66. {L}ecture {N}otes in {M}ath 1775, Springer-Verlag, 2002.
\newblock G. P. Pirola, editor. MR1901613, Zbl 1028.53005.

\bibitem{po1}
A.~V. Pogorelov.
\newblock On the stability of minimal surfaces.
\newblock {\em Soviet Math. Dokl.}, 24:274--276, 1981.
\newblock MR0630142, Zbl 0495.53005.

\bibitem{sa2}
T.~Sakai.
\newblock On continuity of injectivity radius function.
\newblock {\em Math. J. Okayama Univ.}, 25(1):91--97, 1983.
\newblock MR701970, Zbl 0525.53053.

\bibitem{sc3}
R.~Schoen.
\newblock {\em Estimates for Stable Minimal Surfaces in Three Dimensional
  Manifolds}, volume 103 of {\em Ann. of Math. Studies}.
\newblock Princeton University Press, 1983.
\newblock MR0795231, Zbl 532.53042.

\bibitem{sor1}
M.~Soret.
\newblock Maximum principle at infinity for complete minimal surfaces in flat
  $3$-manifolds.
\newblock {\em Annals of Global Analysis and Geometry}, 13:101--116, 1995.
\newblock MR1336206, Zbl 0873.53039.

\bibitem{tw1}
M.~Traizet and M.~Weber.
\newblock Hermite polynomials and helicoidal minimal surfaces.
\newblock {\em Invent. Math.}, 161(1):113--149, 2005.
\newblock MR2178659, Zbl 1075.53010.

\bibitem{wh13}
B.~White.
\newblock Which ambient spaces admit isoperimetric inequalities for
  submanifolds?
\newblock {\em J. Differential Geom.}, 83:213--228, 2009.
\newblock MR2545035, Zbl 1179.53061.

\bibitem{zakzei}
S.~Zakeri and M.~Zeinalian.
\newblock When ellipses look like circles: the measurable {R}iemann mapping
  theorem.
\newblock {\em Nashr-e-Riazi}, 8:5--14, 1996.

\end{thebibliography}
\end{document}